\documentclass[a4paper, oneside]{amsart}

\usepackage{enumitem}
\usepackage{dsfont}
\usepackage{amssymb,amsthm,amsmath}
\usepackage{mathrsfs,mathtools}

\usepackage[usenames,dvipsnames]{xcolor}
\usepackage[colorlinks=true,linkcolor=Red,citecolor=Green]{hyperref}

\usepackage{tikz}

\usepackage[depth=2]{bookmark}

\usepackage{microtype}

\theoremstyle{plain}
\newtheorem{prop}{Proposition}[section]
\newtheorem{lemma}[prop]{Lemma}
\newtheorem{thm}[prop]{Theorem}
\newtheorem{theorem}{Theorem}

\newtheorem{cor}[prop]{Corollary}

\theoremstyle{definition}
\newtheorem{remdef}[prop]{Remark and Definition}
\newtheorem{example}[prop]{Example}
\newtheorem{remark}[prop]{Remark}

\theoremstyle{remark}

\makeatletter
\@namedef{subjclassname@2020}{%
  \textup{2020} Mathematics Subject Classification}
\makeatother

\DeclareMathOperator{\Iso}{Iso}
\DeclareMathOperator{\sing}{Sing}
\DeclareMathOperator{\cusps}{Cusps}
\DeclareMathOperator{\singcusp}{sc}
\newcommand{\thin}{\hspace{.5pt}}
\newcommand{\ssec}{\textnormal{sec}}
\newcommand{\ff}{\textit{ff}}
\newcommand{\fc}{\textit{fc}}
\newcommand{\cf}{\textit{cf}}
\newcommand{\cc}{\textit{cc}}
\newcommand{\twist}{\chi} 
\newcommand{\ZZ}{\mathbb{Z}}
\newcommand{\CC}{\mathbb{C}}
\DeclareMathOperator{\FP}{FP}
\DeclareMathOperator{\Isom}{Isom}
\DeclareMathOperator{\SL}{SL}
\DeclareMathOperator{\PSL}{PSL}
\DeclareMathOperator{\Unit}{U}
\DeclareMathOperator{\diag}{diag}
\DeclareMathOperator{\Tr}{Tr}
\DeclareMathOperator{\tr}{tr}
\DeclareMathOperator{\Ima}{Im}
\DeclareMathOperator{\Rea}{Re}
\DeclareMathOperator{\supp}{supp}
\DeclareMathOperator{\pr}{pr}
\DeclareMathOperator{\res}{res}
\DeclareMathOperator{\Ind}{Ind}
\DeclareMathOperator{\Stab}{Stab}

\newcommand{\parab}{\textnormal{par}}
\newcommand{\parprim}{\textnormal{pp}}
\newcommand{\hyp}{\textnormal{hyp}}
\newcommand{\hypprim}{\textnormal{hp}}
\newcommand{\ellip}{\textnormal{ell}}
\newcommand{\ellprim}{\textnormal{ep}}
\newcommand{\interior}{\textnormal{int}}

\newcommand\N{\mathbb{N}}
\newcommand\Q{\mathbb{Q}}
\newcommand\R{\mathbb{R}}
\newcommand\Z{\mathbb{Z}}
\newcommand\C{\mathbb{C}}
\newcommand{\h}{\mathbb{H}}

\newcommand{\mc}[1]{\mathcal #1}
\newcommand{\wt}{\widetilde}
\newcommand{\wh}{\widehat}

\newcommand{\eps}{\varepsilon}
\DeclareMathOperator{\id}{id}

\newcommand{\mat}[4]{\begin{pmatrix} #1&#2\\#3&#4\end{pmatrix}}
\newcommand{\bmat}[4]{\begin{bmatrix} #1&#2\\#3&#4\end{bmatrix}}
\newcommand{\textmat}[4]{\left(\begin{smallmatrix} #1&#2 \\ #3&#4
\end{smallmatrix}\right)}
\newcommand{\textbmat}[4]{\left[\begin{smallmatrix} #1&#2 \\ #3&#4
\end{smallmatrix}\right]}

\newcommand{\setmid}{\;:\;}

\newcommand{\Sph}{\mathbb{S}}
\newcommand{\FF}{\mathbf{F}}

\DeclarePairedDelimiter\abs{\lvert}{\rvert}
\DeclarePairedDelimiter\norm{\lVert}{\rVert}
\DeclarePairedDelimiter\ang{\langle}{\rangle}
\DeclarePairedDelimiter\floor{\lfloor}{\rfloor}

\newcommand{\bundle}{E_\twist}
\newcommand{\topo}{\textnormal{top}}
\newcommand{\eulertop}{\chi_e^\topo}
\newcommand{\orb}{\textnormal{orb}}
\newcommand{\eulerorb}{\chi_e^\orb}
\DeclareMathOperator{\EV}{EV}
\DeclareMathOperator{\EVP}{EVP}

\DeclareMathOperator{\HS}{HS}
\newcommand{\loc}{\textnormal{loc}}
\newcommand{\cpt}{\textnormal{cpt}}
\newcommand{\hol}{\textnormal{hol}}

\DeclareMathOperator{\zTr}{0-Tr}
\DeclareMathOperator{\zVol}{0-Vol}
\newcommand{\zInt}[1]{ \sideset{^0}{}{\!\!\!\int_{#1}} }
\newcommand{\pa}{\partial}

\newcommand{\ResTwist}{R_{X,\twist}}
\newcommand{\Prd}{\mathcal{P}}
\newcommand{\ProdTwist}{\mathcal{P}_{X,\twist}}
\newcommand{\LapTwist}{\Delta_{X,\twist}}
\newcommand{\ResSet}{\mathcal{R}}
\newcommand{\rel}{\textnormal{rel}}

\newcommand{\group}{\mathit{\Gamma}}
\newcommand{\gammafunc}{\mathrm{\Gamma}}

\makeatletter
\renewcommand\subsubsection{\@secnumfont}{\bfseries}%
\renewcommand\subsubsection{\@startsection{subsubsection}{3}
  \z@{.5\linespacing\@plus.7\linespacing}{-.5em}%
  {\normalfont\bfseries}}
\makeatother

\usepackage[style=rwtmath, dashed=true]{biblatex}
\begin{document}

\title{The divisor of the twisted Selberg zeta function}
\author[M.~Doll]{Moritz Doll}
\address{Moritz Doll, The University of Melbourne, Department of Electrical and Electronic Engineering, Parkville, VIC 3010, Australia}
\email{moritz.doll@unimelb.edu.au}

\author[A.~Pohl]{Anke Pohl}
\address{Anke Pohl, University of Bremen, Department 3: Mathematics and Computer Science, Institute for Dynamical Systems, Bibliothekstr.~5,  28359 Bremen, Germany}
\email{apohl@uni-bremen.de}


\begin{abstract} 
For the Selberg zeta function of geometrically finite infinite-area hyperbolic orbisurfaces with twists by finite-dimensional unitary representations, we establish a factorization formula in terms of a Weierstrass product of the Laplace resonances of the considered hyperbolic orbisurface, Barnes G-functions, gamma functions, and the singularity degrees of the representation. We thereby provide an interpretation of the zeros and poles of the Selberg zeta function by spectral and geometric entities of the orbisurface and the representation. This formula generalizes the factorization result by Borthwick, Judge and Perry to hyperbolic orbisurfaces with orbifold singularities as well as to unitary twists. Also in the untwisted case, the presence of orbifold singularities yields a separate, previously unobserved contribution to the factorization formula.

\end{abstract}

\maketitle

\tableofcontents

\section{Introduction}

Let $X$ be a hyperbolic orbisurface of finite geometry and infinite area. As it is well-known, $X$ can be represented as an orbit space~$\group\backslash\h$, where $\h$ is the hyperbolic plane (realized, e.g., by its upper half-plane model) and $\group$ is the fundamental group of~$X$ (realized, e.g., as a Fuchsian group, i.e., as a discrete subgroup of~$\PSL_2(\R)$, acting by linear fractional transformations on~$\h$). We emphasize that $\group$ may contain elliptic elements and hence $X$
is allowed to be a genuine orbifold rather than a hyperbolic surface in the classical sense (meaning, being a Riemannian manifold).
Let 
\[
\twist\colon \group \to \Unit(V)
\]
be a unitary representation of~$\group$ on a finite-dimensional Hermitian vector space~$V$. We will often refer to~$\twist$ as a \emph{twist}. The Selberg zeta function~$Z_{X,\twist}$ associated to~$X$ and~$\chi$ is determined by the infinite product 
\begin{equation}\label{eq:szfinfprod}
Z_{X,\twist}(s) \coloneqq  \prod_{[g]\in [\group]_{\hypprim}}\prod_{k=0}^\infty \det\left( 1 - \chi(g) N(g)^{-(s+k)} \right)\,,
\end{equation}
valid and convergent for $s\in\C$ with $\Rea s$ sufficiently large (see, e.g., \cite{FP_szf}). Here, $[\group]_{\hypprim}$ denotes the set of $\group$-conjugacy classes of primitive hyperbolic elements in~$\group$, and $N(g)$ denotes the norm of~$g$ for any hyperbolic element~$g\in\group$. We refer to Section~\ref{sec:prelims} for precise definitions. In terms of periodic geodesics, $N(g) = \log \ell_g$ with $\ell_g$ being the (appropriate multiple of the minimal) length of the periodic geodesic on~$X$ associated to~$g$. As part of our main result, we will provide a proof that $Z_{X,\twist}$ extends meromorphically to all of~$\C$. We comment below on the relation to the already known results in this direction.

The main achievement of this article is to determine the divisor of the Selberg zeta function~$Z_{X,\twist}$ in terms of spectral properties of the Laplacian~$\Delta_{X,\twist}$ and of topological properties of~$X$ in combination with the twist~$\twist$. In order to state our main result, we denote by~$\ProdTwist$ a certain Weierstrass product  (the Hadamard product of order~$2$) of the resonances of~$\Delta_{X,\twist}$. We let $n_p=n_p(\twist)$ be the sum of the singularity degrees of~$\twist$ at the cusps of~$X$, and we let $n_d=n_d(\twist)$ be the singularity degree or openness degree of~$\twist$ at disk ends of~$X$. That is, $n_d = 0$ unless $X$ is a parabolic cylinder, in which case $n_d = n_p$. Even though $n_p$ and $n_d$ depend on~$\twist$ (and~$X$), we usually suppress this dependence in the notation for improved readability.
Further, we denote the topological Euler characteristic of~$X$ by~$\eulertop(X)$ and the orbifold Euler characteristic by~$\eulerorb(X)$. For details we refer to Section~\ref{sec:prelims}.

\begin{theorem}\label{thm:selberg-factorization}
Let $X = \group \backslash \h$ be a geometrically finite hyperbolic orbisurface of infinite area and let $\twist\colon\group\to \Unit(V)$ be a finite-dimensional unitary representation of~$\group$ on a Hermitian vector space~$V$.
Then the Selberg zeta function~$Z_{X,\twist}$, as initially given by the infinite product in~\eqref{eq:szfinfprod}, is meromorphic on~$\C$. It admits the factorization
\begin{align}\label{eq:SZF_fact_top}
    Z_{X,\twist}(s) & = e^{q(s)} G_{X^\wedge,\twist}(s)\thin G_\infty(s)^{-\dim(V)\,\eulertop(X)}\cdot
    \frac{\gammafunc\bigl(s-\frac12\bigr)^{n_p}}{\gammafunc\bigl(s+\frac12\bigr)^{n_d}}
    \cdot \ProdTwist(s)
    \\[2mm]\label{eq:SZF_fact_orb}
    & = e^{q(s)} \gammafunc_{X^\wedge,\twist}(s)\thin G_\infty(s)^{-\dim(V)\,\eulerorb(X)}\cdot
    \frac{\gammafunc\bigl(s-\frac12\bigr)^{n_p}}{\gammafunc\bigl(s+\frac12\bigr)^{n_d}}
    \cdot \ProdTwist(s)\,,
\end{align}
valid on all of~$\C$. Here, $q$ is a polynomial of degree at most~$2$, depending on~$X$ and~$\twist$.
The functions~$G_\infty$ and~$G_{X^\wedge,\twist}$ are entire with zeros at the negative integers; they are defined in~\eqref{eq:def-barnes} and \eqref{eq:GXwedge}, respectively.
The function~$\gammafunc_{X^\wedge,\twist}$ is a product of possibly non-integer powers of gamma functions; it is defined in~\eqref{eq:GammaXwedge}.
\end{theorem}

Some comments on this theorem, its relation to the literature and some special cases are in order.

\begin{itemize}
\item For hyperbolic orbisurfaces of \emph{finite area}, a factorization of the Selberg zeta function analogous to the one in Theorem~\ref{thm:selberg-factorization} is well-known and can be deduced from~\cite{Venkov_book}.
Indeed, the divisor of the Selberg zeta function is given by~\cite[Theorem~5.1.2]{Venkov_book}. The bound in~\cite[Lemma~5.2.3]{Venkov_book} implies that
the Selberg zeta function can be written as a product of explicitly known meromorphic functions and an entire function of order~$2$ without zeros.
For more results on the Selberg zeta function for finite-area hyperbolic orbisurfaces, we refer to~\cite{Phillips_scattering_twist, Phillips_perturb_twist, Guillope, Selberg}.
\item For hyperbolic orbisurfaces of \emph{infinite area}, the factorization of the Selberg zeta function was previously known only for \emph{non-elementary} \emph{untwisted} hyperbolic orbisurfaces \emph{without} orbifold points (i.e., no elliptic elements in the fundamental group), proved by Borthwick, Judge and Perry~\cite{BJP}. In the next comment, we show that our factorization formula reproduces the formula from~\cite{BJP} when reduced to non-elementary untwisted hyperbolic surfaces (without orbifold points).
\item If $\twist$ is trivial (in which case we omit it in the notation), then we obtain
\begin{align}
    Z_X(s) = e^{q(s)} G_{X^\wedge}(s)\thin G_\infty(s)^{-\eulertop(X)} \cdot
    \frac{\gammafunc\left(s-\frac12\right)^{n_p}}{\gammafunc\left(s+\frac12\right)^{n_d}}
    \cdot \mathcal{P}_X(s)\,,
\end{align}
where $n_p$ is the number of cusps of~$X$, and $n_d = n_p$ if $X$ is the parabolic cylinder and $n_d = 0$ otherwise.
In particular, if $\twist$ is trivial, $X$ is non-elementary, and $\group$ does not contain elliptic elements, then we recover the following result by Borthwick, Judge and Perry:

\begin{thm}[{\cite[Theorem~4.1]{BJP}}]\label{thm:BJP}
Suppose $X$ is a non-ele\-ment\-ary geometrically finite hyperbolic orbisurface of
infinite area. Then the Selberg zeta function admits a factorization
\[
Z_X(s) = e^{q(s)} G_\infty(s)^{-\eulertop(X)} \thin
\gammafunc\left(s-\tfrac12\right)^{n_p}\thin \mc P_X(s)\,,
\]
where $q$ is a polynomial of degree at most~$2$.
\end{thm}
\item The meromorphic continuation of the Selberg zeta function in~\eqref{eq:szfinfprod} has already been established by Fedosova and Pohl~\cite{FP_szf} for a certain class of geometrically finite hyperbolic orbisurfaces and the class of finite-dimensional representations with non-expanding cusp monodromy (which includes all finite-dimensional unitary representations).
Regarding the realm of hyperbolic orbisurfaces, the meromorphy result in~\cite{FP_szf} is conditional on the existence of a so-called strict transfer operator approach to the Selberg zeta function for the considered hyperbolic orbisurface. However, such approaches are known to exist for a growing class of hyperbolic orbisurfaces, see~\cite{Pohl_Wabnitz, Wabnitz, Pohl_Symdyn2d, Pohl_hecke_infinite, Moeller_Pohl, Mayer_thermoPSL, Mayer_thermo, Chang_Mayer_eigen}. 
In this article, we establish the meromorphic continuation of the Selberg zeta function by an alternative, more classical method, albeit restricted to twists by unitary representations. However, we emphasize again that our main result is the factorization formula.
\item The factor $\gammafunc_{X^\wedge,\twist}(s)$ in~\eqref{eq:SZF_fact_orb} (defined in~\eqref{eq:GammaXwedge}) is the contribution from the orbifold singularities (i.e., the elliptic elements in~$\group$). It compensates (exactly) the non-integer residues of the logarithmic derivative of $G_\infty(s)^{-\dim(V)\,\eulerorb(X)}$, and hence the total product in~\eqref{eq:SZF_fact_orb} is meromorphic (see Lemma~\ref{lem:barnes-cone} and Proposition~\ref{prop:euler-characteristic} for details).
\end{itemize}

\subsection*{Structure of the article}
We start by surveying, in Section~\ref{sec:prelims}, some basic facts of hyperbolic geometry and special functions that are needed for the formulation of Theorem~\ref{thm:selberg-factorization}. In Section~\ref{sec:cyclic} we establish Theorem~\ref{thm:selberg-factorization} for the special case that $\Gamma$ is a cyclic group (``model calculations''). There are four types of such groups, and in each one the Selberg zeta function and its zeros, the resonances and hence also the Weierstrass product of the resonances can be determined and evaluated explicitly. For that reason, these special cases are much easier to investigate. They will serve as building blocks for the general situation.
In Section~\ref{sec:examples} we showcase how to use Theorem~\ref{thm:selberg-factorization} and the Venkov--Zograf formula to obtain relations between hyperbolic orbisurfaces generated by two fundamental groups $\Lambda, \group$ with $\Lambda$ being a subgroup of~$\group$ of finite index.

Starting with Section~\ref{sec:spectral_manuscr} we work towards the proof of Theorem~\ref{thm:selberg-factorization}. In \cite{DFP, DFP2} we developed already a substantial amount of the necessary scattering theory for twisted hyperbolic orbisurfaces; we survey these results in Section~\ref{sec:prelims}, \ref{sec:cyclic}. In Section~\ref{sec:spectral_manuscr} we provide some required additional results.
In Section~\ref{sec:trace} we introduce two kinds of regularization of the trace of the resolvent of the Laplacian~$\Delta_{X,\twist}$, which we then use to connect the Selberg zeta function~$Z_{X,\twist}$ to the relative scattering determinant.
We take advantage of this connection in Section~\ref{sec:proof} to prove a functional equation for~$Z_{X,\twist}$, which then yields that this zeta functions admits a meromorphic continuation to all of~$\C$ and which allows us to completes the proof of Theorem~\ref{thm:selberg-factorization}.

While the structure of the proof mainly follows the one by Borthwick, Judge and Perry~\cite{BJP}, we note that
the inclusion of elliptic elements in $\group$ and the twist~$\twist$ makes the analysis more involved.
The twist has a significant influence on the contribution by parabolic elements~$p$ of~$\Gamma$ as the resolvent on an eigenspace associated to an eigenvalue
$\lambda \not = 1$ of~$\twist(p)$ is drastically different than the resolvent for the case that $\lambda = 1$.
Therefore, we have to include more model calculations (see Section~\ref{sec:cyclic}) and also the decomposition of the resolvent
in terms of the various group elements becomes more involved (see Section~\ref{sec:trace}).

\subsection*{Acknowledgements}
AP's research is funded by the Deutsche Forschungsgemeinschaft (DFG, German Research Foundation) -- project no.~441868048 (Priority Program~2026 ``Geometry at Infinity''). MD was partially funded by a Universit\"at Bremen ZF 04-A grant.

\section{Fundamental objects}\label{sec:prelims}

In this section we will present the objects used in the statement of Theorem~\ref{thm:selberg-factorization}.

\subsection{Hyperbolic plane and its isometries}
We denote the hyperbolic plane by~$\h$ and its group of orientation-preserving Riemannian isometries by~$\Isom^+(\h)$. We may and shall identify $\h$ with its upper half-plane model
\[
\h \coloneqq \{ z\in\C : \Ima z > 0 \}\,,\quad ds_z^2 = \frac{dz\,d\overline{z}}{(\Ima z)^2}\,,
\]
and $\Isom^+(\h)$ with the projective special linear group
\[
\PSL_2(\R)= \SL_2(\R)/\{\pm\id\}\,,
\]
acting by fractional linear transformations on~$\h$. More precisely, we denote an element in~$\PSL_2(\R)$ by any of its representing matrices in~$\SL_2(\R)$ but with square brackets, i.e.,
\begin{equation}\label{eq:genelem_g}
 g = \bmat{a}{b}{c}{d}  = \bmat{-a}{-b}{-c}{-d}
\end{equation}
is the element in~$\PSL_2(\R)$ represented by either of the matrices $\textmat{a}{b}{c}{d}$ and $\textmat{-a}{-b}{-c}{-d}$ in~$\SL_2(\R)$ (and these are indeed all representatives of~$g$ in~$\SL_2(\R)$). The action of~$\PSL_2(\R)$ on~$\h$ is then
\begin{equation}\label{eq:action}
 g.z = \frac{az+b}{cz+d}
\end{equation}
for $z\in\h$ and $g\in\PSL_2(\R)$ given as in~\eqref{eq:genelem_g}.

In the upper half plane model, the geodesic boundary~$\partial\h$ of~$\h$ is identified with $P^1(\R) = \R\cup\{\infty\}$. The action of~$\PSL_2(\R)$ extends continuously to~$\partial\h$; the action formula remains to be~\eqref{eq:action} subject to the convention $1/0=\infty$.

\subsection{Hyperbolic orbisurfaces and Fuchsian groups}\label{sec:fuchsian}
Given a Fuchsian group~$\group$, i.e., a discrete subgroup of~$\Isom^+(\h) = \PSL_2(\R)$, the quotient space $\group\backslash\h$ naturally carries the structure of a connected, two-dimensional real (or one-dimensional complex) orbifold with hyperbolic Riemannian metric. We refer to any such space as a \emph{hyperbolic orbisurface}. In the case that $\group$ is torsion-free, the hyperbolic orbisurface~$\group\backslash\h$ is a manifold (in the classical sense). However, if $\group$ has torsion, then $\group\backslash\h$ is a genuine orbifold. In the language of singular analysis, $\group\backslash\h$ is a manifold with conical singularities of rational angles, i.e., with angles in $\pi\Q$. We emphasize that in our investigations, we allow $\group$ to have torsion.

As it is well-known, several properties and entities of hyperbolic orbisurfaces can be characterized by entities of its fundamental Fuchsian group. For the convenience of the reader, we collect here those characterizations that we will take advantage of.

\medskip 

\begin{center}
\framebox{
\begin{minipage}{.9\textwidth}
From now on let $\group$ be a finitely generated Fuchsian group. Let $X \coloneqq \group\backslash\h$ denote the associated hyperbolic orbisurface. 
\end{minipage}
}
\end{center}

\medskip

We note that requesting $\group$ to be finitely generated is equivalent to requesting $X$ to be geometrically finite. Throughout we let
\begin{equation}\label{eq:def_pi}
 \pi\colon \h\to X=\group\backslash\h\,,\quad z\mapsto \group.z\,,
\end{equation}
denote the canonical quotient map (not to be confused with the number~$\pi$).

\subsubsection{Classification of elements in~$\group$}
The elements of~$\group$ can be classified according to their action properties on~$\h\cup\partial\h$. We call $g\in\group$, $g\not=\id$, \emph{hyperbolic} if its action fixes two points in~$\partial\h$ (equivalently, $\abs{\tr g}>2$). We call $g$ \emph{parabolic} if its action fixes a single point in~$\partial\h$ (equivalently, $\abs{\tr g}=2$), and we call $g$ \emph{elliptic} if its action fixes a single point in~$\h$ (equivalently, $\abs{\tr g}<2$). Further, we say that $g$ is \emph{primitive} if non-trivial roots of~$g$ are not contained in~$\group$, i.e., if $h^n=g$ for $h\in\group$ and $n\in\N$ implies $n=1$. Each of these four properties is stable under conjugation within~$\group$.
We denote the set of conjugacy classes of hyperbolic and primitive hyperbolic elements in~$\group$ by $[\group]_\hyp$ and $[\group]_\hypprim$, respectively. Similarly, we define $[\group]_\parab$, $[\group]_\parprim$ for parabolic elements, and $[\group]_\ellip$, $[\group]_\ellprim$ for elliptic elements.

\subsubsection{Relation between hyperbolic elements and periodic geodesics}\label{sec:hypgeod}
The elements of~$[\group]_\hyp$ are in bijection with the periodic geodesics on~$X$ (with multiplicities in the period lengths). Geodesics on~$X$ are precisely the images of geodesics on~$\h$ under the quotient map~$\pi$. We consider all geodesics on~$\h$ and all geodesics on~$X$ to be of unit speed, and we are only interested in time changes along geodesic arcs, not in absolute time parametrizations. In other words, if $\gamma,\eta$ are geodesics on~$\h$ and there exists $t_0\in\R$ such that for all $t\in\R$, $\gamma(t) = \eta(t+t_0)$ (i.e., $\gamma$ is a time-shifted variant of~$\eta$), then we identify $\gamma$ and $\eta$. However, if $g\in\group$ and the action of~$g$ on~$\gamma$ results in $g.\gamma(t) = \gamma(t+t_g)$ for some $t_g\in\R$, then we do notify that $g$ induces a time shift of~$t_g$ on~$\gamma$. (Note that $g$ induces the same time shift on~$\eta$.)

Let $[g]\in [\group]_\hyp$ be represented by the (hyperbolic) element~$g$ and let $\gamma$ be the geodesic on~$\h$ that satisfies 
\begin{align*}
\gamma(-\infty) & = \lim_{t\to\infty}\gamma(t) = \lim_{n\to\infty} g^{-n}.i
\intertext{and}
\gamma(+\infty) & = \lim_{t\to\infty}\gamma(t) = \lim_{n\to\infty} g^n.i\,,
\end{align*}
i.e., $\gamma$ is the unique geodesic from the repelling fixed point of~$g$ to the attracting fixed point of~$g$. As $g$ necessarily fixes the geodesic arc of~$\gamma$, its action on~$\gamma$ results in a time-shifted version of~$\gamma$.
With  
\[
N(g)\coloneqq \max\{\mu^2 : \text{$\mu$ eigenvalue of~$g$} \}\,,
\]
the \emph{norm} of~$g$, the time shift is 
\[
 t_g = \log N(g)\,.
\]
The associated geodesic~$\pi(\gamma)$ on~$X$ is periodic with \emph{length}
\[
 \ell(\pi(\gamma)) = t_g = \log N(g)\,.
\]
If $g$ is primitive, then this length is indeed the minimal period of~$\pi(\gamma)$, otherwise a certain (positive integer) multiple of it. If $h$ is any other representative of~$[g]$, say $h=k^{-1}gk$ with $k\in\group$, and $\eta$ is the geodesic on~$\h$ from the repelling fixed point of~$h$ to its attracting fixed point, then $k.\eta = \gamma$ and hence $\pi(\eta) = \pi(\gamma)$ and $t_h=t_g$. Thus, $[g]$ gets associated to a certain unique periodic geodesic on~$X$ with length $\log N(g)$.

Vice versa, if $\wh\gamma$ is any periodic geodesic on~$X$ with minimal period~$\ell_0$, then we pick a representing geodesic $\gamma$ on~$\h$. We find a unique primitive hyperbolic element~$g\in\group$ such that $\gamma$ connects the repelling fixed point of~$g$ to its attracting fixed point. If we further fix a multiple $\ell$ of~$\ell_0$, say $\ell = n\ell_0$ with $n\in\N$, then $g^n$ is a hyperbolic element with $\ell = \log N(g^n)$. We assign to $\wh\gamma$ and the choice of $n$ the conjugacy class $[g^n]\in [\group]_\hyp$. A different choice of representing geodesic on~$\h$ results in the same conjugacy class.

In this way, we construct a natural bijection between $[\group]_\hyp$ and the periodic geodesics on~$X$ with period lengths considered with multiplicities. 

\subsubsection{Relation between parabolic elements and cusps}\label{sec:parabcusps}
We consider a \emph{cusp end} or \emph{cusp} of~$X$ to be an infinite end of~$X$ of finite area. We let $\cusps(X)$ denote the set of cusps of~$X$. We may characterize cusp ends of~$X$ as the $\group$-orbits of those points in~$\partial\h$ where a Dirichlet fundamental domain for~$X$ touches~$\partial\h$. The latter means that the closure of the fundamental domain in $\h\cup\partial\h$ intersects~$\partial\h$ locally at the considered point of~$\partial\h$ only in this point. These are exactly the fixed points of parabolic elements in~$\group$. Thus, to $[g]\in [\group]_\parab$ we assign the cusp of~$X$ that is characterized by the $\group$-orbit of the fixed point of~$g$.  In this way, we get a bijection between $[\group]_\parprim$ and~$\cusps(X)$. This bijection extends naturally to a map with domain $[\group]_\parab$. (We do not discuss the bijective extension with multiplicities here.)

\subsubsection{Relation between elliptic elements and orbifold points}\label{sec:elliporb}
We let $\sing(X)$ denote the set of singular points (\emph{orbifold points}) of~$X$. In a similar way to the relation in Section~\ref{sec:parabcusps}, $\sing(X)$ is in bijection with $[\group]_\ellprim$. More precisely, let $[g]\in [\group]_\ellprim$. Then any (primitive elliptic) representative~$g$ of~$[g]$ in~$\group$ fixes a unique point in~$\h$, say~$z$. We assign to $[g]$ the element~$\pi(z)\in X$. Then $\pi(z)\in\sing(X)$ and the stabilizer (isotropy) group~$\Iso(\pi(z))$ of~$\pi(z)$ is isomorphic to the stabilizer group~$\Stab_\group(z) = \langle g\rangle$ of~$z$ in~$\group$. In this way, we obtain a bijection between $\sing(X)$ and~$[\group]_\ellprim$ which is compatible with the information on stabilizer groups. It extends to a map using $[\group]_\ellip$ instead of only~$[\group]_\ellprim$ with the caveat that the generators of stabilizer groups are necessarily primitive.

\medskip 

\begin{center}
\framebox{
\begin{minipage}{.9\textwidth}
From now on let $V$ be a finite-dimensional Hermitian vector space (i.e., a complex vector space with
inner product) and let $\twist \colon \group \to \Unit(V)$ be a unitary representation of~$\group$. 
\end{minipage}
}
\end{center}

\medskip

\subsection{Selberg zeta function}\label{sec:szfdef}
The \emph{Selberg zeta function} $Z_{X,\twist}$ for the pair $(X,\twist)$ is determined by the infinite product 
\begin{equation}\label{def:zeta}
    Z_{X,\twist}(s) \coloneqq \prod_{[g] \in [\group]_\hypprim}
\prod_{k=0}^\infty \det \left( 1 - \twist(g) N(g)^{-(s+k)}\right)\,,
\end{equation}
as already mentioned in the Introduction. This infinite product converges for $s\in\CC$, $\Rea s > \delta$ with $\delta$ being the Hausdorff dimension of the limit set of~$\group$. (For our purposes it is sufficient to know that it converges for $\Rea s$ sufficiently large.)
The factors in the infinite product on the right hand side of~\eqref{def:zeta} are indeed invariant under the choice of representatives for the $\group$-conjugacy classes of primitive hyperbolic elements as the discussion in Section~\ref{sec:hypgeod} immediately implies.

This zeta function was first defined by Selberg~\cite[(3.3)]{Selberg} for \emph{finite-area} hyperbolic orbisurfaces, and its convergence and many more properties were studied. We refer also to~\cite[Chapter~5.1]{Venkov} and, in particular for \emph{infinite-area} situations, to~\cite{FP_szf}. As a by-product of our results we will establish the meromorphic continuation of~$Z_{X,\twist}$ to all of~$\C$. For several setups, this is already known (and not our main result anyway) by different methods. We refer to \cite{Selberg, Venkov, FP_szf}. Furthermore, \cite[Theorem~D, Proposition~B]{FP_szf} also discusses the realm of extendability of the definition in~\eqref{def:zeta} to non-unitary representations, and establishes meromorphic continuation in certain setups.

If $\twist$ is the trivial character of~$\group$, i.e., $\twist=\mathbf{1}\colon \group\to\C^\times$, $g\mapsto 1$, then $Z_{X,\twist}=Z_{X,\mathbf{1}}$ coincides with the untwisted (classical) Selberg zeta function. We will often denote it by~$Z_X$, omitting the trivial twist in the notation.

\subsection{Resonances}\label{sec:resonances}
A substantial amount of the necessary spectral theoretic results needed here have been achieved in~\cite{DFP, DFP2}. We recollect some parts here. 

The representation~$\twist\colon \group\to \Unit(V)$ induces the Hermitian vector orbibundle
\begin{equation}\label{eq:Etwist}
E_\twist\coloneqq\h\times_\twist V\to X=\group\backslash\h\,,\quad [z,v]\mapsto [z]\,,
\end{equation}
with typical fiber~$V$. Here, $[z,v]$ denotes the element of~$E_\twist$ represented by $(z,v)\in\h\times V$, i.e., the equivalence class
\[
 [z,v] = \{ (g.z,\twist(g)v) : g\in\group\}\,.
\]
Further, $[z]$ denotes the element in~$X=\group\backslash\h$ represented by $z\in\h$. Using the canonical connection on this orbibundle, we may define the Laplacian~$\Delta_{X,\twist}$ associated to $(X,\twist)$ as an operator on the space~$C^\infty(X,E_\twist)$ of smooth sections of the orbibundle in~\eqref{eq:Etwist}. The space $C^\infty(X,E_\twist)$ is canonically isomorphic to the space $C^\infty(\h, V)_\ssec$ of smooth functions $f\colon \h\to V$ obeying the equivariance relation
\[
f(g.z) = \twist(g)f(z)
\]
for all $g\in\group$, $z\in\h$ (see \cite[Lemma~3.3]{DFP}). Under this isomorphism, the Laplacian~$\Delta_{X,\twist}$ becomes the classical hyperbolic Laplacian $\Delta_\h$, but applied to vector-valued functions. We denote by $C^\infty_c(X,E_\twist)$ the space of compactly supported functions in~$C^\infty(X,E_\twist)$. Further, we endow $X$ with the Haar measure~$\mu_X$ induced from~$\h$, and we denote by $L^2(X,E_\twist)$ the space of (equivalence classes) of square-integrable sections of the orbibundle in~\eqref{eq:Etwist} with inner product induced by the inner product of~$V$.
The operator
\[
\Delta_{X,\twist} \colon C^\infty_c(X,E_\twist)\to L^2(X,E_\twist)
\]
is essentially self-adjoint. We use $\Delta_{X,\twist}$ to denote its extension as well, thus,
\[
\LapTwist \colon D(\LapTwist) \to L^2(X, \bundle)\,,
\]
where $D(\LapTwist)$ is the unique self-adjoint domain of $\LapTwist$ which satisfies $C^\infty_c(X, \bundle) \subseteq D(\LapTwist) \subseteq L^2(X, \bundle)$. See~\cite[Lemma~3.4]{DFP}.
Its resolvent
\[
R_{X,\twist}(s) \coloneqq \bigl(\Delta_{X,\twist}-s(1-s)\bigr)^{-1}\colon L^2(X,E_\twist) \to L^2(X,E_\twist)
\]
is well-defined and holomorphic on $\{ s\in\C : \Rea s>\tfrac12\,,\ s\notin [\tfrac12,1]\}$ by~\cite[Proposition~5.1]{DFP}. Let $L^2_\cpt(X,E_\twist)$ denote the essentially compactly supported elements in~$L^2(X,E_\twist)$, and let $L^2_\loc(X,E_\twist)$ denote the space of (equivalence classes of) measurable functions that are square-integrable on all compact subsets of~$X$. As an operator
\[
R_{X,\twist}(s) \colon L^2_\cpt(X,E_\twist) \to L^2_\loc(X,E_\twist)\,,
\]
the resolvent~$R_{X,\twist}$ admits a meromorphic continuation to all of~$\C$ with all poles being of finite multiplicity~\cite[Theorem~A]{DFP}. We continue to denote this meromorphic continuation by~$R_{X,\twist}$. Its poles are the \emph{resonances}. For any $s\in\C$, we let $m_{X,\twist}(s)$ denotes its multiplicity as a resonance, and we let $\mc R_{X,\twist}$ denote the multiset of resonances.

\subsection{Weierstrass product of resonances} 
We define the \emph{Weierstrass product of the resonances} as the Hadamard product $\ProdTwist\colon\C\to\C$,
\begin{equation}\label{eq:weierstrass}
 \ProdTwist(s) \coloneqq s^{m_{X,\twist}(0)} \prod_{\substack{\mu\in \mc R_{X,\twist} \\ \mu \not= 0}} E_2\left( \frac{s}{\mu} \right)\,,
\end{equation}
of order~$2$, where
\begin{equation}\label{eq:2ndfactor}
E_2(z) \coloneqq (1-z) \exp\left( z + \frac{z^2}{2} \right)
\end{equation}
is the second elementary factor. By~\cite[Theorem~B]{DFP}, the resonance counting function (for $r>0$)
\[
N_{X,\twist}(r) \coloneqq \#\{ \mu\in \mc R_{X,\twist} : |\mu|<r \}
\]
satisfies the asymptotics
\[
N_{X,\twist}(r) = O(r^2)\qquad \text{as $r\to\infty$}\,.
\]
By well-known results in complex analysis,
the function~$\ProdTwist$ is therefore entire and of order~$2$. Its zero set coincides with the resonance set~$\mc R_{X,\twist}$, including multiplicities. See, e.g., \cite[Chapter~5]{Stein_Shakarchi}.

\subsection{Eigenvalues and multiplicities}\label{sec:EV}
For any unitary endomorphism~$A$ of~$V$, we let $\EV(A)$ denote the multiset of eigenvalues of~$A$. Further, for any~$\lambda\in\C$ we let $m_A(\lambda)$ denote the multiplicity of~$\lambda$ as an eigenvalue of~$A$. We note that $m_A(\lambda)=0$ if and only if $\lambda$ is not an eigenvalue of~$A$.
By the unitarity of~$A$, algebraic and geometric multiplicities of the
eigenvalues of~$A$ coincide, and
\[
\sum_{\lambda\in\C} m_A(\lambda) = \dim V = \#\EV(A)\,.
\]

\subsection{Degrees of singularity at ends of~\texorpdfstring{$X$}{X}}
We require two types of degrees of singularity at the ends of~$X$.
For their definitions we first note that for any $g\in\group$ and any $\lambda\in\C$, the multiplicity~$m_{\twist(g)}(\lambda)$ of $\lambda$ as an eigenvalue of~$\twist(g)$ is independent of the $\group$-conjugacy class of~$g$.

\subsubsection{Degree of singularity at cusps}
Let $c$ be a cusp of~$X$. The \emph{degree of singularity}~$\singcusp_\twist(c)$ of~$c$ with respect to~$\twist$ is the dimension of the singularity part of~$c$. To be more explicit, let $[g]$ be the element in $[\group]_\parprim$ associated to~$c$ as explained in Section~\ref{sec:parabcusps}. Then (see, e.g.,~\cite{Phillips_scattering_twist, Venkov_book})
\[
\singcusp_\twist(c) = m_{\twist(g)}(1)\,.
\]
For this reason, $m_{\twist(g)}(1)$ is commonly also called the \emph{degree of singularity} of~$g$ (or $[g]$) relative to~$\twist$.

The sum of these degrees about all cusps,
\begin{align}\label{eq:n_p}
    n_p &\coloneqq n_p(\twist) \coloneqq \sum_{c\in\cusps(X)} \singcusp_\twist(c) = \sum_{[g] \in [\group]_\parprim} m_{\twist(g)}(1)\,,
\end{align}
is the \emph{degree of singularity at cusps} of~$\group$ or~$X$ relative to~$\twist$. This number reflects the overall ``openness'' of the cusps. In the proof of Theorem~\ref{thm:selberg-factorization}  we will see that it is precisely the weight for the singularity contributions from the cusps (see Lemma~\ref{lem:decomp-phi}).

The change of notation from~$\singcusp$ to $n_p$ reflects another aspect of this degree: if $\twist$ is trivial and one-dimensional, then $n_p$ is the number of cusps of~$X$. Thus, for arbitrary~$\twist$, we may understand $n_p$ as the number of ``effective cusps'' of~$(X, \twist)$. The subscript~$p$ at~$n_p$ should indicate that we need to count cusps weighted by the influence of~$\twist$ applied to the \emph{parabolic} elements of~$\Gamma$.

\subsubsection{Degree of singularity at disk ends}
If $X$ is a \emph{parabolic cylinder}, i.e., if its fundamental group~$\group$ is conjugate within~$\PSL_2(\R)$ to the group generated by $\textbmat{1}{1}{0}{1}$, then $X$ has two ends: a cusp end and a disk end. Parabolic cylinders are the only hyperbolic orbisurfaces that have disk ends. We set
\[
n_d \coloneqq n_d(\twist) \coloneqq
\begin{cases}
n_p(\twist) & \text{if $X$ is a parabolic cylinder}
\\
0 & \text{otherwise.}
\end{cases}
\]
We call $n_d$ the \emph{degree of singularity at disk ends} of~$X$ relative to~$\twist$. This number encodes the openness degree of~$(X,\twist)$ at this kind of end.

\subsection{Euler characteristics}
We consider two types of Euler characteristics for~$X$, namely the topological and the orbifold one.

\subsubsection{Topological Euler characteristic}
Let $n$ be the number of ends of~$X$, and let $g$ be the genus of~$X$. The \emph{topological Euler characteristic} of~$X$ is well-known to be
\[
 \eulertop(X) = 2 - 2g - n\,.
\]

\subsubsection{Orbifold Euler characteristic}
If $X$ has no orbifold points, then its \emph{orbifold Euler characteristic}, $\eulerorb(X)$, coincides with the topological Euler characteristic~$\eulertop(X)$.

Let now $X=\group\backslash\h$ be any hyperbolic orbisurface. We provide two ways to calculate the orbifold Euler characteristic, $\eulerorb(X)$, of~$X$. For the first variant, which we use here as a definition, we pick a torsion-free subgroup~$\wt\group$ of~$\group$ of finite index~$[\group:\wt\group]$. Such subgroups exist by Selberg's Lemma~\cite[Lemma~8]{Selberg_lemma}.

Then $\wt X \coloneqq \wt\group\backslash\h$ is a hyperbolic surface (without orbifold points). The \emph{orbifold Euler characteristic} of~$X$ is
\[
\eulerorb(X) \coloneqq \frac{\eulerorb(\wt X)}{[\group:\wt\group]} = \frac{\eulertop(\wt X)}{[\group:\wt\group]}\,.
\]
One easily checks that this definition is indeed independent of the choice of~$\wt\group$.

The second variant of calculating the orbifold Euler characteristic of~$X$, presented in Lemma~\ref{lem:top_orb_euler} below, exhibits the relation to the topological Euler characteristic of~$X$ and the contribution of each orbifold point of~$X$. We recall from Section~\ref{sec:elliporb} that, for an orbifold point~$w\in X$ (i.e., $w\in\sing(X)$), we denote by $\Iso(w)$ its isotropy group.

\begin{lemma}\label{lem:top_orb_euler}
The orbifold Euler characteristic of~$X$ satisfies
\[
\eulerorb(X) = \eulertop(X) - \sum_{w\in\sing(X)}\left( 1 - \frac{1}{\#\Iso(w)} \right)\,.
\]
\end{lemma}

\begin{proof}
For \emph{compact} hyperbolic orbisurfaces, this statement is shown in~\cite[p.~427]{Scott83}. In our situation of a \emph{non-compact} hyperbolic orbisurface~$X$, the proof is almost identical as all orbifold points of~$X$ are contained in a compact subset of~$X$. For the convenience of the reader, we provide a short summary of the proof.

Pick a family~$(D_w)_{w\in\sing(X)}$ of bounded, open, connected, simply connected, pairwise disjoint subsets of~$X$ with smooth boundary such that for each $w\in\sing(X)$, the set~$D_w$ contains~$w$, is $\Iso(w)$-invariant (as a set) and has topological Euler characteristic~$1$. It is indeed possible; e.g., one can pick suitable Euclidean disks about representatives of the orbifolds points of~$X$ in~$\h$ and take for $(D_w)_w$ the subsets in~$X$ corresponding to these disks.

By Selberg's Lemma we find a torsion-free subgroup~$\wt\group$ of~$\group$ of finite index. We set $\wt X \coloneqq \wt\group\backslash\h$ and let
\[
\pr\colon \wt X  \to X = \group\backslash\h
\]
denote the canonical projection map. For each $w\in\sing(X)$, $\pr^{-1}(D_w)\subseteq \wt X$ is a disjoint union of $[\group:\wt\group]/\#\Iso(w)$ connected components, each of which has topological Euler characteristic~$1$. Thus
\[
\eulertop(\pr^{-1}(D_w)) = \frac{[\group:\wt\group]}{\#\Iso(w)}\,.
\]
Further,
\[
W \coloneqq X\setminus\bigcup_{w\in\sing(X)}D_w
\]
is a smooth manifold with boundary, and hence $\wt W\coloneqq \pr^{-1}(W)$ as well. As $\pr\colon \wt W \to W$ is a covering of manifolds of degree~$[\group:\wt\group]$,
\[
\eulertop(\wt W) = [\group:\wt\group]\,\eulertop(W)\,.
\]
Combining these results now allows us to calculate the orbifold Euler characteristic of~$X$ as
\begin{align*}
\eulerorb(X) & = \frac{\eulertop(\wt X)}{[\group:\wt\group]}
\\
& = \frac{1}{[\group:\wt\group]}\left(\eulertop(\wt W) + \sum_{w\in\sing(X)}\eulertop(\pr^{-1}(D_w))  \right)
\\
& = \eulertop(W) + \sum_{w\in\sing(X)} \frac{1}{\#\Iso(w)}
\\
& = \eulertop(X) - \sum_{w\in\sing(X)} \eulertop(D_w) + \sum_{w\in\sing(X)} \frac{1}{\#\Iso(w)}
\\
& = \eulertop(X) - \sum_{w\in\sing(X)}\left(1-\frac{1}{\#\Iso(w)}\right)\,.
\end{align*}
This completes the proof.
\end{proof}

\subsection{Special functions}\label{sec:special-functions}
We require some special functions. These are needed to capture the zeros and poles of the Selberg zeta function that are not Laplacian resonances. The first function, i.e., the function~$G_\infty$, is as in~\cite{BJP, Borthwick_book}, to which we refer for all claimed properties. (Note that $G_\infty^{\eulertop(X)}$ is called $Z_\infty$ in~\cite{BJP}.) The other two functions, i.e., $G_{X^\wedge,\twist}$ and $\gammafunc_{X^\wedge,\twist}$, handle the contributions from the orbifold points of~$X$. As they are not present in~\cite{BJP, Borthwick_book}, we discuss their properties in more detail.

\subsubsection{The function~\texorpdfstring{$G_\infty$}{Ginfty}}
\label{sec:Ginfty}
For $s\in\C$ we set
\begin{equation}\label{eq:def-barnes}
G_\infty(s) = \frac1{2\pi} e^{s-s^2 - \gamma(s-1)^2}\,\gammafunc(s)\,\prod_{k\in\N} E_2\left(\frac{1-s}{k}\right)^{2k}\,,
\end{equation}
where $\gamma$ is the Euler-Mascheroni constant and $E_2$ is the second elementary factor~\eqref{eq:2ndfactor}. Using the Barnes $G$-function~$G$ (see~\cite{Barnes}),
\[
G(s) = (2\pi)^{\frac{s-1}2} e^{\frac12(s-s^2-\gamma(s-1)^2)}\prod_{k\in\N}E_2\left(\frac{1-s}k\right)^{k}
\]
for $s\in\C$, the function~$G_\infty$ becomes
\[
G_\infty(s) = \frac{1}{(2\pi)^s} \Gamma(s)\, G(s)^2 = \frac{1}{(2\pi)^s} G(s) G(s+1)\,.
\]
Here, the first equality is a straightforward reformulation of~\eqref{eq:def-barnes}, and the second equality follows from a product identity for~$\Gamma(s)$ or, equivalently, the identity $G(s+1) = \Gamma(s)G(s)$ (see~\cite{Barnes}). By Hadamard's factorization theorem, $G$ is an entire function of order~$2$. Its zero set is $-\N_0$, the zero at $-m$ for $m\in\N_0$ has order $m+1$. Therefore~$G_\infty$ is entire and of order~$2$; in particular, the poles of~$\Gamma$ are indeed compensated by zeros of the infinite product in~\eqref{eq:def-barnes}. Its zero set is~$-\N_0$, and the order of $-m\in -\N_0$ as a zero is $2m+1$.

\subsubsection{The function~\texorpdfstring{$G_{X^\wedge,\twist}$}{GXwedge}}
The function~$G_{X^\wedge, \twist}$ that we are about to define is a finite product of functions, each of which provides the contribution of a single orbifold point of~$X$. We start by defining the latter functions.

For $q \in \N$, $\alpha \in \{0,\dotsc,q-1\}$ and $m\in\N_0$ we set
\begin{align}\label{eq:Nq}
    N_q(\alpha,m) & \coloneqq \# \big( (\alpha + q\Z) \cap \{-m,\dotsc,m\} \big)
    \\ 
    & \ = \# \{ k\in\Z : |k|\leq m\,,\ k\equiv \alpha\ \text{mod}\ q\} \nonumber
    \\
    & \ = 1 + \floor*{\frac{m + \alpha}{q}} + \floor*{\frac{m - \alpha}{q}}\,. \label{eq:count-lattice-points}
\end{align}
\begin{lemma}\label{lem:barnes-cone}
For each $q\in\N$, $\alpha\in\{0,\ldots, q-1\}$, the map
\[
H(s) \coloneqq G_\infty(s)^{q-1} \prod_{m=0}^{q-1} \gammafunc\left( \frac{s+m}{q} \right)^{qN_q(\alpha,m)-(2m+1)}
\]
is entire. It admits an entire $q$-th root function. It can be chosen to be positive on~$(0,\infty)$, hence it coincides with the principal $q$-th root function on~$(0,\infty)$. The zeros of any root function of~$H$ are in $-\N_0$. The order of $-n \in -\N_0$ as a zero is $(2n+1) - N_q(\alpha,n)$.
\end{lemma}

\begin{proof}
We start by showing that~$H$ is entire. As $G_\infty$ and $\gammafunc$ are meromorphic on all of~$\C$, and $\gammafunc$ does not have zeros, the map $H$ is meromorphic on~$\C$. Its potential poles are determined by the zeros and poles of~$G_\infty$ and~$\gammafunc$. In what follows we will show that $H$ does not have any poles, and hence it is entire. For later purposes, we at once also determine the location and order of potential zeros of~$H$. Also these are determined by the zeros and poles of~$G_\infty$ and~$\gammafunc$.

The function~$G_\infty$ is entire. Its zero set is $-\N_0$, and the order of $-k\in-\N_0$ as a zero is $(2k+1)$. The gamma function~$\gammafunc$ is meromorphic without zeros. Its poles are at $-\N_0$, each pole is simple. Therefore the potential poles and zeros of~$H$ are in $-\N_0$. To simplify the following argumentation, we understand poles as zeros of negative order. The order of $-n\in-\N_0$ as a zero is
\begin{equation}\label{eq:orderzero}
(q-1)(2n+1) - \sum_{\substack{m=0 \\ m-n\in -q\N_0}}^{q-1} \bigl( qN_q(\alpha,m) - (2m+1) \bigr)\,.
\end{equation}
We now calculate this value more explicitly to show that it is non-negative. Let $m\in\{0,\ldots, q-1\}$ be the unique value such that $m\equiv n \mod q$. Thus,
\[
m = n + bq
\]
for a unique $b=b(n)\in\Z$. The admissible range for $m$ combined with the positivity of~$n$ implies that $b\in -\N_0$. Then \eqref{eq:orderzero} becomes
\begin{align*}
\left( q- 1 \right)(2n+1) - qN_q(\alpha, n+bq) + 2n &+ 2bq + 1
\\
& = q(2n+1) + 2bq - qN_q(\alpha, n+bq)\,.
\end{align*}
We have
\begin{align*}
N_q(\alpha, n+bq) & = 1 + \left\lfloor \frac{n+bq+\alpha}q \right\rfloor +  \left\lfloor \frac{n+bq-\alpha}q \right\rfloor
\\
& = 1 + 2b + \left\lfloor \frac{n+\alpha}q \right\rfloor + \left\lfloor \frac{n-\alpha}q \right\rfloor  = 2b + N_q(\alpha,n)\,.
\end{align*}
Thus, the order of $-n$ as a zero of~$H$ is
\begin{equation}\label{H:orderzero}
q(2n+1) - qN_q(\alpha,n) \geq 0\,,
\end{equation}
and hence $H$ is entire.

To show that the map $H$ admits an entire $q$-th root function we take advantage of the fact that this is the case if and only if for each closed ($C^1$-)path~$\gamma$ in~$\C$, we have (see, e.g., \cite[Theorem~33.3.2]{Napkin})
\[
\int_\gamma \frac{H'}{H}(s)\, ds \in 2\pi i q \Z\,.
\]

Obviously, the function $H'/H$ is meromorphic on~$\C$ with being holomorphic on~$\C\setminus (-\N_0)$.
We now pick $-n\in-\N_0$ and let $\gamma$ be a closed path in~$\C$ with winding number~$1$ around $-n\in-\N_0$ and $0$ around any other point in $-\N_0$. By the argument principle
\begin{align*}
\int_\gamma \frac{H'}{H}(s)\,ds = 2\pi i q\bigl((2n+1) - N_q(\alpha,n)\bigr)  \in 2\pi i q\Z \,,
\end{align*}
using the result~\eqref{H:orderzero} on the order of the zeros (and poles) of~$H$.
Again using the argument principle, for any closed path~$\gamma$ in~$\C$ not enclosing any element of~$-\N_0$, the integral of $H'/H$ along~$\gamma$ vanishes. Further, the calculation for any other closed path reduces to the previous two situations. Therefore $H$ admits an entire $q$-th root function.

Any two $q$-th root functions for~$H$ differ by a multiplicative constant of the form $\exp(2 \pi i k/q)$ for some $k\in\{0,\ldots, q-1\}$, hence a $q$-th root of unity. Further, on $(0,\infty)$, the gamma function factors in the definition of~$H$  as well as $G_\infty$ are positive, and hence so is $H$. Therefore, for $s\in (0,\infty)$, any $q$-th root function of~$H$ is of the form
\[
e^{2\pi i \frac{k}{q}} |H(s)|^\frac1q
\]
for some $k\in\{0,\ldots, q-1\}$, and all such options are realized by some choice of a root function. \textit{A priori}, $k$ may depend on~$s$, but discreteness of the set of choices and continuity of root functions yield that $k = k(s)$ is constant on $(0,\infty)$. Hence we may choose $k=0$ and thereby calibrate the chosen $q$-th root function of~$H$ to be the principal $q$-th root function on~$(0,\infty)$.
\end{proof}

\begin{remdef}\label{remdef:root}
Let $q\in\N$ and $\alpha\in\{0,\ldots, q-1\}$.
\begin{enumerate}[label=$\mathrm{(\roman*)}$, ref=$\mathrm{\roman*}$]
\item For any $n\in\N_0$, we have $N_q(\alpha,n) \leq 2n + 1$. Therefore, for any $q$-th root function of~$H$ in Lemma~\ref{lem:barnes-cone}, the order of $-n\in\N_0$ as a zero is indeed non-negative.
\item Throughout we will use the $q$-th root function of~$H$ that coincides with the principal $q$-th root function on~$(0,\infty)$, which we denote $G_{q,\alpha}$. For simplicity, we write
\begin{align}\label{eq:def-barnes-cone}
    G_{q,\alpha}(s) = G_\infty(s)^{1-\frac1q} \prod_{m=0}^{q-1} \gammafunc\left( \frac{s+m}{q} \right)^{N_q(\alpha,m)- \frac{2m+1}q}
\end{align}
with the $q$-th root of complex numbers on the right hand side being understood in way consistent with this choice of root function.
\item\label{remdef:rootiii} On the cut plane $\C\setminus (-\infty,0]$, the map~$H$ admits a logarithm as the calculation of the path integrals of $H'/H$ in the proof of Lemma~\ref{lem:barnes-cone} imply. Therefore, on the cut plane,
any $q$-th root function of~$H$ can be defined via a logarithm of~$H$ (with $H$ necessarily restricted to $\C\setminus (-\infty,0]$). By an analogous argument as in the proof of Lemma~\ref{lem:barnes-cone}, we may choose the principal branch of the logarithm, which then yields the principal $q$-th root function for $H$ on~$(0,\infty)$. However, obviously, globally on all of~$\C$, the $q$-th root function of~$H$ is not deduced from a logarithm.
\item Choosing another $q$-th root function of~$H$ will merely result in a change of the precise polynomial~$q$ in Theorem~\ref{thm:selberg-factorization} but not in a qualitative change of the main result.
\end{enumerate}
\end{remdef}

For any elliptic element~$g\in\group$ we denote by $\deg(g)$ the \emph{degree} of~$g$, i.e., the minimal $n\in\N$ such that $g^n=\id$ (in $\PSL_2(\R)$). We recall from Section~\ref{sec:EV} that $\EV(\twist(g))$ denotes the multiset of eigenvalues of~$\twist(g)$. As $\twist(g)^{\deg(g)} = \twist(g^{\deg(g)}) = \twist(\id) = \id$, each eigenvalue of~$\twist(g)$ is of the form
\[
e^{2\pi i \alpha/\deg(g)}
\]
for a unique $\alpha\in\{0,\ldots, \deg(g)-1\}$. We call $\alpha$ the \emph{parameter} of this eigenvalue, and we let $\EVP(\chi(g))$ denote the multiset of the parameters of all eigenvalues of~$\twist(g)$. We note that $\deg(g)$, $\EV(\twist(g))$ and $\EVP(\twist(g))$ only depend on the $\group$-conjugacy class of~$g$, not on~$g$ itself.

With these preparations we now define~$G_{X^\wedge,\twist}\colon\C\to\C$ by
\begin{equation}\label{eq:GXwedge}
G_{X^\wedge,\twist}(s) \coloneqq \prod_{[g]\in[\group]_\ellprim} \prod_{\alpha\in \EVP(\twist(g))} G_{\deg(g),\alpha}(s)\,.
\end{equation}
We emphasize that $G_{X^\wedge,\twist}$ is entire.

\begin{remark}
\begin{enumerate}[label=$\mathrm{(\roman*)}$, ref=$\mathrm{\roman*}$]
\item For $q = 1$ and $\alpha = 0$, we have that $G_{1,0}\equiv 1$.
\item In particular, if $\group$ is torsion-free, then $G_{X^\wedge,\twist}\equiv 1$ and hence does not contribute to the divisor of the Selberg zeta function. More precisely, each factor in~\eqref{eq:GXwedge} is~$1$, and hence indeed each factor separately does not contribute to the divisor.
\item If $\twist$ is the trivial character, we have
\begin{align*}
    G_{X^\wedge}(s) \coloneqq G_{X^\wedge,\twist}(s) = \prod_{\gamma \in [\group]_\ellprim} G_{\deg(\gamma),0}(s)\,.
\end{align*}
\end{enumerate}
\end{remark}

\subsubsection{The function~\texorpdfstring{$\gammafunc_{X^\wedge,\twist}$}{GammaXwedge}}
\label{sec:defGXwedge}

For $\C\setminus(-\infty, 0]$ we define
\begin{align}\label{eq:GammaXwedge}
    \gammafunc_{X^\wedge,\twist}(s) \coloneqq
    \prod_{[g] \in [\group]_\ellprim}  \prod_{\alpha\in\EVP(\twist(g))} \prod_{m=0}^{\deg(g)-1}
    \gammafunc\left( \frac{s+m}{\deg(g)} \right)^{N_{\deg(g)}(\alpha,m) - \frac{2m+1}{\deg(g)}}\,,
\end{align}
where we pick the principal branch of the logarithm for the definition of the roots. See Remark~\ref{remdef:root}\eqref{remdef:rootiii} (and also the proof of Lemma~\ref{lem:barnes-cone}) for a justification. With this choice, the definition of $\gammafunc_{X^\wedge,\twist}$ becomes consistent with the one of~$G_{X^\wedge,\twist}$. We note that $\gammafunc_{X^\wedge,\twist}$ typically does not admit a meromorphic extension to all of~$\C$.
However, on~$\C\setminus(-\infty,0]$ a straightforward calculation shows the identity
\begin{align}\label{eq:barnes_orbifold}
    G_{\infty}(s)^{-\dim(V)\,\eulertop(X)} G_{X^\wedge,\twist}(s) &=
    G_\infty(s)^{-\dim(V)\,\eulerorb(X)} \gammafunc_{X^\wedge,\twist}(s)\,,
\end{align}
where the left hand side is meromorphic on all of~$\C$. Hence, the product of the maps on the right hand side extends meromorphically to all of~$\C$.

\section{Model cases and geometric decomposition}\label{sec:cyclic}
In this section, we will establish Theorem~\ref{thm:selberg-factorization} for all hyperbolic orbisurfaces with cyclic fundamental groups, i.e., $\group = \ang{g}$ for some $g \in \PSL_2(\R)$, and any finite-dimensional unitary representation $\twist\colon\group \to U(V)$. Due to the classification of elements in~$\PSL_2(\R)$, there are four types of such fundamental groups, namely $g$ either being the identity or elliptic or parabolic or hyperbolic, and we will ultimately consider each of these cases separately.

In the case of $g$ being the identity, the space $X=\group\backslash\h$ is just the hyperbolic plane~$\h$. For $g$ being elliptic, the orbifold~$X$ is commonly called a \emph{cone}. For $g$ being parabolic, $X$ is called a \emph{parabolic cylinder} or a \emph{cusp (end)} (see also Section~\ref{sec:parabcusps}), and for $g$ being hyperbolic, $X$ is called a \emph{hyperbolic cylinder}. We will use the discussion below as well to introduce some normalized representing orbifolds, namely the \emph{model cones}, the \emph{model cusp} and the \emph{model hyperbolic cylinders}, including rather common notation for them.

We emphasize that we could subsume the discussing of the case of the identity into the discussion of elliptic elements as the hyperbolic plane is---in a certain sense---a degenerate cone. However, for clarity and as the argument for the identity is considerably easier than for elliptic elements, we will discuss them separately.

Common to all four cases are some considerations for the resolvent of the Laplacian for~$(X,\twist)$, which we will provide first and then discuss Theorem~\ref{thm:selberg-factorization} separately for each case.

Afterwards, in Sections~\ref{sec:modelfunnel}, we introduce a further model case, namely the \emph{model funnel}, which will allow us to survey, in Section~\ref{sec:decomposition}, a geometric decomposition of hyperbolic orbisurfaces, which will be crucial for our further discussions.

\subsection{Properties of the resolvent}
The resolvent~$\ResTwist$ of the Laplacian~$\Delta_{X,\twist}$ is often given by its Schwartz kernel, which deduces from the Schwartz kernel for the resolvent of the Laplacian~$\Delta_\h$ for~$\h$ in the untwisted, scalar setting. As often done, we identify here the resolvent with its Schwartz kernel. In what follows we recollect the necessary information on the resolvents, i.e. their Schwartz kernels.

For the definition of the (Schwartz kernel of) the resolvent~$R_\h$ for $\h$ we require the hypergeometric function~${}_2F_1$ and the regularized hypergeometric function~$\FF$ deducing from it.

For $a,b\in\C$, $c\in\C\setminus(-\N_0)$ and $z\in\C$, $|z|<1$, the \emph{hypergeometric function} is given by
\[
{}_2F_1(a,b;c;z) \coloneqq \sum_{n=0}^\infty \frac{\gammafunc(a+n)\gammafunc(b+n)\gammafunc(c)}{\gammafunc(a)\gammafunc(b)\gammafunc(c+n)} \cdot \frac{z^n}{n!}
\]
and the \emph{regularized hypergeometric function} by
\begin{align*}
    \FF(a,b;c;z) & \coloneqq \frac{1}{\gammafunc(c)} \cdot {}_2F_1(a,b;c;z)
    \\
    & = \sum_{n=0}^\infty \frac{\gammafunc(a+n)\gammafunc(b+n)}{\gammafunc(a)\gammafunc(b)}\frac{1}{\gammafunc(c+n)} \cdot \frac{z^n}{n!}\,.
\end{align*}
We note that $\FF$ is indeed defined for $c\in\C$, not only $c\in\C\setminus(-\N_0)$, by the infinite sum in the last line of the previous formula.

Let $d_\h$ denote the hyperbolic metric on~$\h$, and for $z,w\in\h$ set
\[
\sigma(z,w)\coloneqq \left(\cosh\frac{d_\h(z,w)}2\right)^2\,.
\]
The \emph{resolvent}~$R_\h$ of~$\Delta_\h$ is then
\[
  R_\h(s; z,z') = \frac{\gammafunc(s)^2}{4\pi} \sigma(z,z')^{-s} \FF\bigl(s,s;2s;\sigma(z,z')^{-1}\bigr)
\]
for $s\in\C\setminus(-\N_0)$, $z,z'\in\h$, $z\not=z'$. See, e.g., \cite[Eq. (4.6)]{Borthwick_book}.
By means of the function $g_s\colon (1,\infty)\to\C$ for $s\in\C\setminus(-\N_0)$,
\begin{equation}\label{eq:def_gs}
g_s(x) \coloneqq \frac{\gammafunc(s)^2}{4\pi} x^{-s} \FF\bigl(s,s;2s;x^{-1}\bigr) =  \frac{1}{4\pi}\sum_{n=0}^\infty \frac{\gammafunc(s+n)^2}{\gammafunc(2s+n)} \cdot \frac{x^{-n-s}}{n!}\,,
\end{equation}
we may write
\[
R_\h(s;z,z') = g_s\bigl(\sigma(z,z')\bigr)\,.
\]

The \emph{resolvent}~$\ResTwist$ of~$\LapTwist$ is defined as an average of~$R_\h$ about~$\Gamma$ in the presence of~$\twist$, i.e.,
\begin{equation}\label{def:restwist}
\ResTwist(s;z,z') \coloneqq \sum_{h\in\Gamma} \twist(h)R_\h(s;z,hz')
\end{equation}
for $\Rea s \gtrsim 1$. See, e.g., \cite[Theorem~1.4.5]{Venkov_book}.
The meromorphic continuation of~$R_\h$ and~$\ResTwist$ is known, as already mentioned in Section~\ref{sec:resonances}. For the situations needed here, we refer to~\cite[Theorem~A]{DFP}.

We also require an expression for~$R_\h$ in terms of geodesic polar coordinates. To be more precise, we base the latter at~$i\in\h$.  For $\varphi\in \Sph^1 = \R/\Z$ we denote  by~$\gamma_\varphi$ the unit speed geodesic passing through $i$ at time~$0$ and enclosing an angle to $\partial_y|_i$ at $i$ of~$\varphi$. Then we may identify $\h\setminus\{i\}$ with $(0,\infty)\times \Sph^1$ via the bijection
\[
(0,\infty)\times \Sph^1 \to \h\,,\quad (r,\varphi) \mapsto \gamma_\varphi(t)\,,
\]
the \emph{geodesic polar coordinates}. For any $\mu,\nu\in\C$, we let $P_\nu^\mu$ and $\mathbf{Q}_\nu^\mu$ be the associated Legendre functions as defined in~\cite[(12.04) and (12.06)]{Olver74} and emphasize that $\mathbf{Q}^{-k}_\mu = \mathbf{Q}^k_\mu$ for $k\in\Z$. As in~\cite[Section~8.1]{Borthwick_book} we set
\begin{align}\label{eq:defuk}
        u_k(s;r,r') \coloneqq \begin{cases}
            \gammafunc(s+\abs{k}) P_{-s}^{-\abs{k}}(\cosh r) \mathbf{Q}^{k}_{s-1}(\cosh r') & \text{for $r\leq r'$}\,\\
            \gammafunc(s+\abs{k}) P_{-s}^{-\abs{k}}(\cosh r') \mathbf{Q}^{k}_{s-1}(\cosh r) & \text{for $r\geq r'$}
        \end{cases}
\end{align}
for $s\in\C$, $r,r'\in (0,\infty)$, $k\in\Z$. The set of poles of $u_k$ is contained in $-\N_0$, and $-n\in-\N_0$ is a pole of $u_k$ only if $|k|\leq n$.
By~\cite[(8.5)]{Borthwick_book} the Fourier expansion of $R_\h$ is given in geodesic polar coordinates $(r,\phi), (r',\phi')\in (0,\infty)\times \Sph^1$ by
\begin{align}\label{eq:resgp}
        R_\h(s;r,\phi,r',\phi') = \frac{1}{2\pi} \sum_{k \in \Z} e^{ik(\phi - \phi')} u_k(s;r,r')\,.
\end{align}

Further, for the estimate of the regularized trace~$\Phi_{X,\twist}$ in Lemma~\ref{lem:growth} we will need the following alternative presentation of~$R_\h$ and its asymptotics.
For $\Rea s > 0$, the resolvent~$R_\h$ admits the integral representation
\begin{align}\label{eq:resolv_H_int}
    R_\h(s; z,w) = \frac{1}{4\pi} \int_0^1 \frac{t^{s-1} (1 - t)^{s-1}}{(\sigma(z,w) - t)^s} \,dt\,.
\end{align}
By, e.g., Borthwick~\cite[(4.12)]{Borthwick_book} we have the asymptotics
\begin{align}\label{eq:resolv_H_diag}
    R_\h(s; z, w) = -\frac{1}{2\pi} \left( \log\left( \frac{d(z,w)}{2} \right) + \psi(s) + \gamma \right) + o(1)
\end{align}
as $d(z,w) \to 0$, where $\psi$ is the digamma function and $\gamma$ is the Euler constant.

\subsection{The identity}\label{sec:modelid}
In the case that $\group$ is the trivial group consisting of the identity~$\id\in\PSL_2(\R)$ only, the space~$X=\group\backslash\h$ is the upper half plane~$\h$ and the twist~$\twist$ merely renders the situation multidimensional without any significant interaction between the dimensions as $\twist \equiv \id_V$.

As $\group = \{\id\}$ does not have any hyperbolic elements (equivalently, as $X=\h$ does not have any periodic geodesics), the Selberg zeta function~$Z_{X,\twist}$ for~$(X,\twist)$ is the empty product, thus
\[
Z_{X,\twist} \equiv 1\,,
\]
which clearly is holomorphic on all of~$\C$.

For determining the resonances of~$X=\h$ in the presence of~$\twist$, we observe that
\[
R_{\h,\twist}(s;z,z') = \id_V R_\h(s;z,z')
\]
by~\eqref{def:restwist}. Thus, the resolvent of~$\Delta_{\h,\twist} = \Delta_{X,\twist}$ is just the vector-valued variant of the resolvent of~$\Delta_\h$ without interaction between the dimensions. By~\cite[§~8.1]{Borthwick_book}, the resonance set of~$\h$ is $-\N_0$ with the multiplicity of~$-n\in -\N_0$ as a resonance being $2n+1$. Therefore the resonance set of~$(X,\twist) = (\h,\twist)$ is $-\N_0$ with the multiplicity of~$-n\in -\N_0$ as a resonance being $(2n+1)\dim(V)$.

As $\group=\{\id\}$ does not have elliptic elements or, equivalently, $X=\h$ does not have any orbifold points,
\[
G_{X^\wedge,\twist} \equiv 1\,.
\]
Further, the degree of singularity at cusps as well as the degree of singularity at disk ends vanish, i.e.,
\[
n_p = 0 = n_d\,,
\]
due to the absence of these types of ends. For the topological Euler characteristic of~$X=\h$ we have
\[
\eulertop(\h) = 1\,.
\]
By Section~\ref{sec:Ginfty}, the function $G_\infty(s)^{-\dim(V)\eulertop(\h)} = G_\infty(s)^{-\dim(V)}$ is meromorphic on all of~$\C$ without any zeros. Its set of poles is $-\N_0$ with order of $-n\in-\N_0$ as a pole being $(2n+1)\dim(V)$. Thus, the product
\[
H(s)\coloneqq G_{X^\wedge,\twist}(s)\thin G_\infty(s)^{-\dim(V)\,\eulertop(X)}\cdot
    \frac{\gammafunc\bigl(s-\frac12\bigr)^{n_p}}{\gammafunc\bigl(s+\frac12\bigr)^{n_d}}
    \cdot \ProdTwist(s)
\]
defines an entire function without zeros. As $H$ and $Z_{X,\twist}$ are entire functions with identical zero sets, there exists an entire function $q$ such that
\[
Z_{X,\twist}(s) = e^{q(s)}H(s)
\]
on all of~$\C$. By Hadamard's factorization theorem, $q$ is a polynomial of degree at most~$2$ as $G_\infty$ and $\ProdTwist$ are entire of order~$2$. This proves Theorem~\ref{thm:selberg-factorization} for the hyperbolic plane (and any finite-dimensional unitary representation).

\subsection{Cones}\label{sec:modelcone}
We suppose now that $\group = \ang{g}$ with $g\in\PSL_2(\R)$ being elliptic. By applying a conjugation within~$\PSL_2(\R)$ if needed, we may and shall assume that
\begin{equation}\label{eq:gtheta}
 g = g_\theta \coloneqq \bmat{\cos(\theta)}{\sin(\theta)}{-\sin(\theta)}{\cos(\theta)}
\end{equation}
with $\theta = \pi/q$ for some $q \in \N$, $q>1$. As $\group$ does not contain hyperbolic elements or, equivalently, the \emph{model cone}
\[
C_\theta^\wedge \coloneqq \ang{g_\theta}\backslash\h \quad (=X)
\]
does not have any periodic geodesics, the Selberg zeta function~$Z_{C_\theta^\wedge,\twist}$ for~$(C_\theta^\wedge,\twist)$ is defined by the empty product and hence
\[
Z_{C_\theta^\wedge,\twist} \equiv 1\,.
\]
In what follows we determine first the set of resonances for~$C_\theta^\wedge$ in the presence of~$\twist$ and then evaluate the product of the spectral functions on the right hand side of the formula claimed in Theorem~\ref{thm:selberg-factorization}.
As $\twist(g_\theta)^q = \twist(g_\theta^q) = \twist(\id)=\id_V$,
by~\eqref{def:restwist} the resolvent of~$\Delta_{C_\theta^\wedge,\twist}$ is
\begin{align}\label{eq:rescone}
    R_{C_\theta^\wedge,\twist}(s;z,z') = \sum_{m=0}^{q-1} \twist(g_\theta)^m R_\h(s;z,g_{\theta}^m z')\,.
\end{align}
We recall $N_q$ from~\eqref{eq:count-lattice-points}.

\begin{prop}\label{prop:multcone}
The poles of $R_{C_\theta^\wedge,\twist}$ (thus, the resonances of~$(C_\theta^\wedge,\twist)$) are at $s = -n \in -\N_0$ with multiplicity
\begin{align*}
m_{C_\theta^\wedge,\twist}(-n) = \sum_{\alpha\in\EVP(\twist(g_\theta))} N_q(\alpha,n) \,.
\end{align*}
\end{prop}

\begin{proof}
We take advantage of~$\eqref{eq:resgp}$ and~\eqref{eq:rescone}. As the elliptic fixed point of~$g_\theta$ is~$i$, we have
\begin{align*}
        R_{C_\theta^\wedge,\twist}(s;r,\phi,r',\phi') = \frac1{2\pi}\sum_{m=0}^{q-1} \twist(g_\theta)^m \sum_{k \in \Z} e^{ik(\phi - \phi' - 2m\theta)} u_k(s;r,r')\,.\end{align*}
We note that the factor~$2$ is caused by the fact that $g_\theta$ acts as rotation by~$2\theta$ on the tangent bundle.

As $\twist(g_\theta)^q = \id_V$, we find an eigenbasis~$(\psi_j)_{j=1}^{\dim V}$ of~$V$ such that $\twist(g_\theta) \psi_j = e^{2\pi i \alpha_j/q} \psi_j$ for suitable $\alpha_j \in \{0,\dotsc,q-1\}$. We note that $\{\alpha_1,\ldots, \alpha_q\} = \EVP(\twist(g_\theta))$ (as multisets). Thus, for each $j\in\{1,\ldots,\dim V\}$,
    \begin{align*}
        \ang{R_{C_\theta^\wedge,\twist}(s;r,\phi,r',\phi')\psi_j, \psi_j} = \frac{1}{2\pi} \sum_{k \in \Z} \sum_{m=0}^{q-1} e^{2\pi i m (\alpha_j - k) /q} e^{ik(\phi - \phi')} u_k(s;r,r')\,.
    \end{align*}
From
    \begin{align*}
        \sum_{m=0}^{q-1} e^{2\pi i m (\alpha_j - k)/q} =
        \begin{cases}
            q & \text{if $k \equiv \alpha_j\mod q$}
            \\
            0 & \text{otherwise}
        \end{cases}
    \end{align*}
    we obtain
    \begin{align}\label{eq:resconegp}
        R_{C_\theta^\wedge,\twist}(s;r,\phi,r',\phi')^{(j)} &\coloneqq \ang{R_{C_\theta^\wedge,\twist}(s;r,\phi,r',\phi')\psi_j, \psi_j} \\
        &= \frac{q}{2\pi} \sum_{k \in \Z} e^{i(qk + \alpha_j) (\phi - \phi')} u_{qk + \alpha_j}(s;r,r')\,. \nonumber
    \end{align}
Thus, the poles of~$R_{C_\theta^\wedge,\twist}^{(j)}$ are in $-\N_0$. To calculate the multiplicity of $s=-n$ as a pole, we employ the same argument as in~\cite[Section~8.1]{Borthwick_book}. Using the identity
    \begin{align*}
        \mathbf{Q}_{-n-1}^k(z) = (-1)^n \gammafunc(n+\abs{k}+1) P_n^{-\abs{k}}(z)
    \end{align*}
in~\eqref{eq:defuk} and then in~\eqref{eq:resconegp},
    we obtain that the multiplicity of $-n$ is
    \begin{align*}
        \# \{k \in \Z \setmid \abs{qk + \alpha_j} \leq n\} = N_q(\alpha_j,n)\,.
    \end{align*}
    This implies that the multiplicity of $s=-n$ as a pole of $R_{C_\theta^\wedge,\twist}$ is
    \begin{align*}
        m_{C_\theta^\wedge,\twist}(-n) = \sum_{j=1}^{\dim V} N_q(\alpha_j,n)\,.
    \end{align*}
This completes the proof.
\end{proof}

For the cone~$C_\theta^\wedge$ we have
\[
n_p = 0 = n_d
\]
and
\[
\eulertop(C_\theta^\wedge) = 1\,.
\]
By a straightforward analysis of poles and zeros we obtain (as in Section~\ref{sec:modelid}) that
\[
G_{(C_\theta^\wedge)^\wedge,\twist}(s)\thin G_\infty(s)^{-\dim(V)\,\eulertop(C_\theta^\wedge)}\cdot
    \frac{\gammafunc\bigl(s-\frac12\bigr)^{n_p}}{\gammafunc\bigl(s+\frac12\bigr)^{n_d}}
    \cdot \mathcal{P}_{C_\theta^\wedge,\twist}(s)
\]
defines an entire function (in~$s$) without zeros and hence
\begin{equation}\label{eq:thmcone}
Z_{C_\theta^\wedge,\twist}(s) = e^{q(s)}G_{(C_\theta^\wedge)^\wedge,\twist}(s)\thin G_\infty(s)^{-\dim(V)\,\eulertop(C_\theta^\wedge)}\cdot
    \frac{\gammafunc\bigl(s-\frac12\bigr)^{n_p}}{\gammafunc\bigl(s+\frac12\bigr)^{n_d}}
    \cdot \mathcal{P}_{C_\theta^\wedge,\twist}(s)
\end{equation}
for some entire function~$q$. We consider \eqref{eq:thmcone} in the variant
\[
 G_\infty(s)^{\dim(V)} = e^{q(s)}G_{(C_\theta^\wedge)^\wedge,\twist}(s)\thin \mathcal{P}_{C_\theta^\wedge,\twist}(s)
\]
using $Z_{C_\theta^\wedge,\twist}\equiv 1$. As $G_{(C_\theta^\wedge)^\wedge,\twist}$ is entire of order at most~$2$ (the dominant contribution for the order is the $G_\infty$-factor in~\eqref{eq:def-barnes-cone}), both sides define entire functions of order~$2$. By Hadamard's factorization theorem, $q$ is a polynomial of order at most~$2$. This establishes Theorem~\ref{thm:selberg-factorization} for cones.

\begin{remark}
For $\theta=\pi$ in~\eqref{eq:gtheta}, i.e. $q=1$, the element $g_\theta$ is not elliptic as it is the identity in~$\PSL_2(\R)$. Nevertheless, for $\theta=\pi$ and $\dim V = 1$, the considered situation reduces to the hyperbolic plane~$\h$ and Proposition~\ref{prop:multcone} recovers the multiplicities $m_\h(-n) = 2n+1$.
\end{remark}

\subsection{Parabolic cylinders or cusps}\label{sec:cusps}

Let $\group=\ang{g}$ with $g\in\PSL_2(\R)$ being parabolic. After conjugation within~$\PSL_2(\R)$ if necessary, we may and shall assume that
\[
g = T \coloneqq \bmat{1}{1}{0}{1}
\]
and hence $X=\group\backslash\h$ is the \emph{(model) parabolic cylinder} or \emph{model cusp}
\[
C_\infty \coloneqq \ang{T}\backslash \h\,.
\]
The Selberg zeta function~$Z_{C_\infty,\twist}$ is the empty product, hence
\[
Z_{C_\infty,\twist} \equiv 1\,.
\]
Furthermore, the singularity degrees are
\[
n_p = n_d = m_{\twist(T)}(1)\,,
\]
i.e., the multiplicity of $1$ as an eigenvalue of~$\twist(T)$, which might be nonzero. Therefore, for parabolic cylinders, both gamma factors in Theorem~\ref{thm:selberg-factorization} may have a contribution. As $\group$ does not contain elliptic elements or, equivalently, $C_\infty$ has no orbifold points,
\[
G_{C_\infty^\wedge,\twist}\equiv 1\,.
\]
Moreover,
\[
\eulertop(C_\infty) = 0\,.
\]
By~\cite[Proposition~4.17]{DFP}, the only resonance of~$C_\infty$ is at $\tfrac12$ with multiplicity~$m_{\twist(T)}(1) = n_p$. Therefore the Weierstrass product of the resonances for the parabolic cylinder, $\mathcal{P}_{C_\infty,\twist}\colon\C\to\C$, is
\[
\mathcal{P}_{C_\infty,\twist}(s) = (1-2s)^{n_p}\exp\left( 2n_p(s+s^2) \right)\,.
\]
We obtain, for all $s\in\C$,
\begin{align*}
G_{C_\infty^\wedge,\twist}(s)\, &G_\infty(s)^{-\dim(V)\,\eulertop(C_\infty)}\,\frac{\gammafunc\left( s - \frac12\right)^{n_p}}{\gammafunc\left(s + \frac12\right)^{n_p}}\,\mathcal{P}_{C_\infty,\twist}(s)
\\
&\quad = \frac{\gammafunc\left(s - \frac12\right)^{n_p}}{\left(s-\frac12\right)^{n_p}\gammafunc\left(s - \frac12\right)^{n_p}}\,\mathcal{P}_{C_\infty,\twist}(s)
\\
&\quad = (-2)^{n_p}\exp\left(2n_p(s+s^2) \right)
\\
&= e^{-q(s)} = e^{-q(s)}Z_{C_\infty,\twist}(s)
\end{align*}
with
\[
q(s) \coloneqq -2n_p(s+s^2) - i\pi n_p - n_p\log 2\,.
\]
Since $p$ is a polynomial of degree~$2$, this proves Theorem~\ref{thm:selberg-factorization} for parabolic cylinders.

\subsection{Hyperbolic cylinders}\label{sec:hypcyl}
We now suppose that $\group=\ang{g}$ with $g\in\PSL_2(\R)$ hyperbolic. Up to conjugation within~$\PSL_2(\R)$ we may and shall assume that
\[
g = g_\ell \coloneqq \bmat{e^{\ell/2}}{0}{0}{e^{-\ell/2}}
\]
for some $\ell>0$. Then $X$ is the \emph{model hyperbolic cylinder}
\[
C_\ell \coloneqq \ang{g_\ell} \backslash \h\,.
\]
The hyperbolic surface~$C_{\ell} $ has two primitive periodic geodesics, one of them related to the $\group$-conjugacy class of~$g_\ell$, the other one to the $\group$-conjugacy class of~$g_\ell^{-1}$ (see Section~\ref{sec:hypgeod}), both of length~$\ell$. Thus, for $s\in\C$, $\Rea s \gtrsim 1$, the Selberg zeta function for~$(C_\ell,\twist)$ is
\begin{align*}
Z_{C_\ell,\twist}(s) & = \prod_{k=0}^\infty \prod_{\kappa\in\{\pm1\}} \det\left( 1 - \twist(g_\ell^\kappa)e^{-(s+k)\ell}  \right)
\\
& = \prod_{\kappa\in\{\pm 1\}} \prod_{\lambda\in\EV(\twist(g_\ell))} \prod_{k=0}^\infty \left( 1 - e^{\kappa\log\lambda - (s+k)\ell} \right)\,.
\end{align*}
Here, $\log\lambda$ is any logarithm of~$\lambda$, and hence also negative eigenvalues of~$\twist(g_\ell)$ can be handled. As $\sum_{k\in\N} e^{-k\ell}$ is (absolutely) convergent, the theory of indefinite products yields that
\[
h\colon \C\to\C\,,\quad h(z)\coloneqq \prod_{k=0}^\infty \left( 1-\frac{z}{e^{k\ell}} \right)
\]
is an entire function with $\{e^{k\ell}:k\in\N_0\}$ as zero set and each zero of multiplicity~$1$. See, e.g., \cite[Theorem~2.6.5]{Boas}; the convergence exponent here is~$0$. Thus, $f\colon \C\to\C$,
\[
f(t) \coloneqq h\left(e^{-t\ell}\right)\,,
\]
is an entire function as well with zero set
\[
-\N_0 + \frac{2\pi i}{\ell}\Z\,,
\]
each zero being of multiplicity~$1$. In turn
\[
Z_{C_\ell,\twist}(s) = \prod_{\kappa\in\{\pm 1\}} \prod_{\lambda\in\EV(\twist(g_\ell))} f\left(s-\frac{\kappa}{\ell}\log\lambda\right)
\]
is indeed defined on all of~$\C$, and it is entire with the multiset
\begin{align}\label{eq:hypzeros}
    \bigcup_{\lambda \in \EV(\chi(g_\ell))} \bigcup_{\kappa \in \{\pm 1\}} \left( -\N_0 + \kappa \ell^{-1} ( \log \lambda + 2 \pi i \ZZ) \right)
\end{align}
as zero set. Let $\mathbf{1}\colon\group\to\C^\times$, $h\mapsto 1$, denote the trivial character. As
\[
Z_{C_\ell}(s)\coloneqq Z_{C_\ell,\mathbf{1}}(s) = \prod_{k=0}^\infty \left( 1 - e^{-(s+k)\ell}\right)^2
\]
we obtain
\[
\left|Z_{C_\ell,\twist}(s)\right| = \prod_{\kappa\in\{\pm 1\}} \prod_{\lambda\in\EV(\twist(g_\ell))} \left| Z_{C_\ell}\left(s-\frac{\kappa}{\ell}\log\lambda\right) \right|^{\frac12}\,.
\]
(We do the detour with the function~$f$ in the discussion above to avoid to discuss the square root of $Z_{C_\ell}$.) By~\cite[Proposition~15.11]{Borthwick_book}, $Z_{C_\ell}$ has order~$2$, and hence $Z_{C_\ell,\twist}$ has order~$2$ as well. An alternative proof of $Z_{C_\ell,\twist}$ being entire of order~$2$ and finding its zero set can be given based on twisted transfer operators by combination of~\cite{FP_szf} and~\cite{APW}, mimicking the proof of~\cite[Proposition~15.11]{Borthwick_book}.

For $(C_\ell, \twist)$ we have
\[
n_p = 0 = n_d\,,\quad G_{C_\ell^\wedge,\twist} \equiv 1\,,\quad \eulertop(C_\ell) = 0\,.
\]
By \cite[Proposition~4.8]{DFP}, the resonance set for~$(C_\ell,\twist)$ is the multiset in~\eqref{eq:hypzeros}. Therefore Hadamard's factorization theorem shows that $Z_{C_\ell,\twist}$ and $\Prd_{C_\ell,\twist}$ coincide up to a multiplicative zero-free factor, more precisely,
\[
 Z_{C_\ell,\twist}(s) = e^{q(s)} \Prd_{C_\ell,\twist}
\]
for some polynomial~$q$ of degree at most~$2$ (as both functions, $Z_{C_\ell,\twist}$ and $\Prd_{C_\ell,\twist}$, are of order~$2$). This proves Theorem~\ref{thm:selberg-factorization} for hyperbolic cylinders.

\subsection{Model funnels}\label{sec:modelfunnel}
Hyperbolic orbisurfaces may have infinite ends of finite area (called cusp ends, see Section~\ref{sec:parabcusps}) and infinite ends of infinite area. If the latter type is not a disk end, meaning, if the considered orbisurface is not a parabolic cylinder, then this end is called a \emph{funnel} or \emph{funnel end}. The \emph{model funnel} can be thought of as half of a (model) hyperbolic cylinder. More precisely, for any $\ell\in(0,\infty)$, we define the \emph{model funnel} as
\begin{align*}
    F_\ell \coloneqq C_\ell \cap \{x + iy \in \C \colon x \geq 0\}\,.
\end{align*}

\subsection{Geometric decomposition}\label{sec:decomposition}
We may and often shall decompose $X$ into a union of a compact core, a finite number of funnel and cusp ends,
\begin{align*}
    X = X_{\interior} \sqcup X_f \sqcup X_c\,.
\end{align*}
We refer to, e.g., \cite{Borthwick_book,DFP,DFP2} for details and recollect here only the information that we will need later on.

The \emph{compact core}~$X_{\interior}$ is compact, of finite area and does not have any funnels or cusps ($\interior$ stands for ``inner part''). The \emph{funnel part}~$X_f$ decomposes further into a finite number (here $n_f$) of connected components,
\[
X_f = \bigsqcup_{j=1}^{n_f} X_{f,j}\,,
\]
and each such component~$X_{f,j}$ is isometric to a model funnel~$F_{\ell_j}$ for some $\ell_j \in (0,\infty)$, $j\in\{1,\ldots, n_f\}$. Likewise, the \emph{cusp part}~$X_c$ decomposes into a finite number (here $n_c$) of connected components,
\[
X_c = \bigsqcup_{k=1}^{n_c} X_{c,k}\,,
\]
and each such component~$X_{c,k}$ is isometric to the model cusp~$C_\infty$ from Section~\ref{sec:cusps}. See \cite[Section 2.4.1]{Borthwick_book} or \cite[Section 3.2]{DFP} for more details.

\section{Two examples}\label{sec:examples}

Before starting with the proof of Theorem~\ref{thm:selberg-factorization} for arbitrary geometrically finite hyperbolic orbisurfaces and arbitrary finite-dimensional unitary representations in the subsequent section we will present in this section, with two examples, how Theorem~\ref{thm:selberg-factorization} can be used to determine the resonance set in one situation given that the resonance set in a certain suitably related situation is known. More precisely, we let $\group, \Lambda \subset \PSL_2(\R)$ be two discrete subgroups such that $X = \group \backslash \h$ and $Y = \Lambda \backslash \h$ have infinite area. Let $\eta\colon \Lambda\to \Unit(V)$ be a finite-dimensional unitary representation of~$\Lambda$. If $\Lambda \leq \group$ and $[\group : \Lambda] < \infty$, then the Venkov--Zograf formula (for infinite-area hyperbolic orbisurfaces see~\cite[Section~7]{FP_szf}) states that the Selberg zeta functions for~$(Y,\eta)$ and~$(X,\Ind_\Lambda^\Gamma\eta)$ are equal. In this situation, Theorem~\ref{thm:selberg-factorization} allows us to deduce a relationship between the sets of resonances of~$(Y,\eta)$ and~$(X,\Ind_\Lambda^\Gamma\eta)$. We now illustrate this in two examples.

\begin{example}\label{EX:elementary}
Let $\ell \in (0,\infty)$ and set
\[
h_\ell \coloneqq \bmat{e^{\ell/2}}{0}{0}{e^{-\ell/2}}\quad\text{and}\quad 
S \coloneqq \bmat{0}{1}{-1}{0}\,.
\]
Let 
\[
\Lambda \coloneqq \ang{h_\ell} \quad\text{and}\quad \group \coloneqq \ang{h_\ell, S}\,.
\]
The hyperbolic orbisurface~$Y \coloneqq \Lambda \backslash \h$ is a hyperbolic cylinder (denoted $C_\ell$ in Section~\ref{sec:hypcyl}), whereas $X \coloneqq \group \backslash \h$ is a hyperbolic orbisurface with one funnel end and two elliptic points, being twofold-covered by~$Y$.

In order to discuss the untwisted Selberg zeta functions of~$X$ and~$Y$, we consider the fundamental groups~$\Lambda$ and~$\group$.
Obviously, $\Lambda \leq \group$, more precisely $[\group : \Lambda] = 2$ and
\begin{align*}
 \group & = \{ S h_\ell^n \setmid n\in\Z\} \cup \{ h_\ell^n \setmid n\in\Z\}
 \\
 & = S \Lambda \cup \Lambda\,.
\end{align*}
The latter can also be deduced immediately from the identity
\begin{equation}\label{eq:Sconjhyp}
S h_\ell^n S = h_\ell^{-n}\,,
\end{equation}
valid for all~$n\in\Z$, within~$\PSL(2,\R)$.
Obviously, $\group$ and $\Lambda$ have the same hyperbolic elements.
However, their conjugacy classes of hyperbolic elements, and in particular their conjugacy classes of primitive hyperbolic elements, differ.
While $\Lambda$ has two conjugacy classes of primitive hyperbolic elements, namely
\[
[\Lambda]_\hypprim = \{ [h_\ell], [h_\ell^{-1}] \}\,,
\]
the group~$\group$ has only one, namely
\[
[\group]_\hypprim = \{ [h_\ell] \}\,,
\]
as follows immediately from~\eqref{eq:Sconjhyp}.
Consequently, the Selberg zeta functions of~$X$ and~$Y$ differ from each other. Indeed, for~$\Rea s \gtrsim 1$ we have
\begin{align*}
Z_X(s) & = \prod_{k=0}^\infty \left( 1 - e^{-(s+k)\ell}\right)
\intertext{and}
Z_{Y}(s) & = \prod_{k=0}^\infty \left( 1 - e^{-(s+k)\ell}\right)^2\,.
\end{align*}
Another way to see this is to observe that $Y$ has two different primitive periodic geodesics, which are given by $t \mapsto (0, t)$ and $t \mapsto (0, -t)$ in coordinates $(r,\phi)$ of $Y \cong \R \times \R / 2\pi \Z$.
These two geodesics project down to the same geodesic on $X$, which bounces between the two orbifold points.

Due to uniqueness of meromorphic continuations, 
\begin{equation}\label{eq:ZXZY}
    Z_{Y}(s) = Z_X(s)^2
\end{equation}
for all~$s\in\C$. This relation can also be deduced taking advantage of the Venkov--Zograf formula
as explained in what follows. We denote the trivial character of~$\Lambda$ by~$\mathbf{1}$. Thus, 
\[
\mathbf{1}\colon \Lambda \to \C^\times\,,\quad g\mapsto 1\,.
\]
We set 
\[
\chi\coloneqq \Ind_\Lambda^\group \mathbf{1}\,,
\]
the (unitary) representation of~$\group$ induced by~$\mathbf{1}$. The Venkov--Zograf formula states that 
\[
Z_{X,\chi}(s) = Z_{Y,\mathbf{1}}(s) = Z_Y(s)\,.
\]
Since 
\[
\chi(h_\ell) = \mat{1}{0}{0}{1}\,,
\]
we obtain 
\begin{align*}
    Z_{Y}(s) & = Z_{X,\chi(s)} = \prod_{k=0}^\infty\det\left( 1 - \chi(h_\ell) e^{-(s+k)\ell} \right) = \prod_{k=0}^\infty \left( 1 - e^{-(s+k)\ell} \right)^2
\\
& = Z_X(s)^2
\end{align*}
for all~$s\in\C$. Throughout this section we omit the trivial character from the notation.

\begin{prop}\label{prop:relationWeierstrass}
    There exists a polynomial~$q$ of degree at most~$2$ such that
    \begin{align}\label{eq:weierstrass_example}
        \mc P_{Y}(s) = e^{q(s)} \gammafunc\left(\frac{s}{2}\right)^2 \gammafunc\left(\frac{s+1}{2}\right)^{-2} \mc P_X(s)^2
    \end{align}
    on all of~$\C$.
\end{prop}

\begin{proof}
We note that
\begin{align*}
\eulertop(Y) & = 0\,, & \eulertop(X) & = 1\,,
\intertext{or, equivalently,}
\eulerorb(Y) & = 0\,, & \eulerorb(X) & = 0\,.
\end{align*}
Since $\Lambda$ does not contain elliptic elements, we have $G_{Y^\wedge} \equiv 1$. Therefore, as consistent with Theorem~\ref{thm:BJP},
\begin{equation}\label{eq:ZYspec}
    Z_{Y}(s) = e^{q_Y(s)} \mc P_{Y}(s)
\end{equation}
for a suitable polynomial~$q_Y$ of degree at most~$2$. Further, the group~$\group$ contains two non-conjugate elliptic elements, namely $S$ and $S h_\ell$ with fixed points $i$ and $i e^{\ell/2}$, respectively. Both elliptic elements have order $2$. Therefore
\begin{align*}
    G_{X^\wedge}(s) &= G_{2,0}(s)^2 \\
    &= G_\infty(s)^{2\left(1-\frac12\right)} \gammafunc\left(\frac{s}2\right)^{2 \left(-\frac12 + N_2(0,0)\right) } \gammafunc\left( \frac{s+1}2 \right)^{ 2 \left( -\frac32 + N_2(0,1)\right) } \\
    &= G_\infty(s) \gammafunc\left( \frac{s}2 \right) \gammafunc\left( \frac{s+1}2 \right)^{-1}\,.
\end{align*}
Thus, by Theorem~\ref{thm:selberg-factorization}, 
\begin{equation}\label{eq:ZXspec}
Z_X(s) = e^{q_X(s)} \gammafunc\left(\frac{s}{2}\right) \gammafunc\left(\frac{s+1}{2}\right)^{-1} \mc P_X(s)
\end{equation}
for all~$s\in\C$, where $q_X$ is a suitable polynomial of degree at most~$2$. Combining~\eqref{eq:ZXZY}, \eqref{eq:ZYspec} and~\eqref{eq:ZXspec} yields
\[
\mc P_{Y}(s) = e^{q(s)} \gammafunc\left(\frac{s}{2}\right)^2 \gammafunc\left(\frac{s+1}{2}\right)^{-2} \mc P_X(s)^2
\]
on all of~$\C$ for a suitable polynomial~$q$ of degree at most~$2$.
\end{proof}

Proposition~\ref{prop:relationWeierstrass} in combination with Section~\ref{sec:hypcyl} allows us to determine the resonances of~$X$ as follows: In Section~\ref{sec:hypcyl} we saw that the resonance set of~$Y$ and hence the zero set of~$\mc P_Y$ is
\[
-\N_0 + \frac{2\pi i}{\ell}\Z\,,
\]
each resonances respectively zero with multiplicity~$1$. Using that
\[
\left(\frac{\gammafunc\left(\frac{s}2\right)}{\gammafunc\left(\frac{s+1}2\right)}\right)^2
\]
has a pole of order~$2$ at each $s\in-2\N_0$ and a zero of order~$2$ at each $s\in-2\N_0-1$, Proposition~\ref{prop:relationWeierstrass} implies that the zero set of~$\mc P_X$ and hence the resonance set of~$X$ is
\[
-2\N_0 \cup \left(-\N_0 + \frac{2\pi i}{\ell}\Z\setminus\{0\} \right)\,,
\]
where the zeros respectively resonances in~$-2\N_0$ have multiplicity~$2$ and those in $-\N_0 + \frac{2\pi i}{\ell}\Z\setminus\{0\}$ have multiplicity~$1$.

\end{example}

\begin{example}
Let
\[
 T \coloneqq \bmat{1}{4}{0}{1} \quad\text{and}\quad S \coloneqq 
\bmat{0}{1}{-1}{0}\,.
\]
Let 
\[
\group \coloneqq \ang{T, S} \quad\text{and}\quad \Lambda \coloneqq \ang{T, STS}
\]
Then $\Lambda$ is torsion-free and $[\group : \Lambda] = 2$.
\begin{figure}
    \centering
    \begin{tikzpicture}[scale=1]
    \fill[lightgray] (-2,0) rectangle (2,3.1);
    \filldraw[white,draw=black] (1,0) arc (0:180:1);
    \draw[-latex](-2.1,0) -- (2.5,0);
    \draw[-latex](0,-0.1) -- (0,3.2);
    \draw (-2,0) -- (-2,3.1);
    \draw (2,0) -- (2,3.1);
    \end{tikzpicture}
    \caption{A fundamental domain of $X = \ang{T,S} \backslash \h$.}
    \label{fig:funddom}
\end{figure}
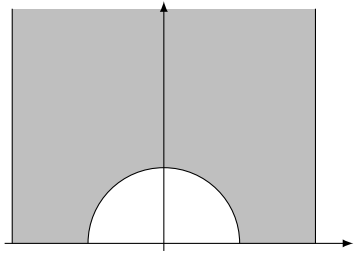
Denote by $X \coloneqq \group \backslash \h$ and $Y \coloneqq \Lambda\backslash\h$ the corresponding hyperbolic orbisurfaces. See Figure~\ref{fig:funddom} for a fundamental domain of~$X$.
We have
\begin{align*}
    \eulerorb(Y) &= \eulertop(Y) = -1\,,\\
    \eulerorb(X) &= -\frac12\,,\quad \eulertop(X) = 0\,.
\end{align*}
This can be argued by observing that $Y$ is of genus zero and has three ends (two cusp ends and one funnel) and $X$ is topologically a cylinder.
It is also possible to derive the orbifold Euler characteristic of $X$ from Lemma~\ref{lem:top_orb_euler}, from which then also the Euler characteristic of $Y$ follows immediately.

The periodic geodesics of~$X$ and~$Y$ differ, but each 
periodic geodesic of~$Y$ projects to a periodic geodesic 
of~$X$ of same length or with half the length.

Let $\eta \colon \Lambda\to \Unit(V)$ be any finite-dimensional representation of~$\Lambda$, and let
$\twist \coloneqq \Ind_\Lambda^\group\eta$. By the Venkov--Zograf formula (see~\cite[Section~7]{FP_szf})
\[
Z_{Y,\eta} = Z_{X,\twist}\,.
\]
From Theorem~\ref{thm:selberg-factorization} we obtain
\begin{align}\label{eq:PXYidentity}
    \mc{P}_{X,\twist}(s) &= e^{q(s)} \frac{G_\infty(s)^{\dim(V)}}{G_{X^\wedge,\twist}(s)} \gammafunc\left(s-\frac12\right)^{n_p(\eta) - n_p(\twist)} \mc{P}_{Y,\eta}(s)\,,
\end{align}
where $n_p(\eta)$ and $n_p(\twist)$ is the degree of singularity at cusps for $(Y,\eta)$ and $(X,\twist)$, respectively, and $q$ is a polynomial of degree at most $2$.
We have that
\begin{align*}
    n_p(\eta) - n_p(\twist) = m_{\eta(STS)}(1)\,.
\end{align*}
To evaluate $G_{X^\wedge,\twist}$, we note that $[\Gamma]_\ellprim = \{ [S]\}$ and $\EVP(\twist(S)) = \{0,1\}$. From the explicit description of $G_{X^\wedge,\twist}$ in Section~\ref{sec:defGXwedge} we obtain, on all of~$\C$,
\begin{align*}
G_{X^\wedge,\twist}(s) & = G_{2,0}(s)\thin G_{2,1}(s)
\\
& = G_\infty(s)^{\frac12}\thin\gammafunc\left(\frac{s}2\right)^{N_2(0,0)-\frac12}\thin\gammafunc\left(\frac{s+1}2\right)^{N_2(0,1)-\frac32}
\\
& \quad\cdot G_\infty(s)^{\frac12}\thin \gammafunc\left(\frac{s}2\right)^{N_2(1,0)-\frac12}\thin \gammafunc\left(\frac{s+1}2\right)^{N_2(1,1)-\frac32}
\\
& = G_\infty(s)\gammafunc\left(\frac{s}2\right)\,.
\end{align*}
In turn,
\[
\frac{G_\infty(s)^{\dim(V)}}{G_{X^\wedge,\twist}(s)} = G_\infty(s)^{\dim V - 1}\thin \frac{1}{\gammafunc\left(\frac{s}2\right)}
\]
is entire. Its zero set is contained in $-\N_0$. The poles of $\gammafunc\left(s-\tfrac12\right)^{n_p(\eta)-n_p(\twist)}$ are contained in $-\N_0+\frac12$; this function does not have zeros. Therefore we deduce from~\eqref{eq:PXYidentity} that the resonance sets of~$(X,\twist)$ and~$(Y,\eta)$ coincide including multiplicities except for resonances at points in~ $-\tfrac12\N_0$. The discrepancy at the latter elements can also be deduced from~\eqref{eq:PXYidentity} in the same way as in Example~\ref{EX:elementary}.
\end{example}

\section{Elements of spectral theory}\label{sec:spectral_manuscr}

With this section we initiate the proof of Theorem~\ref{thm:selberg-factorization} in full generality. As we established already in Section~\ref{sec:cyclic} the validity of Theorem~\ref{thm:selberg-factorization} for cyclic groups (and any unitary representation as twist), we will restrict most of the considerations in what follows to non-cyclic groups.

We require some results from spectral theory and scattering theory for our purposes, most of which we deliver already with \cite{DFP, DFP2}. In this section we will provide some additional results that are often direct consequences of~\cite{DFP, DFP2} or follow with an argumentation similar to the one in~\cite{DFP, DFP2}.

Throughout this section let $X = \group \backslash \h$ be a geometrically finite hyperbolic orbisurface with \emph{non-cyclic} fundamental group~$\group$.

\subsection{Boundary defining functions and measures}\label{sec:boundary}
We recall from Section~\ref{sec:decomposition} that $X$ admits a decomposition into a union of a compact core, a finite number of funnel ends and a finite number of cusp ends from Section~\ref{sec:decomposition}:
\begin{align*}
    X = X_{\interior} \sqcup X_f \sqcup X_c\,,
\end{align*}
where the funnel part~$X_f$ is a disjoint union of connected components~$X_{f,j}$ for $j\in\{1,\ldots, n_f\}$ such that $X_{f,j}$ is isometric to the model funnel~$F_{\ell_j}$ for some $\ell_j \in (0,\infty)$, and the cusp part~$X_c$ is a disjoint union of connected components~$X_{c,k}$ for $k\in\{1,\ldots, n_c\}$ such that $X_{c,k}$ is isometric to the model cusp~$C_\infty$. Thus,
\[
X_f = \bigsqcup_{j=1}^{n_f} X_{f,j}\quad\text{and}\quad X_c = \bigsqcup_{k=1}^{n_c} X_{c,k}\,.
\]
We now fix, as in~\cite[Section 3.2.4]{DFP}, a funnel boundary defining function~$\rho_f\in C^\infty(X, (0,\infty))$ such that, in geodesic polar coordinates for~$X$,
\begin{align*}
    \rho_f\vert_{X_{f,j}}(r,\phi) &= \frac{1}{\cosh(r)} &&\text{for each $j\in\{1,\ldots, n_f\}$}\,,
    \intertext{and}
    \rho_f\vert_{X_c} &\equiv 1\,.
\end{align*}
We also fix a cusp boundary definition function~$\rho_c \in C^\infty(X, (0,\infty))$ such that, in geodesic polar coordinates for~$X$,
\begin{align*}
 \rho_c\vert_{X_{c,k}}(r,\phi) &= e^{-r}
 \intertext{for each $k\in\{1,\ldots, n_c\}$, and}
 \rho_c\vert_{X_f} & \equiv 1\,.
\end{align*}
Further we set
\[
\rho \coloneqq \rho_f\rho_c\,,
\]
which is a boundary defining function $\rho \in C^\infty(X, (0,\infty))$ such that, in geodesic polar coordinates for~$X$,
\begin{align*}
    \rho\vert_{X_{f,j}}(r,\phi) &= \frac{1}{\cosh(r)} &&\text{for each $j\in\{1,\ldots, n_f\}$}\,,
    \\
    \rho\vert_{X_{c,k}}(r,\phi) &= e^{-r} &&\text{for each $k\in\{1,\ldots, n_c\}$}\,.
\end{align*}
For $\eps\geq 0$ we set
\[
X_\eps \coloneqq \{x \in X \setmid \rho(x) \geq \eps\}\,.
\]
For $\eps < \tfrac12$, the space~$X_\eps$ is independent of the specific choice of boundary defining function
and we define
\begin{align*}
X_{f,j,\eps} &\coloneqq X_{f,j} \cap X_\eps && \text{for $j \in \{1,\dotsc,n_f\}$}
\intertext{and}
X_{c,k,\eps} &\coloneqq X_{c,k} \cap X_\eps && \text{for $k \in \{1,\dotsc,n_c\}$\,.}
\end{align*}
We denote by $d\sigma_\eps$ the measure on $\pa X_\eps = \{x \in X \setmid \rho(x) = \eps\}$ that is induced by restricting the Riemannian metric of~$X$ (as a Riemannian orbifold) or of~$\h$ (including inheriting it to $X_\eps$).

We denote by~$d\mu_X$ the measure on~$X$ induced by the Riemannian metric.
On the boundary at infinity, $\pa_\infty X$, of $X$, we consider the induced measure $d\mu_\partial$. See~\cite{DFP} for details.

\subsection{Resolvent}
In~\cite{DFP} we have shown that the resolvent of the Laplacian, $\ResTwist$, admits a meromorphic continuation to all of~$\C$ as an operator $L^2_\cpt \to L^2_\loc$. More precisely, we obtained an explicit representation.

To state this representation, we fix $s_0 \in \C$ with $\Rea s_0>1$, and, for $\bullet\in\{f,c\}$, we let $R_{X_\bullet,\twist}$ denote the \emph{model resolvent} for the funnel part~$X_f$ (if $\bullet =f$) respectively the cusp part~$X_c$ (if $\bullet=c$). See~\cite{DFP}. Let $d_X$ denote the metric on~$X$ induced by the metric~$d_\h$ on~$\h$. For $r\in\{0,1,2\}$, we pick smooth cut-off functions $\eta_{\bullet,r}\in C^\infty(X)$ such that
\[
\eta_{\bullet,r}(z) =
\begin{cases}
1 & \text{if $d_X(X\setminus X_\bullet,z)<r$}\,,
\\
0 & \text{if $d_X(X\setminus X_\bullet,z)>r+\tfrac12$}\,.
\end{cases}
\]
Then $n_{f,r}$ vanishes up in the funnels, whereas $n_{c,r}$ vanishes up in the cusps. See \cite[Section~5]{DFP} for more details. We define, for all $s\in\C$, the parametrices (interior, funnel part, cusp part, respectively) by
\begin{align}
M_i &\coloneqq \eta_{f,2} \eta_{c,2} \ResTwist(s_0) \eta_{f,1} \eta_{c,1}\,, \nonumber
\\
M_f(s) &\coloneqq (1 - \eta_{f,0}) R_{X_f,\twist}(s) ( 1 - \eta_{f,1})\,, \nonumber
\\
M_c(s) &\coloneqq (1 - \eta_{c,0}) R_{X_c,\twist}(s) ( 1 - \eta_{c,1})\,, \nonumber
\intertext{and we set}
M(s) &\coloneqq M_i + M_f(s) + M_c(s)\,. \label{def:M}
\end{align}
Moreover we set, again for all $s\in\C$ and with $[\cdot,\cdot]$ denoting the commutator,
\begin{align*}
L_i(s) & \coloneqq -[\LapTwist, \eta_{f,2}\eta_{c,2}] \ResTwist(s_0) \eta_{f,1}\eta_{c,1}
\\
&\hphantom{\coloneqq - \LapTwist} + (s(1-s)- s_0(1-s_0)) M_i(s_0)\,,
\\
L_f(s) &\coloneqq [\LapTwist, \eta_{f,0}] R_{X_f,\twist}(s) (1 - \eta_{f,1})\,,
\\
L_c(s) &\coloneqq [\LapTwist, \eta_{c,0}] R_{X_c,\twist}(s) (1 - \eta_{c,1})
\intertext{and}
L(s) &\coloneqq L_i(s) + L_f(s) + L_c(s)\,.
\end{align*}
With the operators~$M(s)$ and $L(s)$ we obtain the explicit representation
\begin{align*}
    \ResTwist(s) &= M(s) (\id - L(s))^{-1}
\end{align*}
on all of~$\C$.

We require the following estimate of the Hilbert--Schmidt norm for the remainder operator~$L(s)$ weighted by the total boundary defining function~$\rho$.
\begin{lemma}\label{lem:bound-Ls}
    For any $\eps > 0$,  on $\left\{s\in\C : \abs{\Rea s - \tfrac12} \leq \eps\right\}$, we have the estimate
    \begin{align*}
        \norm{L(s)\rho}_{\HS} \lesssim \ang{s}^2\,,
    \end{align*}
    where $\ang{s} = \sqrt{1+|s|^2}$ denotes the Japanese bracket of~$s$.
    \end{lemma}
\begin{proof}
    The operator~$L_i(s)$ is polynomial of degree~$2$ in~$s$ and has smooth compactly support integral kernel. Hence it is Hilbert--Schmidt.

    To treat the funnel term~$L_f(s)$ and the cusp term~$L_c(s)$, we adopt the argumentation from~\cite[Lemma 3]{Zworski89} (see also \cite{DFP}).
For all $s\in\C$ we have
\begin{align}\label{eq:stubest}
[\LapTwist, \eta_{f,0}] R_{X_f,\twist}(s) (1-\eta_{f,2})\rho &=
(\Delta_X \eta_{f,0}) R_{X_f,\twist}(s) (1-\eta_{f,2})\rho \\
&\phantom{==} + \ang{\nabla_X \eta_{f,0}, \nabla_{X,\twist} R_{X_f,\twist}(s) (1-\eta_{f,2})\rho}\,. \nonumber
\end{align}
To bound the first term on the right hand side of~\eqref{eq:stubest} we note that for any $\psi_0 \in C^\infty_c(X)$ with $\supp \psi_0 \cap \supp (1 - \eta_{f,2}) = \emptyset$ we obtain the bound
    \begin{align}\label{eq:hs-bound}
        \norm{\psi_0 R_{X_f,\twist}(s) (1-\eta_{f,2})\rho}_{\HS} \lesssim \ang{s}^2
    \end{align}
    on $\{s\in\C : \abs{\Rea s - \tfrac12} < \eps\}$ analogously to~\cite[Lemma 6.12]{DFP}.

To handle the second term on the right hand side of~\eqref{eq:stubest}, we set
\[
R(s) \coloneqq R_{X_f,\twist}(s) (1-\eta_{f,2})\rho\,.
\]
We pick an orthonormal basis $\{e_j:j=1,\ldots, \dim V\}$ of $V$ such that $R(s) e_j = R_j(s) e_j$, we choose a function $\psi_0 \in C^\infty_c(X)$ such that the support of every component of $\nabla_X \eta_{f,0}$ is contained in the set $\{\psi_0 = 1\}$, and we set
\[
d\mu_{X\times X}(z,z') \coloneqq d\mu_X(z)\,d\mu_X(z') \qquad\text{for all $z,z'\in X$}\,.
\]
Using
    \begin{align}\label{eq:estfunnel}
        \ang{\nabla_X \eta_{f,0}, \nabla_{X,\twist} R_{X_f,\twist}(s) (1-\eta_{f,2})\rho}
        &= \sum_{j=1}^{\dim V} \ang{\nabla_X \tilde\psi_0, \nabla_X R_j(s)} e_j
    \end{align}
we can estimate the $L^2$-norm of the left hand side of~\eqref{eq:estfunnel} by
    \begin{align*}
        &\norm*{ \ang{\nabla_X \eta_{f,0}, \nabla_{X,\twist} R_{X_f,\twist}(s) (1-\eta_{f,2})\rho} }_{\HS}^2 \\
        &\qquad= \int_{X\times X} \abs*{ \ang{\nabla_X \eta_{f,0}(z), \nabla_{X,\twist} R(s;z,z')} }_2^2 \,d\mu_{X\times X}(z,z') \\
        &\qquad\leq \sum_{j=1}^{\dim V} \int_{X\times X} \abs*{ \ang{\nabla_X \eta_{f,0}(z), \nabla_X R_j(s;z,z')} }^2 \,d\mu_{X\times X}(z,z') \\
        &\qquad\leq C \sum_{j=1}^{\dim V} \int_{X\times X} \abs*{ \psi_0(z) \nabla_X R_j(s;z,z')}^2 \,d\mu_{X\times X}(z,z')
    \end{align*}
for a suitable constant $C>0$.
Integration by parts now yields
    \begin{align}\label{eq:estfunnel2}
        &\int_{X\times X} \abs*{ \psi_0(z) \nabla_X R_j(s,z,z')}^2 \,d\mu_{X}(z)\,d\mu_{X}(z')
        \\
        &\qquad = - \int_{X\times X} \psi_0(z)^2 R_j(s;z,z') \Delta_X R_j(s;z,z')\,d\mu_{X \times X}(z,z') \nonumber
        \\
        &\qquad\phantom{=}\ - 2 \int_{X\times X} \ang{\nabla_X \psi_0(z), \nabla_X R_j(s;z,z')} \psi_0(z) R_j(s;z,z') \,d\mu_{X\times X}(z,z')\,. \nonumber
    \end{align}
    We note that $R_j(s)$ solves the equation $(\Delta_X - s(1-s))R_j(s) = \Pi_j$,
    where $\Pi_j$ is the projection onto the $1$-dimensional subspace of $V$ spanned by $e_j$.
    The second term of~\eqref{eq:estfunnel2} can be estimated by
    \begin{align*}
        &\abs*{\int_{X\times X} \ang{\nabla_X \psi_0(z), \nabla_X R_j(s;z,z')}\psi_0(z) R_j(s;z,z') \,d\mu_{X \times X}(z,z') } \\
        &\leq \int_{X\times X} \abs{ \nabla_X \psi_0(z) } \abs{ \nabla_X R_j(s;z,z') } \abs{\psi_0(z)} \abs{R_j(s;z,z')} \,d\mu_{X\times X}(z,z') \\
        &\leq \frac{1}{2\eta} \int_{X\times X} \abs{ R_j(s;z,z') \nabla_X \psi_0(z) }^2 \,d\mu_{X\times X}(z,z') \\
        &\phantom{\leq} + \frac{\eta}{2} \int_{X\times X} \abs{ \psi_0(z) \nabla_X R_j(s;z,z') }^2 \,d\mu_{X\times X}(z,z')
    \end{align*}
    for any $\eta > 0$ by Young's inequality. Taking $\eta = 2$, we conclude
    \begin{align*}
        &\int_{X\times X} \abs*{ \psi_0(z) \nabla_X R_j(s;z,z')}^2 \,d\mu_{X\times X}(z,z') \\
        &\leq C \int_{X\times X} \abs{ R_j(s;z,z') \nabla_X \psi_0(z) }^2 \,d\mu_{X\times X}(z,z')\,.
    \end{align*}
        Using~\cite[Lemma 6.8]{DFP} we obtain a similar result for the cusp situation. This yields the claimed estimate.
\end{proof}

\subsection{Resonance at \texorpdfstring{$s = 1/2$}{s = 1/2}}
\begin{lemma}\label{lem:resonance-onehalf}
    For $s \in \C$ near $\tfrac12$ the resolvent~$\ResTwist$ of~$\LapTwist$ satisfies
    \begin{align}\label{eq:resolvent-onehalf}
        \ResTwist(s) = \frac{1}{2s-1} \sum_{k=1}^{m_{X,\twist}\left(\tfrac12\right)} \phi_k(s) \ang{\phi_k(s), \cdot} + H(s)\,,
    \end{align}
    where $H$ is holomorphic near $\tfrac12$
    and, for each $k\in\{1,\ldots, m_{X,\twist}\left(\tfrac12\right)\}$,
    \begin{align}\label{eq:phik_asymptotics}
        \phi_k(s) \in \rho_f^s\rho_c^{s-1}C^\infty(X,\bundle)
    \end{align}
    and satisfies
    \begin{align}\label{eq:phik_equation}
        \left(\LapTwist - \frac14\right)\phi_k\left(\frac12\right) = 0\,.
    \end{align}
    The family $\{\phi_k(s):k =1,\ldots,m_{X,\twist}\left(\tfrac12\right)\}$ is linearly independent.
\end{lemma}
\begin{proof}
    \cite[Lemma 5.9]{DFP2} guarantees the existence of $\phi_k$, $k\in\{1,\ldots, m_{X,\twist}\left(\tfrac12\right)\}$, such that $\phi_k \in \rho_f^{\tfrac12} \rho_c^{-\tfrac12}C^\infty(X, \bundle)$ and that satisfy \eqref{eq:resolvent-onehalf} and \eqref{eq:phik_equation} independent of~$s$.
    Setting $\phi_k(s) \coloneqq \rho^{s-\tfrac12} \phi_k$ for $k\in \{1,\ldots, m_{X,\twist}\left(\tfrac12\right)\}$, we have $\phi_k(s) \in \rho_f^s \rho_c^{s-1}C^\infty(X, \bundle)$. Moreover, $\phi_k(\tfrac12) = \phi_k$ solves \eqref{eq:phik_equation}.
    Since $(1 - \rho^{s-\tfrac12}) \phi_k$ is holomorphic near $s = \tfrac12$ and vanishes at $s = \tfrac12$, we find a function~$h_k$ that is holomorphic near $\tfrac12$ and satisfies
    \[
    (2s - 1) h_k(s) = (1 - \rho^{s-\tfrac12}) \phi_k = \phi_k - \phi_k(s)
    \]
    near $s = \tfrac12$. From this, the statement of the lemma follows immediately.
\end{proof}

For $s\in\C$, $s\notin \mc R_{X,\twist}\cup \tfrac12\Z$, we let
\[
S_{X,\twist}\colon C^\infty(\partial_\infty X, E_\twist\vert_{\partial_\infty X}) \to C^\infty(\partial_\infty X, E_\twist\vert_{\partial_\infty X})
\]
denote the scattering matrix for~$(X,\twist)$. It is defined analogously to the untwisted case by means of the twisted Poisson operator. In particular, for each $\psi\in C^\infty(\partial_\infty, E_\twist\vert_{\partial_\infty X})$ we find $u\in C^\infty(X,E_\twist)$ such that $u$ is an eigenfunction of~$\LapTwist$ with eigenvalue $s(1-s)$ and
\[
(2s-1)u \sim \rho_f^{1-s}\rho_c^{-s}\psi + \rho_f^s\rho_c^{s-1}S_{X,\twist}(s)\psi \quad\text{as $\rho_f\rho_c\to0$}\,.
\]
We refer to~\cite[Section~5]{DFP2} for details.

\begin{lemma}
For $k\in\{1,\ldots, m_{X,\twist}\left(\tfrac12\right)\}$ and $s\in\C$ let $\phi_k(s)$ be as in Lemma~\ref{lem:resonance-onehalf} and set $\phi_k^\#(s) \coloneqq \rho_f^{-s} \rho_c^{1-s} \phi_k(s)|_{\pa_\infty X}$. Then
    \begin{align}\label{eq:smatrix-eigenfunctions}
        S_{X,\twist}\left(\tfrac12\right) = -\id + \sum_{k=1}^{m_{X,\twist}\left(\tfrac12\right)} \phi_k^\#\left(\tfrac12\right) \ang{\phi_k^\#\left(\tfrac12\right), \cdot}\,.
    \end{align}
\end{lemma}
\begin{proof}
    This follows directly from \cite[Remark 5.10]{DFP2} and the fact that $\phi_k^\#\left(\tfrac12\right)$ agrees with $\phi_k^\#$ as defined in \cite{DFP2}.
\end{proof}

\section{Regularized traces}\label{sec:trace}

The resolvent of the Laplacian is not trace class if $X$ is not compact. Nevertheless we can define a regularized trace
of the resolvent operator.
In this section we will calculate this regularized trace in two different ways to relate it to the logarithmic derivative of the Selberg zeta function on the one hand and to the scattering determinant on the other hand.

More precisely, we will introduce two types of regularized traces of the resolvent: an algebraic
regularization and a scattering theoretic regularization, which are closely related to each other.
The algebraic regularization can be calculated in terms of the Selberg zeta
function and special functions.
The scattering theoretic regularized trace will be expressed in terms of the logarithmic derivative of the scattering determinant
and the regularized trace of the model funnel ends.

Throughout this section let $\group$ be a finitely generated, non-finite Fuchsian group and let $X=\group\backslash\h$ denote the associated hyperbolic orbisurface.

\subsection{Zero-trace}\label{sec:zero-trace}

Let $f \in C^\infty(X)$ be a function polyhomogeneous in~$\rho$. Then we define the
\emph{0-integral} of $f$ by
\begin{align*}
    \zInt{X} f\,d\mu_X \coloneqq \FP_{\eps \to 0} \int_{X_\eps} f \, d\mu_X\,,
\end{align*}
where $\FP$ is the \emph{Hadamard finite part}.
The \emph{zero-volume}~$\zVol$ of~$X$ is given by
\begin{align*}
    \zVol(X) \coloneqq \zInt{X} \, d\mu_X\,.
\end{align*}
For $A \in C^\infty(X \times X)$ we set
\begin{align*}
    \zTr(A) \coloneqq \zInt{X} A(z,z) \, d\mu_X(z)\,,
\end{align*}
the \emph{zero-trace} of~$A$, provided that $A(z,z)$ admits a polyhomogeneous expansion in~$\rho$.

The following proposition provides a relation between the zero-volume~$\zVol(X)$ of~$X$, its topological Euler characteristic~$\eulertop(X)$ and the isotropy groups at the orbifold points (conical singularities) of~$X$.

\begin{prop}\label{prop:euler-characteristic}
    Let $z_1,\ldots, z_m$ be the conical singularities of $X = \group \backslash \h$. (Here, $m\in\N_0$. If $X$ has no orbifold points, then $m=0$.) For each $j\in\{1,\ldots, m\}$ let $q_j$ be the order of the isotropy group at~$z_j$.
    Then
    \begin{align}\label{eq:zvol_euler}
        \zVol(X) = -2\pi \left( \eulertop(X) - \sum_{j=1}^m \left(1 - \frac1{q_j}\right)\right)\,.
    \end{align}
\end{prop}

\begin{proof}
    Selberg's Lemma guarantees the existence of a subgroup~$\tilde{\group}$ of~$\group$ of finite index such that $\tilde{X} \coloneqq \tilde{\group} \backslash \h$ has no elliptic elements. Let $q\coloneqq [\group:\tilde{\group}]$ and set $\tilde X \coloneqq \tilde\group\backslash\h$.
    By \cite[Lemma~10.3]{Borthwick_book},
    \begin{align}\label{eq:gaussbonnet-mfd}
        \zVol(\tilde{X}) = -2\pi \eulertop(\tilde{X})\,.
    \end{align}
    As
    \[
    \zVol(X) = \frac1q\zVol(\tilde{X})
    \]
we obtain
    \begin{align*}
        \zVol(X) = -\frac{2\pi}{q} \eulertop(\tilde{X})\,.
    \end{align*}
    Together with Lemma~\ref{lem:top_orb_euler} this proves the claim.
\end{proof}

\subsection{Regularizing the trace of the resolvent}
We provide two regularizations of the kernel of the resolvent at the
diagonal.
The first regularization, the \emph{algebraic regularization}, is natural from the perspective of
averaging the resolvent $R_\h$ over the group elements:
\begin{align*}
    \varphi_{X,\twist}(s) &\in C^\infty(X)\,,\\
    \varphi_{X,\twist}(s,z) &= (2s-1) \tr_{\bundle}\left(\ResTwist(s;z,w) -
R_\h(s;z,w)\right)\big|_{w=z}\,.
\end{align*}
Here, we implicitly lift $\ResTwist(s;z,w)$ to $\h \times \h$, subtract
$R_\h$ there and then project back to $X \times X$. Moreover we use the canonical isomorphism
\[
(\bundle \boxtimes \bundle')|_{\diag
\subset X \times X} \cong \bundle \otimes \bundle'\,,
\]
where $\bundle \boxtimes \bundle'$ denotes the exterior tensor product of~$\bundle$ and~$\bundle'$. Further we let
\[
\tr_{\bundle} \colon C^\infty(X, \bundle \otimes \bundle') \to C^\infty(X)
\]
denote the usual trace operator. From the resolvent construction~\cite[Theorem~4.1]{DFP2} we obtain
\begin{align*}
    \varphi_{X,\twist}(s)|_{X_f} \in \rho^{2s}C^\infty(X_f)
\end{align*}
and
\begin{align*}
    \varphi_{X,\twist}(s)|_{X_c} \in \rho^{2s-2}C^\infty(X_c) + \rho^{-1} C^\infty(X_c)\,,
\end{align*}
where the first term comes from the term $Q(s)$ in \cite[Theorem~4.1]{DFP2} and the second comes from $0$-th Fourier modes in the model resolvent.
This implies that
\begin{align*}
    \int_{X_\eps} \varphi_{X,\twist}(s;z) \,d\mu_X(z) \sim
    a(s) + b_0(s) \log(\eps) + \sum_{m=1}^\infty c_m(s) \eps^{2s-1+m} + O(\eps)
\end{align*}
as $\eps \to 0$, for suitable functions~$a$, $b_0$, and $c_m$, $m\in\N$.
For $s\in\C$, $s \notin \ResSet_{X,\twist} \cup \tfrac12\Z$, we define
\begin{align}\label{eq:defPhi}
    \Phi_{X,\twist}(s) &\coloneqq \zInt{X} \varphi_{X,\twist}(s,z) d\mu_X(z) = a(s)\,,
\end{align}
the \emph{algebraically regularized resolvent trace}.

The second regularization, the \emph{scattering theoretic regularization}, uses the fact that $\LapTwist - s(1-s)$ is invariant under $s \mapsto 1-s$.
We define
\begin{align}\label{eq:def_Upsilon}
    \Upsilon_{X,\twist}(s) \coloneqq (2s-1) \zTr \left( \ResTwist(s) - \ResTwist(1-s)\right),
\end{align}
the \emph{scattering theoretically regularized resolvent trace}, where the $0$-trace is taken on the vector bundle $\bundle \to X$.

Analogously to the result in~\cite{BJP} for the untwisted situation, these two regularizations are closely
related, as indicated by the following lemma.

\begin{lemma}\label{lem:res-trace}
    For \[s \not \in \tfrac12\Z \cup \ResSet_{X,\twist} \cup (1 - \ResSet_{X,\twist})\] we have
    \begin{equation}\label{eq:Phi_Ups}
        \begin{aligned}
            \Phi_{X,\twist}(s) + \Phi_{X,\twist}(1-s) &= \Upsilon_{X,\twist}(s)
            - \left(s-\frac12\right)\cot(\pi s) \thin \dim(V)\thin\zVol(X)\,.
        \end{aligned}
    \end{equation}
\end{lemma}

\begin{proof}
    The proof is the same as in \cite{BJP} with the modification that the
constant is multiplied by $n = \dim(V)$ since
    \begin{align*}
        \zTr \left(\id_V R_\h(s) - \id_V R_\h(1-s)\right) = \dim(V) \cdot \zTr \left(R_\h(s) - R_\h(1-s)\right)\,.
    \end{align*}
\end{proof}

The following lemma provides analytic and spectral properties of~$\Phi_{X,\twist}$, in particular its relation to the discrete spectrum, $\sigma_d(\LapTwist)$, of~$\LapTwist$.

\begin{lemma}\label{lem:Phi_cont}
    \begin{enumerate}[label=$\mathrm{(\roman*)}$, ref=$\mathrm{\roman*}$]
    \item The algebraic regularized resolvent trace~$\Phi_{X,\twist}$ from~\eqref{eq:defPhi} admits a meromorphic continuation to all of~$\C$.
    \item Let $s_0\in \left\{s\in\C : \Rea s > \tfrac12\right\}$. Then $s_0$ is a pole of $\Phi_{X,\twist}$ if and only if $s_0(1-s_0) \in \sigma_d(\LapTwist)$ and
    the pole is simple and its residue is $m_{X,\twist}(s_0)$.
    \item On the critical strip $\left\{ \Rea s = \tfrac12 \right\}$, the function~$\Phi_{X,\twist}$ has a single pole at $s_0 = \tfrac12$. It is simple with residue $m_{X,\twist}\left(\tfrac12\right) - \tfrac12 n_p$.
    \end{enumerate}
\end{lemma}
\begin{proof}
    By definition, $\Phi_{X,\twist}(s) = a(s)$ for $s \not \in \ResTwist \cup \tfrac12(1 - \N_0)$. Since $\ResTwist$ is meromorphic, it follows that
    $a$ is meromorphic as well, which shows that $\Phi_{X,\twist}$ admits a meromorphic continuation to $s \in \C$.

    Recall that for $\Rea s > \tfrac12$,
    \[
    \Phi_{X,\twist}(s) = \zInt{X} \varphi(s;z)\,d\mu_X(z)\,,
    \]
    where
    \begin{align*}
        \varphi_{X,\twist}(s;z) = (2s-1) \tr_{\bundle}\left( \ResTwist(s;z,w) - \id_V R_\h(s;z,w)\right)|_{w = z}\,.
    \end{align*}
    We consider $\varphi_{X,\twist}(s)$ here as a function on~$X = \group \backslash \h$.
    In view of the structure of the resolvent as shown in~\cite[Theorem~4.1]{DFP2} we note that the subtraction of $R_\h(s;z,w)$ only affects the interior term and the model terms,
    which are singular on the diagonal.

    The model terms are holomorphic, hence poles come from the remainder $Q(s)$ (see \cite[Theorem~4.1]{DFP2}) and are only possible if $s(1-s)$ is an eigenvalue of $\LapTwist$.
    Let $s_0(1 - s_0) \in \sigma_d(\LapTwist)$. From the structure of the
    resolvent near a resonance \cite[Section~4.1]{DFP2} (see also the proof of \cite[Proposition 5.13]{DFP2}) we obtain
    \begin{align*}
        \varphi_{X,\twist}(s,z) = \frac{1}{s - s_0}
        \sum_{k=1}^{m_{X,\twist}(s_0)} \ang{\phi_k(z), \phi_k(z)} + (2s-1)
        \tilde{\varphi}(s,z)\,,
    \end{align*}
    where $\{\phi_k\}$ is an orthonormal basis of the eigenspace of $s_0(1-s_0)$
    and $\tilde{\varphi}(s,z)$ is holomorphic in $s$ near $s = s_0$.
    By the resolvent structure we have \[\tilde{\varphi}(s,z)|_{X_f} \in \rho_f^{2s}C^\infty(\overline{X}_f)\] and hence
    \begin{align*}
        h(s) \coloneqq \zInt{X} \tilde{\varphi}(s,z)\,d\mu_X(z)
    \end{align*}
    is holomorphic near $s = s_0$. Here, $\overline{X}_f$ denotes the geodesic compactification of~$X_f$ (see~\cite[Section~3.3]{DFP} or~\cite[Section~3.1]{DFP2} for details).
    We obtain that
    \begin{align*}
        \Phi_{X,\twist}(s) = \frac{m_{X,\twist}(s_0)}{s - s_0} + (2s-1)h(s)\,.
    \end{align*}
    Similarly, the set
    \[
    \left\{s_0\in\C\setminus\left\{\tfrac12\right\}: \Rea s_0 = \tfrac12\right\}
    \]
    is free of resonances by
    \cite[Corollary 6.7]{DFP}, and the zero-integral of the remainder term for points in this set is
    holomorphic.

    Let $s \in \C$ be sufficiently close to~$\tfrac12$. By Lemma~\ref{lem:resonance-onehalf},
    \begin{align*}
        \varphi_{X,\twist}(s,z) = \sum_{k=1}^{m_{X,\twist}\left(\tfrac12\right)} \ang{\phi_k(s,z), \phi_k(s,z)} + (2s-1) H(s,z)\,,
    \end{align*}
    where the functions
    $\phi_k(s, \cdot) \in \rho_f^s \rho_c^{s-1} C^\infty(\overline{X})$ for $k\in\{1,\ldots,m_{X,\twist}\left(\tfrac12\right)\}$  and the function $H(s,\cdot) \in \rho_f^{2s}\rho_c^{2s-2} C^\infty(\overline{X})$ are defined on the geodesic compactification of~$X$. The functions $\psi_k(s,\cdot)$ linearly independent, are holomorphic in $s$ and satisfy
    \begin{align*}
        \left(\LapTwist - \frac14\right)\phi_k\left(\frac12\right) &= 0\,.
    \end{align*}
    The function $H(s,\cdot) \in \rho_f^{2s}\rho_c^{2s-2} C^\infty(\overline{X})$ is holomorphic near $s = \tfrac12$.
    We conclude that
    \begin{align*}
        \Phi_{X,\twist}(s,z)
        &= \sum_{j = 0}^{n_f} \FP_{\eps\to 0} \int_{X_{f,j, \eps}} \varphi_{X,\twist}(s;z) \,d\mu_X(z)\\
        &\phantom{=} + \sum_{j = 0}^{n_c} \FP_{\eps\to 0} \int_{X_{c,j,\eps}} \varphi_{X,\twist}(s;z) \,d\mu_X(z)\\
        &\phantom{=} + \int_{X_{\interior}} \varphi_{X,\twist}(s;z) \,d\mu_X(z)\,.
    \end{align*}
    As the interior term does not contribute to the residue, we only have to consider the contributions of the funnel ends and cusp ends.

    We first consider the contribution of the funnel ends.
    Setting
    \[
    \phi_k^\# \coloneqq \rho_f^{-\frac12}\rho_c^{\frac12}\phi_k\vert_{\partial_\infty X}
    \]
    for $k\in\left\{1,\ldots, m_{X,\twist}\left(\tfrac12\right)\right\}$, we obtain, for $s \not = \tfrac12$ sufficiently close to~$\tfrac12$,
    \begin{align*}
        \int_{X_{f,j, \eps}} \abs{\phi_k(s,z)}^2 \,d\mu_X(z) &=
        \frac{\eps^{2s-1}}{2s-1} \int_{\pa_\infty X_{f,j}} \abs{\phi_k^\#(s, \theta)}^2 \,d\mu_\pa(\theta) + h_k(s,\eps)\,,
          \end{align*}
    where $h_k$ is a suitable functions that is holomorphic in~$s$ near $\tfrac12$.

    Therefore the residue of $\FP_{\eps\to 0} \int_{X_{f,j,\eps}} \varphi_{X,\twist}(s;z) \,d\mu_X(z)$ at $s = \tfrac12$ is given by
    \begin{align*}
        \frac{1}{2} \int_{\pa_\infty X_{f,j}} \abs{\phi_k^\#(\tfrac12,z)}^2 \,d\mu_{\pa X}(z) &= \frac{1}{2} \norm{\phi_k^\#\left(\tfrac12\right)}^2_{L^2(\pa_\infty X_{f,j}, \bundle)} \\
        &= \frac{1}{2} \Tr\left( S_{j j}^{\ff}\left(\tfrac12\right) + \id \right)\,,
    \end{align*}
    where we used \eqref{eq:smatrix-eigenfunctions} and where $S_{j j}^{\ff}$ is the funnel-funnel component for $X_{f,j}$ of the scattering matrix. See \cite{DFP2} for details.

    For understanding the contribution of the cusp ends let $z \in X_{c,j}$.
    We use the decomposition $Q(s) = (2s-1)^{-1} \tilde{Q}(s) + Q_\hol(s)$ for the remainder term, where $Q^\#$ and $Q_\hol$ are holomorphic near $s = \tfrac12$
    and $\tilde{Q}(s) \in (\rho_c \rho_c')^{s-1} C^\infty(\overline{X}_{c,j})$.

    By the decomposition of the resolvent we have
    \begin{align*}
        \varphi_{X,\twist}(s;z) &= (2s-1) \tr_{\bundle} \left( M_c(s;z,w) - R_{\h}(s;z,w) + Q(s) \right)|_{w=z}\\
        &= (1 - \eta_{c,j,0}) \varphi_{C_\infty}(s;z) n_p + (2s-1)\tr_{\bundle} Q(s;z,z)\,.
    \end{align*}
    By \cite[(56)]{DFP2}
    \begin{align*}
        (2s-1) Q(s;z,z) &= \rho^{2s-2} S_{jj}^{\cc}(s) + O(\rho^{2s-1})\,,
    \end{align*}
    where $S_{jj}^{\cc}$ is the cusp-cusp component for~$X_{c,j}$ of the scattering matrix.
    Therefore
    \begin{align*}
        \int_{X_{c,j, \eps}} (2s-1) Q(s;z,z) \,d\mu_X(z) = \frac{\eps^{2s-1}}{2s-1} S_{jj}^{\cc}(s) + h_j(s, \eps)\,,
    \end{align*}
    where $h_j$ is holomorphic in $s$ near $s = \tfrac12$. Thus,
    \begin{align*}
        \zInt{X_{c,j}} (2s-1) Q(s;z,z) \,d\mu_X(z) = \frac{1}{2} S_{jj}^{\cc}\left(\frac12\right)\,.
    \end{align*}
    By the same reasoning as in~\cite[(10.52)]{Borthwick_book} we see that the term $(1 - \eta_{c,j,0}) \varphi_{C_\infty}(s;z)$ does not contribute to the $0$-integral.
    Thus
    \begin{align*}
        \res_{s=\tfrac12} \Phi_{X,\twist}(s) &= \frac{1}{2} \Tr\left( S_{X,\twist}\left(\frac12\right) + \id \right) - \frac{n_p}{2} \\
        &= m_{X,\twist}\left(\frac12\right) - \frac{n_p}{2}\,.
    \end{align*}
    This completes the proof.
\end{proof}

\subsection{Relation between Selberg zeta function and regularized resolvent for some model cases}
\label{sec:relationmodels}

In the subsequent section we will establish a relation between the Selberg zeta function~$Z_{X,\twist}$ for~$(X,\twist)$ and the scattering theoretically regularized resolvent~$\Upsilon_{X,\twist}$ from~\eqref{eq:def_Upsilon}. As auxiliary results we will discuss in this section such a relation for some model cases.

\subsubsection{Hyperbolic cylinder}\label{sec:relhypcyl}
We now investigate how the Selberg zeta function is related to the algebraically regularized resolvent trace for the hyperbolic cylinder. Let $\ell>0$ and recall the hyperbolic cylinder~$C_\ell$ and the generating element~$g_\ell$ for its fundamental group from Section~\ref{sec:hypcyl}.
Recall further from Section~\ref{sec:hypcyl} that the Selberg zeta function for~$C_\ell$ and the twist $\twist\colon\ang{g_\ell}\to\Unit(V)$ is given by
\begin{align*}
    Z_{C_\ell,\twist}(s) = \prod_{\kappa \in \{\pm 1\}}\prod_{j=0}^\infty \det\left(1 - e^{-(s+j)\ell} \twist(g_\ell^\kappa)\right)
\end{align*}
for $\Rea s > \tfrac12$.
The algebraically regularized trace of the resolvent is given by
\begin{align*}
    \Phi_{C_\ell,\twist}(s) = (2s-1)\zInt{C_\ell} \sum_{k\in \Z\setminus\{0\}} \tr\twist(g_\ell^k) R_\h(s;z,e^{|k|\ell}z) d\mu_{C_\ell}(z).
\end{align*}
The following proposition extends a result from~\cite{Patterson_zeta}.
\begin{prop}\label{prop:Cell-zeta-deriv}
    The Selberg zeta function and the algebraically regularized resolvent trace for~$(C_\ell,\twist)$ satisfy the relation
    \begin{align*}
        \frac{Z_{C_\ell,\twist}'}{Z_{C_\ell,\twist}} = \Phi_{C_\ell,\twist}\,.
    \end{align*}
\end{prop}
\begin{proof}
Since the functions on both sides of the claimed equality are meromorphic, it suffices to prove the equality on~$\{s\in\C : \Rea s > \tfrac12\}$. To that end let $s\in\C$ with $\Rea s >\tfrac12$ and set $A\coloneqq \twist(g_\ell)$.

    The logarithmic derivative of the Selberg zeta function~$Z_{X,\twist}$ can be calculated by Jacobi's formula as
    \begin{equation}
        \label{eq:calculate-ZCell}
        \begin{aligned}
        \frac{Z_{C_\ell,\twist}'(s)}{Z_{C_\ell,\twist}(s)} &= \sum_{\kappa \in \{\pm 1\}}
            \sum_{j=0}^\infty \partial_s \log \det \left(1 - e^{-(s+j)\ell} A^\kappa\right)\\
        &= \ell \sum_{\kappa \in \{\pm 1\}}\sum_{j=0}^\infty \tr\left( (1 - e^{-(s+j)\ell}A^{\kappa})^{-1} \, e^{-(s+j)\ell}A^\kappa \right).
        \end{aligned}
    \end{equation}
    Analogously to the reasoning in~\cite{Patterson_zeta} (see also \cite[p. 227]{Borthwick_book}) it follows that
    \begin{align*}
        \Phi_{C_\ell,\twist}(s) = \ell \sum_{k\in\Z\setminus\{0\}} \tr(A^k) \frac{e^{-s|k| \ell}}{1 - e^{-|k|\ell}}\,.
    \end{align*}
    Applying the geometric series twice yields
    \begin{align*}
        \Phi_{C_\ell,\twist}(s) = \ell\sum_{\kappa \in \{\pm 1\}}\sum_{j=0}^\infty \tr\left( (1 - e^{-(s+j)\ell}A^\kappa)^{-1} \, e^{-(s+j)\ell}A^\kappa \right).
    \end{align*}
    Combining this with \eqref{eq:calculate-ZCell} yields the claimed equality.
\end{proof}

We recall from Section~\ref{sec:hypcyl} that $Z_{C_\ell,\twist}$ is entire and its set of zeros is
\begin{align*}
    \bigcup_{\lambda \in \EV(\chi(h_\ell))} \bigcup_{\kappa \in \{\pm 1\}} \left( -\N_0 + \kappa \ell^{-1} ( \log \lambda + 2 \pi i \ZZ) \right)\,.
\end{align*}

\subsubsection{Funnel}
We now consider the model funnel~$F_\ell$ for any $\ell>0$, i.e.,
\begin{align*}
    F_\ell = C_\ell \cap \{x + iy \in \C \colon x \geq 0\}\,.
\end{align*}
See Section~\ref{sec:modelfunnel}. We let $\twist$ be a unitary representation of the fundamental group~$\group$ of~$C_\ell$. By Section~\ref{sec:hypcyl} $\group$ is generated by
\[
g_\ell = \bmat{e^{\ell/2}}{0}{0}{e^{-\ell/2}}\,.
\]
Set
\[
A \coloneqq \twist(g_\ell)\,.
\]
Restricting the twisted Laplacian $\Delta_{C_\ell,\twist}$ on $C_\ell$ to $F_\ell$ and imposing Dirichlet boundary conditions on $x = 0$ gives a self-adjoint operator $\Delta_{F_\ell,\twist}$.
We denote its resolvent by
\begin{align*}
    R_{F_\ell,\twist}(s) = (\Delta_{F_\ell,\twist} - s(1-s))^{-1}\,.
\end{align*}
For its Schwartz kernel we have
\begin{align*}
    R_{F_\ell,\twist}(s;z,w) = R_{C_\ell, \twist}(s;z,w) - R_{C_\ell, \twist}(s;z,-\overline{w}).
\end{align*}
We refer to \cite[Section 4.7]{DFP} for a detailed discussion of the resolvent for the funnel.

We define the Selberg zeta function for $(F_\ell,\twist)$ by
\begin{align}\label{eq:def-zeta-funnel}
    Z_{F_\ell,\twist}(s) = e^{-sn\frac{\ell}4} \prod_{\kappa \in \{\pm 1\}} \prod_{j=0}^\infty \det\left( 1 - e^{-(s+2j + 1)\ell} A^\kappa\right)\,,
\end{align}
which is an entire function on $\C$ as the Selberg zeta function for~$(C_\ell,\twist)$ is entire. See Section~\ref{sec:hypcyl}.
\begin{prop}\label{prop:funnel-zeta-deriv}
    The Selberg zeta function and the algebraically regularized resolvent trace for~$(F_\ell,\twist)$ obey the relation
    \begin{align*}
        \frac{Z_{F_\ell,\twist}'}{Z_{F_\ell,\twist}} = \Phi_{F_\ell,\twist}\,.
    \end{align*}
\end{prop}
\begin{proof}
    The proof is analogous from the proof of~\cite[Proposition~2.3]{BJP} with the modifications of Proposition~\ref{prop:Cell-zeta-deriv}. We have that
    \begin{align*}
        \Phi_{F_\ell,\twist}(s) &= -\frac{\ell}{4}\tr(\id_V)
        + \frac{\ell}2 \sum_{k\in\Z \setminus\{0\}} \tr(A^k) \left( \frac{e^{-s|k|\ell}}{1 - e^{-|k|\ell}} - \frac{e^{-s|k|\ell}}{1 + e^{-|k|\ell}}\right)\\
        &= -\frac{\ell n}{4} + \ell \sum_{\kappa \in \{\pm 1\}}\sum_{j=0}^\infty \partial_s \log \det \left(1 - e^{-(s+2j+1)\ell} A^\kappa\right)\,.
    \end{align*}
\end{proof}

Using the fact that the Euler characteristic of a funnel is zero, we obtain from the same argument as in the proof of Lemma~\ref{lem:res-trace} the following relationship between $\Upsilon_{F_\ell,\twist}$ and $Z_{F_\ell,\twist}$.

\begin{cor}\label{cor:upsilon-funnel}
    The scattering theoretically regularized trace of the resolvent and the Selberg zeta function for~$(F_\ell,\twist)$ obey the relation
    \begin{align}\label{eq:upsilon-funnel}
        \Upsilon_{F_\ell,\twist}(s) = \frac{d}{ds} \log \left( \frac{Z_{F_\ell,\twist}(s)}{Z_{F_\ell,\twist}(1-s)}\right)
    \end{align}
    for all~$s\in\C$.
\end{cor}

By the method of images, we immediately obtain that the resolvent for the funnel end with Dirichlet boundary conditions on the central geodesic extends meromorphically to $\CC$. Recall the funnel boundary defining function~$\rho_f$ from Section~\ref{sec:boundary}. We let $\Sph^1_\ell$ denote the boundary at infinity of~$F_\ell$ and let $\iota\colon \Sph^1_\ell \to \overline{F_{\ell}}$ be the associated embedding.
The scattering matrix for~$(F_\ell,\twist)$ is the map
\[
S_{F_{f,\ell}, \twist}(s)\colon C^\infty(\Sph^1_\ell, \iota^*\bundle) \to C^\infty(\Sph^1_\ell, \iota^*\bundle)\,,
\]
defined by
\begin{align*}
    S_{F_{f,\ell}, \twist}(s;\theta,\theta') = \lim_{\rho_f,\rho_f' \to 0} (\rho_f \rho_f')^{-s} R_{F_\ell, \twist}(s; \rho_f, \theta, \rho_f', \theta')\,.
\end{align*}

By the same arguments as in the case of the hyperbolic cylinder, we can deduce a factorization of the zeta function in terms of the Weierstrass product of the resonances for the model funnel.

\begin{prop}[Selberg's trace formula for the model funnel]\label{prop:selberg-funnel}
    There exists an entire function $q$ such that
    \begin{align*}
        Z_{F_\ell,\twist}(s) = e^{q(s)} \mc P_{F_\ell,\twist}(s)
    \end{align*}
    on all of~$\C$.
\end{prop}

\subsubsection{Parabolic cylinder}
We recall the model parabolic cylinder (model cusp) $C_\infty = \ang{T} \backslash \h$ from Section~\ref{sec:cusps}. It has no periodic orbits, hence it does not contribute directly to the Selberg zeta function.
However,  parabolic elements directly contribute to the resolvent trace. Thus, to connect the resolvent trace and the Selberg zeta function for~$(C_\infty,\twist)$, we need to calculate these contributions.

Let $\twist\colon \ang{T}\to\Unit(V)$, and set $B\coloneqq \twist(T)$, $n\coloneqq \dim V$. Let $(\lambda_j)_{j=1}^n$ be the (algebraic) eigenvalues of $B$ (with multiplicities) and let  $n_0$ be the number of eigenvalues that are equal to $1$ (counted with algebraic multiplicity).

\begin{lemma}\label{lem:phi-cusp}
     On all of~$\C$ we have
    \begin{align*}
        \Phi_{C_\infty,\twist}(s) = n_0 \Phi_{C_\infty}(s) - \frac12 \sum_{\kappa \in \{\pm 1\}}\sum_{\lambda_j \not = 1} \log(1 - \lambda_j^\kappa),
    \end{align*}
    where
    \begin{align*}
        \Phi_{C_\infty}(s) = -\log(2) - \frac{\gammafunc'\left(s+\tfrac12\right)}{\gammafunc\left(s+\tfrac12\right)} + \frac{1}{2s-1}
    \end{align*}
    is the algebraically regularized resolvent trace of the scalar, untwisted Laplacian on the model cusp~$C_\infty$.
\end{lemma}

\begin{proof}
We use the argumentation from~\cite{Patterson_zeta}. We have that
    \begin{align*}
        \Phi_{C_\infty,\twist}(s) &= (2s - 1)\zInt{C_\infty} \sum_{k\not = 0} \tr(B^k) R_\h(s;z,z+k)\, d\mu_{C_\infty}(z)\\
        &= (2s-1) \sum_\pm \FP_{\eps\to 0} \int_0^{\eps^{-1}} \int_0^1 \sum_{k=1}^\infty \tr(B^{\pm k}) g_s\left(1+\frac{k^2}{4y^2}\right) \, \frac{dx\,dy}{y^2}\,.
    \end{align*}
    Substituting $u = k/(2y)$ yields
    \begin{align*}
        \Phi_{C_\infty,\twist}(s) = 2(2s-1) \sum_{\pm} \FP_{\eps\to 0}\int_0^\infty \tr\left(\sum_{1\leq k\leq 2u/\eps} \frac{B^{\pm k}}k\right) g_s(1+u^2) du
    \end{align*}
    Since $\tr B^{\pm k} = \sum_{j=1}^n \lambda_j^{\pm k}$, we need to evaluate
    \begin{align*}
        \FP_{N \to \infty}\sum_{k=1}^N \frac{\lambda_j^{\pm k}}k\,.
    \end{align*}
    There are two cases: if $\lambda_j = 1$, we can use the exact same reasoning as in \cite{Borthwick_book}.
    Further, if $\lambda_j \in \{ z \in \C \colon |z|=1, z\not = 1\}$, then the series converges to
    \begin{align*}
        \sum_{k=1}^\infty \frac{\lambda_j^{\pm k}}k = -\log\left(1-\lambda_j^{\pm 1}\right)\,,
    \end{align*}
    where $\log$ denotes the principal branch of the complex logarithm.
    Together with
    \[
    2(2s-1) \int_0^\infty g_s(1 + u^2) du = 1
    \]
    this implies that the contribution of the eigenvalue $\lambda_j \not = 1$ is given by
    $-\log(1 - \lambda_j^{\pm 1})$.

    Summing up the trace, we obtain
    \begin{align*}
        \Phi_{C_\infty,\twist}(s) = n_0 \Phi_{C_\infty}(s) - \sum_{\kappa \in \{\pm 1\}}\sum_{\lambda_j \not = 1} \log(1 - \lambda_j^\kappa)\,.
    \end{align*}
    This completes the proof.
\end{proof}

\subsubsection{Cone}

Recall from Section~\ref{sec:modelcone} the model cone~$C_\theta^\wedge$ whose fundamental group~$\group$ is generated by
\begin{equation*}
 g = g_\theta \coloneqq \bmat{\cos(\theta)}{\sin(\theta)}{-\sin(\theta)}{\cos(\theta)}
\end{equation*}
with $\theta = \pi/q$ for some $q \in \N$, $q>1$. To investigate the algebraically regularized resolvent trace~$\Phi_{C_\theta^\wedge,\twist}$ of~$(C_\theta^\wedge,\twist)$ we require some preparations. We let $\psi$ denote the digamma function:
\[
\psi = \frac{\gammafunc'}{\gammafunc}\,.
\]

\begin{lemma}\label{lem:fischer}
    Let $s \in \C$ with $\Rea s > \tfrac12$ and $\vartheta = \tfrac{p}{q}\pi$ with $q \in \N$ and $p \in \{1,\dotsc,q-1\}$. Then
    \begin{align*}
        \int_0^1 & \left(1-x \cos(\vartheta)^2\right)^{-\frac12} x^{-\frac32} g_s\left(\frac1x\right) dx
        \\ & = \frac{1}{2s-1} \frac{i}{2\pi q} \sum_{j=0}^{q-1} \left( e^{i\vartheta(2j+1)} - e^{-i\vartheta(2j+1)} \right) \psi\left( \frac{s+j}q \right) \,.
    \end{align*}
\end{lemma}
\begin{proof}
    By~\cite[Lemma 2.3.2]{Fischer},
    \begin{gather*}
        \int_0^1 \left(1-x \cos(\vartheta)^2\right)^{-\tfrac12} x^{-\tfrac32} g_s\left(\frac1x\right) dx =
        \frac{1}{2s-1} \frac{i}{2\pi} \sum_{m=0}^{\infty} \frac{ e^{i\vartheta(2m+1)} - e^{-i\vartheta(2m+1)} }{s+m}  \,.
    \end{gather*}
    Since the series does not converge absolutely, we cannot directly rearrange the order of summation.
    We write the conditionally convergent series as a Lerch zeta function,
    \begin{align*}
        \sum_{m=0}^\infty \frac{e^{-2 im\vartheta}}{s+m} = \FP_{z \to 1} \sum_{m=0}^\infty \frac{e^{-2im\vartheta}}{(s+m)^z}\,.
    \end{align*}
    Then we have that
    \begin{align*}
        \sum_{m=0}^\infty \frac{e^{-2 i m\vartheta}}{(s+m)^z} &= \sum_{n=0}^\infty \sum_{k=0}^{q-1} \frac{e^{-2i k\vartheta}}{(s+k+nq)^z}
    \end{align*}
    Since the sum converges absolutely, we may interchange the summation and obtain
    \begin{align*}
        \sum_{m=0}^\infty \frac{e^{-2 i m\vartheta}}{(s+m)^z} &= \sum_{k=0}^{q-1} e^{-2i k\vartheta} \sum_{n=0}^\infty (s+k+nq)^{-z} \\
        &= q^{-z} \sum_{k=0}^{q-1} e^{-2i k\vartheta} \zeta\left( z, \frac{s+k}q \right)\,,
    \end{align*}
    where $\zeta$ denotes the Hurwitz zeta function. From this we can calculate the finite part using \cite[1.10.(9)]{ErdelyiI} as
    \begin{align*}
        \FP_{z \to 1} \sum_{m=0}^\infty \frac{e^{-2im\vartheta}}{(s+m)^z} = - \frac{1}{q} \sum_{k=0}^{q-1} e^{-2i k\vartheta} \psi\left(\frac{s+k}q\right)\,.
    \end{align*}
\end{proof}

\begin{lemma}\label{lem:exp-sum}
    Let $m \in \Z$ and $q \in \N$. Then
    \begin{align*}
        \sum_{k=1}^{q-1} \frac{e^{\pm 2\pi i k m /q}}{1 - e^{\mp 2\pi i k/q}} = \frac12(q-1) - m + q \floor*{\frac{m}{q}}\,.
    \end{align*}
\end{lemma}
\begin{proof}
    Taking advantage of the geometric series, we obtain
    \begin{align*}
        \frac{e^{\pm 2\pi i k m /q}}{1 - e^{\mp 2\pi i k/q}} &= \lim_{t\to 1} \sum_{n=0}^\infty \sum_{\nu = 0}^{q-1} t^{\nu + nq} e^{\pm 2\pi i k (m - \nu - nq)/q} \\
        &= \lim_{t\to 1} \sum_{\nu=0}^{q-1} \frac{t^\nu}{1 - t^q} e^{\pm 2\pi i k (m - \nu)/q}\,.
    \end{align*}
    We note that
    \begin{align*}
        \sum_{k=1}^{q-1} e^{\pm 2\pi i k (m - \nu)/q} =
        \begin{cases}
            q-1, & m \equiv \nu \mod q\,, \\
            -1, & \text{otherwise}\,.
        \end{cases}
    \end{align*}
    Let $\tilde{m} \in \{0,\dotsc,q-1\}$ such that $m \equiv \tilde{m} \mod q$.
    Using l'H\^opital's rule, we obtain
    \begin{align*}
        \sum_{k=1}^{q-1} \frac{e^{2\pi ik m/q}}{1 - e^{-2\pi ik /q}} &= \lim_{t\to 1} \sum_{\nu=0}^{q-1} \frac{t^\nu}{1 - t^q} \sum_{k=1}^{q-1} e^{2\pi i k (m - \nu)/q}\\
        &= \lim_{t\to 1} \sum_{\substack{\nu=0\\\nu \not = \tilde{m}}}^{q-1} \frac{-t^\nu}{1 - t^q} + \lim_{t \to 1} \frac{(q-1) t^{\tilde{m}}}{1-t^q}\\
        &= \sum_{\substack{\nu=0\\\nu \not = \tilde{m}}}^{q-1} \frac{\nu}{q} - \left(1-\frac{1}{q}\right)\tilde{m} \\
        &= \sum_{\nu = 0}^{q-1} \frac{\nu}{q} - \tilde{m} \\
        &= \frac12(q-1) - \tilde{m} \,.
    \end{align*}
    Since \[\tilde{m} = m - q\left\lfloor\frac{m}{q}\right\rfloor\] this finishes the proof.
\end{proof}

With these preparations we can now determine the algebraically regularized resolvent trace $\Phi_{C_\theta^\wedge,\twist}$. Recall that $\theta = \pi/q$.
\begin{prop}\label{prop:phi-cone}
    For $s \in \C$ with $\Rea s > \tfrac12$, the algebraically regularized trace of the resolvent, $\Phi_{C_\theta^\wedge,\twist}$, is
    \begin{align*}
        \Phi_{C_\theta^\wedge,\twist}(s) = \frac{d}{ds} \log \left( \prod_{m=0}^{q-1} \gammafunc\left( \frac{s+m}q\right)^{\dim(V) \frac{2m+1}{q} -\alpha_m} \right)\,,
    \end{align*}
    where
    \begin{align*}
        \alpha_m &\coloneqq \sum_{\alpha\in\EVP\left( \twist(g_\theta) \right) } N_q(\alpha,m) \,.
    \end{align*}
\end{prop}
\begin{proof}
    The proof is analogous to one in~\cite{Fischer}. In what follows we provide the details.
    We note that
    \begin{align*}
        \sigma(z,g_\theta z) = \frac{\sin(\theta)^2(1+x^2+y^2) + 4\cos(\theta)^2 y^2}{4y^2},
    \end{align*}
    where $z = x + iy$.
    For $t > 1$ we have
    \begin{align*}
        J(t) &\coloneqq\int\limits_{t-\sqrt{t^2 -1}}^{t+\sqrt{t^2-1}} \frac{dy}{y\sqrt{t^2 - 1 -(y-t)^2}}\\
        &= - \cot\left( \frac{1-ty}{\sqrt{-y^2 +2ty -1}}\right) \Big|_{y=t-\sqrt{t^2-1}}^{t+\sqrt{t^2-1}}\\
        &= \pi.
    \end{align*}
    Let $\vartheta \coloneqq k \pi/q = k \theta$.
    With the substitution
    \[
        (\tilde x, \tilde y) = \left( (1 + x^2 + y^2)\sin\left(\frac\theta y\right), y\right)
    \]
    we obtain
    \begin{align*}
        \int_\h R(s; z, g_\vartheta z) d\mu_\h(z) &= \int_0^\infty \int_{-\infty}^\infty g_s\left( \frac{(\sin\vartheta)^2(1+x^2+y^2) + 4(\cos\vartheta)^2 y^2}{4y^2} \right) \frac{dx \, dy}{y^2}\\
        &= (\sin\vartheta)^{-1} \int_0^\infty g_s\left(\frac{\tilde{x}^2 + 4 (\cos\vartheta)^2}4\right) J\left(\frac{\tilde{x}}{2\sin(\vartheta)}\right) d\tilde{x}\\
        &= \pi (\sin\vartheta)^{-1} \int_{2\sin(\vartheta)}^\infty g_s\left(\frac{\tilde{x}^2 + 4 (\cos\vartheta)^2}4\right) d\tilde{x}\,,
    \end{align*}
    where $g_s$ is given by \eqref{eq:def_gs}.
    Using now the substitution
    \[
    t = \frac{4}{\tilde{x}^2 + 4(\cos\vartheta)^2}
    \]
    Lemma~\ref{lem:fischer} yields
    \begin{align*}
        \int_\h R(s; & z, g_\vartheta z)  d\mu_\h(z)
        \\
        &= \frac{\pi}{\sin(\vartheta)} \int_0^1 t^{-\frac32} \left(1 - t(\cos\vartheta)^2 \right)^{-\frac12} g_s\left(\frac1t\right) dt
        \\
        &= \frac{1}{2s-1} \frac{i}{2 q \sin(\vartheta)}
        \sum_{m=0}^{q-1} \left( e^{i\vartheta(2m+1)} - e^{-i\vartheta(2m+1)} \right) \psi\left( \frac{s+m}q \right)\, .
    \end{align*}
    Hence,
    \begin{align*}
        (2s-1) \int_\h & R(s;z,g_\theta^k z) d\mu_\h(z) \\ & = \frac{i}{2 q \sin( k \pi/q)} \sum_{m=0}^{q-1} \left( e^{i\pi k(2m+1)/q} - e^{-i\pi k(2m+1)/ q}\right) \psi\left( \frac{s+m}q \right)\, .
    \end{align*}
    Using that \[\int_\h R(s; z, g_\theta^k z) d\mu_\h(z) = q \int_{C_\theta^\wedge} R(s; z, g_\theta^k z) \,d\mu_{C_\theta^\wedge}(z)\,,\] we obtain
    \begin{align*}
        \Phi_{C_\theta^\wedge,\twist}(s) &= \sum_{k=1}^{q-1} \frac{i}{2 q^2 \sin(\pi k /q )} \Tr\twist(g_\theta^k) \\
        &\phantom{= \sum_{k=1}} \cdot \sum_{m=0}^{q-1} \left( e^{i\pi k(2m+1)/q} - e^{-i\pi k(2m+1)/q}\right) \psi\left( \frac{s+m}q \right)\, .
    \end{align*}
    We now use that
    \[
    \frac1i 2\sin\left(k\pi/q\right) = e^{-i\pi k/q} - e^{i\pi k/q}
    \]
    to conclude
    \begin{align*}
        \frac{i}{2\sin(\pi k/q)} & \Tr\twist(g_\theta^k) \left( e^{i\pi k(2m+1)/q} - e^{-i\pi k (2m + 1)/q} \right) \\
        &= - \sum_{\alpha\in\EVP(\twist(g_\theta))} \frac{e^{2\pi i k (m + \alpha)/q}}{1 - e^{-2\pi i k/q}} + \frac{e^{-2\pi i k (m - \alpha)/q}}{1 - e^{2\pi i k/q}}\,.
    \end{align*}
    Consequently we have
    \begin{align*}
        \Phi_{C_\theta^\wedge,\twist}(s) = -\frac{1}{q^2}\sum_{m=0}^{q-1} \sum_{\alpha\in\EVP(\twist(g_\theta))} \sum_{k=1}^{q-1} \left( \frac{e^{2\pi i k (m + \alpha)/q}}{1 - e^{-2\pi i k/q}} + \frac{e^{-2\pi i k (m - \alpha)/q}}{1 - e^{2\pi i k/q}} \right) \psi\left( \frac{s+m}q \right)\, .
    \end{align*}
    Using Lemma~\ref{lem:exp-sum}, we can simplify the parenthesis to
    \begin{align*}
        \sum_{k=1}^{q-1} &\left( \frac{e^{2\pi i k (m + \alpha)/q}}{1 - e^{-2\pi i k/q}} + \frac{e^{-2\pi i k (m - \alpha)/q}}{1 - e^{2\pi i k/q}} \right) \\
        &= (q-1) -2m + q\left( \floor*{\frac{m+\alpha}{q}} + \floor*{\frac{m-\alpha}{q}} \right) \\
        &= -(2m + 1) + q\left( 1 + \floor*{\frac{m+\alpha}{q}} + \floor*{\frac{m-\alpha}{q}} \right) \\
        &= -(2m + 1) + q N_q(\alpha,m) \,.
    \end{align*}
    Therefore,
    \begin{align*}
        \Phi_{C_\theta^\wedge,\twist}(s)
        &= \sum_{m=0}^{q-1} \left( \dim(V) \frac{2m+1}{q} - \sum_{ \alpha\in\EVP(\twist(g_\theta)) } N_q(\alpha,m)\right) \frac1q\psi\left( \frac{s+m}q\right)\,.
    \end{align*}
    Since \[\frac1q\psi\left( \frac{s+m}q\right) = \frac{d}{ds} \log \gammafunc\left( \frac{s+m}q\right)\,,\]
    we conclude that
    \begin{align*}
        \Phi_{C_\theta^\wedge,\twist}(s)
        &= \frac{d}{ds} \log \left(\prod_{m=0}^{q-1} \gammafunc\left( \frac{s+m}q\right)^{\dim(V)\,(2m+1)/q - \alpha_m} \right) \,,
    \end{align*}
    where \[\alpha_m \coloneqq \sum_{\alpha\in\EVP(\twist(g_\theta))} N_q(\alpha,m)\,.\]
    This completes the proof.
\end{proof}

\subsection{Relation to the Selberg zeta function}\label{sec:trace-zeta}
In this section we will discuss the relation between the logarithmic derivative of the Selberg zeta function~$Z_{X,\twist}$ for~$(X,\twist)$ and the scattering theoretically regularized
resolvent~$\Upsilon_{X,\twist}$ from~\eqref{eq:def_Upsilon}. In particular we will prove that the logarithmic derivative of~$Z_{X,\twist}$ admits a meromorphic continuation to all of~$\C$, which we will denote by~$Z_{X,\twist}'/Z_{X,\twist}$, thus continuing to use the notation on the domain of its meromorphic continuation.

\begin{prop}\label{prop:upsilon-zeta}
    The logarithmic derivative of the Selberg zeta function~$Z_{X,\twist}$ and the scattering theoretically regularized trace of the resolvent~$\Upsilon_{X,\twist}$ satisfy, on all of~$\C$,
    \begin{align*}
        \Upsilon_{X,\twist}(s)
        &= \frac{Z_{X,\twist}'(s)}{Z_{X,\twist}(s)} + \frac{Z_{X,\twist}'(1-s)}{Z_{X,\twist}(1-s)} \\
        &\phantom{=} + \frac{d}{ds} \log \left( e^{p(s)} \frac{\gammafunc\left(\frac12-s\right)^{n_p}}{\gammafunc\left(s-\frac12\right)^{n_p}}
        \frac{G_{X^\wedge,\twist}(1-s)}{G_{X^\wedge,\twist}(s)} \frac{G_\infty(1-s)^{-\dim(V)\,\eulertop(X)}}{G_\infty(s)^{-\dim(V)\,\eulertop(X)}} \right)
    \end{align*}
    for some polynomial~$p$ of degree at most~$1$.
\end{prop}

We outsource some intermediate results for the proof of Proposition~\ref{prop:upsilon-zeta} into the subsequent lemmas.
For $[g] \in [\group]_\ellip$ and $m\in\N_0$ we set
\begin{align*}
    \alpha_g(m) \coloneqq \sum_{\alpha\in\EVP(\twist(g))} N_{\deg(g)}(\alpha,m)\,.
\end{align*}
Let $\delta$ be the Hausdorff dimension of the limit set of~$\group$. Further let $R_\ellprim$ be a set of representatives for~$[\group]_\ellprim$, and likewise $R_\hypprim$ a set of representatives for~$[\group]_\hypprim$, and $R_\parprim$ a set of representatives for~$[\group]_\parprim$.

\begin{lemma}\label{lem:decomp-phi}
    For $\Rea s > \delta$ we have
    \begin{align*}
        \Phi_{X,\twist}(s) &= \frac{Z_{X,\twist}'(s)}{Z_{X,\twist}(s)} + n_p \Phi_{C_\infty}(s) - I_p \\
        &\phantom{=} + \frac{d}{ds} \log
        \prod_{[g] \in [\group]_\ellprim} \prod_{m=0}^{\deg(g)-1} \gammafunc\left(\frac{s+m}{\deg(g)}\right)^{\frac{2m+1}{\deg(g)}\dim(V) - \alpha_g(m)}\,,
    \end{align*}
    where $I_p$ is a constant only depending on the eigenvalues associated to the parabolic elements in $\group$.
    In particular, the logarithmic derivative of $Z_{X,\twist}$ admits a meromorphic continuation to all of~$\C$.
\end{lemma}
\begin{proof}
    We recall that the integral kernel of the
    resolvent of the twisted Laplacian~$\LapTwist$ is
    \[
     \ResTwist(s;z,w) = \sum_{g\in\group} \twist(g) R_\h(s;z,g.w)\,.
    \]
    We have
    \begin{align*}
        \ResTwist(s;z,w)  &= \id_V R_\h(s;z,w) \\
        & + \sum_{[g]\in [\group]_\hypprim} \sum_{h \in \ang{g}\backslash\group} \sum_{k\in\N}
        \twist(h^{-1}g^k h) R_\h(s;z,h^{-1}g^k h  w)
        \\
        &  + \sum_{[g]\in [\group]_\parprim} \sum_{h \in
        \ang{g}\backslash\group} \sum_{k \in\N} \twist(h^{-1}g^kh)
        R_\h(s;z,h^{-1}g^k h  w)\\
        &  + \sum_{[g]\in [\group]_\ellprim} \sum_{h \in
        \ang{g}\backslash\group} \sum_{k = 1}^{\deg(g) - 1}
        \twist(h^{-1}g^kh) R_\h(s;z,h^{-1}g^k h  w)\,.
    \end{align*}
    We note that $\tr \twist(h^{-1} g^k h) = \tr \twist(g^k)$ and
    \begin{align*}
        \sum_{h \in \ang{g}\backslash\group} \zInt{X} R_\h(s;z,h^{-1}g^k h  z) \,d\mu_X(z)
        = \zInt{\ang{g} \backslash \h} R_\h(s;z,g^k z)\,d\mu_X(z)\,.
    \end{align*}
    We obtain that
    \begin{align}\label{eq:Phi-primitive}
        \Phi_{X,\twist}(s) = \sum_{g \in R_\hypprim} \Phi_{\ang{g}\backslash \h,\twist|_{\ang{g}}}(s)
        + \sum_{g \in R_\parprim} \Phi_{\ang{g}\backslash \h,\twist|_{\ang{g}}}(s)
        + \sum_{g \in R_\ellprim} \Phi_{\ang{g}\backslash \h,\twist|_{\ang{g}}}(s)\,.
    \end{align}
    By Proposition~\ref{prop:Cell-zeta-deriv} we have that for $g \in R_\hypprim$,
    \begin{align*}
        \Phi_{\ang{g}\backslash \h,\twist|_{\ang{g}}}(s) = \frac{Z_{\ang{g}\backslash \h,\twist|_{\ang{g}}}'(s)}{Z_{\ang{g}\backslash \h,\twist|_{\ang{g}}}(s)}\,.
    \end{align*}
    For $\Rea s \gtrsim 0$, we have that
    \begin{align*}
        Z_{X,\twist}(s) &= \prod_{[g] \in [\group]_\hypprim} \prod_{k=0}^\infty \det\left( 1 - \twist(g) N(g)^{-(s+k)} \right) \\
        &= \prod_{g \in R_\hypprim} Z_{\ang{g}\backslash \h,\twist|_{\ang{g}}}(s)
    \end{align*}
    and therefore, we obtain
    \begin{align}\label{eq:phi-hypprim}
        \sum_{g \in R_\hypprim} \Phi_{\ang{g}\backslash \h,\twist|_{\ang{g}}}(s) = \frac{Z_{X,\twist}'(s)}{Z_{X,\twist}(s)}\,.
    \end{align}
    By Lemma~\ref{lem:phi-cusp}, we have that
    \begin{align}\label{eq:phi-parprim}
        \sum_{g \in R_\parprim} \Phi_{\ang{g}\backslash \h,\twist|_{\ang{g}}}(s) = n_p \Phi_{C_\infty}(s) - I_p
    \end{align}
for some constant~$I_p$ depending only on the eigenvalues of~$\twist$ at parabolic elements in~$\group$.
    Proposition~\ref{prop:phi-cone} shows that for $[g]\in [\group]_\ellprim$,
    \begin{align}\label{eq:phi-ellprim}
        \Phi_{\ang{g}\backslash\h, \twist|_{\ang{g}}}(s)
        = \frac{d}{ds} \log \prod_{\alpha\in\EVP(\twist(g))} \prod_{m=0}^{q-1} \gammafunc\left( \frac{s+m}q \right)^{(2m+1)/q - N_q(\alpha,m)}
    \end{align}
    where $q \coloneqq \deg(g)$.
\end{proof}

To keep track of all the different contributions to $\Upsilon_{X,\twist}$, we define the following quantities on~$\C$:
\begin{align*}
    \Upsilon_{X,\twist}^{\hyp}(s) &\coloneqq \sum_{g \in R_{\hypprim}} \Phi_{\ang{g}\backslash \h, \twist|_{\ang{g}}}(s) + \Phi_{\ang{g}\backslash \h, \twist|_{\ang{g}}}(1-s)\,,
    \\
    \Upsilon_{X,\twist}^{\parab}(s) &\coloneqq \sum_{g \in R_{\parprim}} \Phi_{\ang{g}\backslash \h, \twist|_{\ang{g}}}(s) + \Phi_{\ang{g}\backslash \h, \twist|_{\ang{g}}}(1-s)\,,
    \\
    \Upsilon_{X,\twist}^{\ellip}(s) &\coloneqq (2s-1) \pi \cot(\pi s) \dim(V) \sum_{w\in\sing(X)} \left(1 - \frac1{\#\Iso(w)}\right)
    \\
    &\ \phantom{=} + \sum_{g \in R_{\ellprim}} \Phi_{\ang{g}\backslash \h, \twist|_{\ang{g}}}(s) + \Phi_{\ang{g}\backslash \h, \twist|_{\ang{g}}}(1-s)\,,
    \\
    \Upsilon_{X,\twist}^{\interior}(s) &\coloneqq -(2s -1) \pi \cot(\pi s) \dim(V)\,\eulertop(X)\,.
\end{align*}
\begin{lemma}\label{lem:decomp-ups}
    We have, for all $s\in\C$,
    \begin{align*}
        \Upsilon_{X,\twist}(s) = \Upsilon_{X,\twist}^{\hyp}(s) + \Upsilon_{X,\twist}^{\parab}(s) + \Upsilon_{X,\twist}^{\ellip}(s) + \Upsilon_{X,\twist}^{\interior}(s)\,.
    \end{align*}
\end{lemma}
\begin{proof}
    This follows directly from \eqref{eq:Phi_Ups} and \eqref{eq:Phi-primitive} together with the calculation of the zero-volume in Proposition~\ref{prop:euler-characteristic}.
\end{proof}

The next step is to relate the different summands of $\Upsilon_{X,\twist}$ to the Selberg zeta function and the factors in the divisor formula in Theorem~\ref{thm:selberg-factorization}, as provided by the following lemma.

\begin{lemma}\label{lem:ups-bullet}
    On all of~$\C$ we have
    \begin{align*}
        \Upsilon_{X,\twist}^{\hyp}(s) &= \frac{Z_{X,\twist}'(s)}{Z_{X,\twist}(s)} + \frac{Z_{X,\twist}'(1-s)}{Z_{X,\twist}(1-s)} \,,\\
        \Upsilon_{X,\twist}^{\parab}(s) &= -\frac{d}{ds} \log \left( e^{\log(4)sn_p - 2sI_p} \frac{\gammafunc\left(s-\frac12\right)^{n_p}}{\gammafunc\left(\frac12-s\right)^{n_p}} \right)\,,\\
        \Upsilon_{X,\twist}^{\interior}(s) &= \frac{d}{ds} \log \left( \frac{G_\infty(1-s)}{G_\infty(s)} \right)^{-\dim(V)\,\eulertop(X)}\,, \\
        \Upsilon_{X,\twist}^{\ellip}(s) &= \frac{d}{ds} \log \frac{G_{X^\wedge,\twist}(1-s)}{G_{X^\wedge,\twist}(s)}\,,
    \end{align*}
    where $I_p$ is the constant from Lemma~\ref{lem:decomp-phi}.
\end{lemma}
\begin{proof}
    The equality for~$\Upsilon_{X,\twist}^\hyp$ follows directly from \eqref{eq:phi-hypprim}.
    By \cite[Corollary~2.5]{BJP} we have for the untwisted case
    \begin{align*}
        \Phi_{C_\infty}(s) + \Phi_{C_\infty}(1-s)
        &= - \frac{d}{ds} \log\left( 4^s \frac{\gammafunc\left(s-\frac12\right)}{\gammafunc\left(\frac12-s\right)}\right)\,.
    \end{align*}
    Together with \eqref{eq:phi-parprim} this gives the claimed expression for~$\Upsilon_{X,\twist}^{\parab}$.
    For the expressions for~$\Upsilon_{X,\twist}^{\interior}$ and~$\Upsilon_{X,\twist}^{\ellip}$ we first note that
    the logarithmic derivative of $G_\infty(s)/G_\infty(1-s)$ is given by
\begin{align}\label{eq:barnes}
    \frac{d}{ds} \log \frac{G_\infty(s)}{G_\infty(1-s)} = -(2s-1) \pi \cot(\pi s)\,.
\end{align}
    See, e.g., \cite[(4.7)]{BJP}. This yields the formula for~$\Upsilon_{X,\twist}^{\interior}$. For the expression for~$\Upsilon_{X,\twist}^{\ellip}$ we combine \eqref{eq:phi-ellprim} and \eqref{eq:barnes} to obtain, for $[g] \in [\group]_\ellprim$ with $q\coloneqq\deg(g)$,
    \begin{align*}
        (2s-1) &\pi \cot(\pi s)\thin \dim(V)\thin\left(1 - \frac1q\right) + \Phi_{\ang{g}\backslash\h, \twist|_{\ang{g}}}(s) + \Phi_{\ang{g}\backslash\h, \twist|_{\ang{g}}}(1-s)
        \\
        & = \frac{d}{ds} \log \prod_{\alpha\in\EVP(\twist(g))} \left( \frac{G_\infty(s)^{-1+\frac1q}}{G_\infty(1-s)^{-1+\frac1q}} \prod_{m=0}^{q-1} \frac{\gammafunc\left( \frac{s+m}q \right)^{\frac{2m+1}q - N_q(\alpha,m)}}{ \gammafunc\left( \frac{1-s+m}q \right)^{\frac{2m+1}q - N_q(\alpha,m)}}\right)
        \\
        & = \frac{d}{ds} \log \prod_{\alpha\in\EVP(\twist(g))} \frac{G_{q,\alpha}(1-s)}{G_{q,\alpha}(s)}\,.
    \end{align*}
    Hence
    \begin{align*}
        \Upsilon_{X,\twist}^{\ellip}(s) &= \frac{d}{ds} \log \prod_{g \in R_\ellprim} \prod_{\alpha\in\EVP(\twist(g))} \frac{G_{\deg(g),\alpha}(1-s)}{G_{\deg(g),\alpha}(s)}
        \\
        &= \frac{d}{ds} \log \frac{G_{X^\wedge,\twist}(1-s)}{G_{X^\wedge,\twist}(s)}\,.
    \end{align*}
    This completes the proof.
\end{proof}

With these preparations we can now complete the proof of Proposition~\ref{prop:upsilon-zeta}.

\begin{proof}[Proof of Proposition~\ref{prop:upsilon-zeta}]
    Lemma~\ref{lem:decomp-ups} and Lemma~\ref{lem:ups-bullet} imply
    \begin{align*}
        \Upsilon_{X,\twist}(s) &= \Upsilon_{X,\twist}^{\hyp}(s) + \Upsilon_{X,\twist}^{\parab}(s) + \Upsilon_{X,\twist}^{\ellip}(s) + \Upsilon_{X,\twist}^{\interior}(s) \\
        &= \frac{Z_{X,\twist}'(s)}{Z_{X,\twist}(s)} + \frac{Z_{X,\twist}'(1-s)}{Z_{X,\twist}(1-s)} - \frac{d}{ds} \log \left( e^{\log(4)sn_p -2sI_p} \frac{\gammafunc\left(s-\frac12\right)^{n_p}}{\gammafunc\left(\frac12-s\right)^{n_p}} \right)
        \\
        &\phantom{=} + \frac{d}{ds} \log \frac{G_{X^\wedge,\twist}(1-s)}{G_{X^\wedge,\twist}(s)} + \frac{d}{ds} \log \left( \frac{G_\infty(1-s)}{G_\infty(s)} \right)^{-\dim(V)\,\eulertop(X)}
        \\
        &= \frac{Z_{X,\twist}'(s)}{Z_{X,\twist}(s)} + \frac{Z_{X,\twist}'(1-s)}{Z_{X,\twist}(1-s)}
        \\
        &\phantom{=} + \frac{d}{ds} \log \bigg( e^{p(s)} \frac{\gammafunc\left(\frac12-s\right)^{n_p}}{\gammafunc\left(s-\frac12\right)^{n_p}}
        \frac{G_{X^\wedge,\twist}(1-s)}{G_{X^\wedge,\twist}(s)} \frac{G_\infty(1-s)^{-\dim(V)\,\eulertop(X)}}{G_\infty(s)^{-\dim(V)\,\eulertop(X)}} \bigg)
    \end{align*}
    with $p(s)\coloneqq s(I_p-n_p\log 4)$ being a polynomial of degree at most~$1$.
\end{proof}

\subsection{Relation to the relative scattering determinant}
In this section we will study the relation between the scattering theoretically regularized resolvent trace~$\Upsilon_{X,\twist}$ for~$(X,\twist)$ and the scattering matrix~$S_{X,\twist}$, more precisely between their relative variants. For any $s\in\C$ we set, as in~\cite[Section~5.5]{DFP2},
\[
S_{X_{f,c},\twist}(s) \coloneqq \mat{S_{X_f,\twist}(s)}{0}{0}{-1}
\]
and define the \emph{relative scattering matrix} by
\[
S_{X,\twist}^{\rel}(s) \coloneqq S_{X_{f,c},\twist}(s)^{-1}S_{X,\twist}(s)
\]
and the \emph{relative scattering determinant} by
\[
\tau_{X,\twist}(s) \coloneqq \det S_{X,\twist}^{\rel}(s)\,.
\]
We refer to~\cite[Section~5.5]{DFP2} for justification and details.
Further, by Corollary~\ref{cor:upsilon-funnel},
\begin{align}\label{eq:upsilon-Xf}
    \Upsilon_{X_f,\twist}(s) = \frac{d}{ds} \log \left( \frac{Z_{X_f,\twist}(s)}{Z_{X_f,\twist}(1-s)}\right)\,,
\end{align}
where the Selberg zeta function of $X_f$ is defined as the product of the Selberg zeta
functions of the connected components
and each connected component is isometrically isomorphic to a model funnel. In this section we will establish the relation between~$\Upsilon_{X,\twist}$, $\Upsilon_{X_f,\twist}$ and $\tau_{X,\twist}$ as stated in the following proposition.

\begin{prop}\label{prop:upsilon-tau}
    For $s\in\C$ with $\Rea s = \tfrac12$ but $s \not = \tfrac12$ we have
    \begin{align}\label{eq:upsilon-tau}
        \Upsilon_{X,\twist}(s) - \Upsilon_{X_f,\twist}(s) = - \frac{\tau_{X,\twist}'(s)}{\tau_{X,\twist}(s)}\,.
    \end{align}
\end{prop}

As preparation for the proof of Proposition~\ref{prop:upsilon-tau} we recall some notation from~\cite{DFP2} regarding components of the Poisson operator~$E_{X,\twist}$. Let $s\in\C$ not being a resonance. For $j\in\{1,\ldots,n_f\}$ we define
\[
E_{X,\twist}^{f,j}\colon C^\infty(\partial_\infty X_{f,j}, E_\twist\vert_{\partial_\infty X_{f,j}}) \to C^\infty(X,E_\twist)
\]
by
\[
E_{X,\twist}^{f,j}(s,z,\theta')\coloneqq (\rho_f')^{-s}R_{X,\twist}(s;z,z')\vert_{X\times\partial_\infty X_{f,j}}\,.
\]
For $j\in \{1,\ldots, n_c\}$ we define
\[
E_{X,\twist}^{c,j}(s) \colon C^\infty(\partial_\infty X_{c,j}, E_\twist\vert_{\partial_\infty X_{c,j}}) \to C^\infty(X,E_\twist)
\]
by
\[
E_{X,\twist}^{c,j}(s,z,\theta') \coloneqq (\rho_c')^{1-s}R_{X,\twist}(s;z,z')\vert_{X\times\partial_\infty X_{c,j}}\,.
\]
For $\alpha, \beta \in \{c,f\}$, $j \in \{1,\dotsc,n_\alpha\}$ and $k \in \{1,\dotsc,n_\beta\}$ we set
\begin{align*}
    I^{\alpha\beta}_{jk}(s,\eps) & \coloneqq
        \int_{\pa X_{\alpha,j,\eps}} \int_{\pa_\infty X_{\beta,k}}
        \Big( \pa_s E_{X,\twist}^{\beta,k}(s;z,\phi')
        \pa_\nu E_{X,\twist}^{\beta,k}(1-s;z,\phi')^T \\
    &\phantom{= (2s - 1)} - \pa_s \pa_\nu E_{X,\twist}^{\beta,k}(s;z,\phi')
        E_{X,\twist}^{\beta,k}(1-s;z,\phi')^T \Big) \, d\mu_\pa(\phi') d\sigma_\eps(z)\,,
\end{align*}
and for $j \in \{1,\dotsc,n_f\}$ we define
\begin{align*}
    I_{f,j}(s,\eps) & \coloneqq
        \int_{\pa X_{f,j,\eps}} \int_{\pa_\infty X_{f,j}}
        \Big( \pa_s E_{X_{f,j},\twist}(s;z,\phi') \pa_\nu E_{X_{f,j},\twist}(1-s;z,\phi')^T \\
    &\phantom{= (2s-1)} - \pa_s \pa_\nu E_{X_{f,j},\twist}(s;z,\phi')
        E_{X_{f,j},\twist}(1-s;z,\phi')^T \Big) \, d\mu_\pa(\phi') d\sigma_\eps(z)\,,
\end{align*}
where $E_{X_{f,j},\twist}$ denotes the Poisson operator for~$(X_{f,j},\twist)$.
Moreover we set
\begin{align*}
    I_{X,\twist}(s,\eps) & \coloneqq \sum_{i,j=1}^{n_f} I^{\ff}_{ij}(s,\eps)
    + \sum_{i=1}^{n_f} \sum_{k=1}^{n_c} \left( I_{ik}^{\fc}(s,\eps) + I_{ki}^{\cf}(s,\eps) \right)
    + \sum_{k,l=1}^{n_c} I_{kl}^{\cc}(s,\eps)
\end{align*}
and
\begin{align*}
    I_{X_f,\twist}(s,\eps) &\coloneqq \sum_{j=1}^{n_f} I_{f,j}(s,\eps)\,.
\end{align*}

With the following lemma we show that we may calculate $\Upsilon_{X,\twist}$ in terms of $I_{X,\twist}$ and $I_{X_f,\twist}$.

\begin{lemma}\label{lem:Upsilon_FP}
    We have that
    \begin{align*}
        \Upsilon_{X,\twist}(s) &= (2s - 1) \FP_{\eps \to 0} \tr_{\bundle} I_{X,\twist}(s, \eps)
        \intertext{and}
        \Upsilon_{X_f,\twist}(s) &= (2s - 1) \FP_{\eps \to 0} \tr_{\bundle} I_{X_f,\twist}(s, \eps)\,.
    \end{align*}
\end{lemma}
\begin{proof}
    This proof is an easy adaptation of the beginning of the proof of \cite[Proposition~10.4]{Borthwick_book}.
    We give a short summary.

    Recall from \cite[Proposition~5.1]{DFP2} that
    \begin{align*}
        \ResTwist(s) - \ResTwist(1-s) &= (1-2s) E_{X,\twist}(s) E_{X,\twist}(1-s)^T\,,
    \end{align*}
    which implies that
    \begin{align*}
        \Upsilon_{X,\twist}(s) &= - (2s-1)^2 \zTr\left( E_{X,\twist}(s) E_{X,\twist}(1 - s)^T \right) \,.
    \end{align*}
    For $a \geq 0$ we obtain
    \begin{equation} \label{eq:Poisson_deriv}
    \begin{aligned}
        E_{X,\twist}(s&+a) \left( \Delta_{X,\twist} E_{X,\twist}(1-s)\right)^T 
        - \left( \Delta_{X,\twist} E_{X,\twist}(s+a) \right) E_{X,\twist}(1-s)^T \\
        &= \left( a(2s-1) + a^2\right) E_{X,\twist}(s+a) E_{X,\twist}(1-s)^T\,.
    \end{aligned}
    \end{equation}
    Using Green's formula we obtain
    \begin{align*}
        - (2&s-1) \int_{\pa X_\eps} \int_{\pa X} E_{X,\twist}(s;z,\theta') E_{X,\twist}(1-s;z,\theta')^T\, d\mu_{\pa_\infty X}(\theta') \,d\sigma_\eps(z) \\
        &= \lim_{a \to 0} \frac{1}{a} \int_{\pa X_\eps} \int_{\pa X} 
        \bigg( E_{X,\twist}(s+a;z,\theta') \pa_\nu E_{X,\twist}(1-s;z,\theta')^T \\
        &\phantom{ = \lim_{a \to 0} \frac{1}{a} } - \pa_\nu E_{X,\twist}(s+a;z,\theta') E_{X,\twist}(1-s;z,\theta')^T \bigg)\, d\mu_{\pa X}(\theta') \, d\sigma_\eps(z)\,.
    \end{align*}
    Applying the same argument for $a = 0$ in \eqref{eq:Poisson_deriv} we achieve that
    \begin{align*}
        0 &= \int_{\pa X_\eps} \int_{\pa X} \bigg( E_{X,\twist}(s;z,\theta') \pa_\nu E_{X,\twist}(1-s;z,\theta')^T \\
        &\phantom{ = \int_{\pa X_\eps} \int_{\pa X}} - \pa_\nu E_{X,\twist}(s;z,\theta') E_{X,\twist}(1-s;z,\theta')^T \bigg)\, d\mu_{\pa X}(\theta') \, d\sigma_\eps(z)\,.
    \end{align*}
    Therefore we can evaluate the limit of $a \to 0$ as
    \begin{align*}
        - (2&s-1) \int_{\pa X_\eps} \int_{\pa X} E_{X,\twist}(s;z,\theta') E_{X,\twist}(1-s;z,\theta')^T\, d\mu_{\pa_\infty X}(\theta') \,d\sigma_\eps(z) \\
        &= \int_{\pa X_\eps} \int_{\pa X} \bigg( \pa_s E_{X,\twist}(s;z,\theta') \pa_\nu E_{X,\twist}(1-s;z,\theta')^T \\
        &\phantom{ = \int_{\pa X_\eps} \int_{\pa X}} - \pa_s \pa_\nu E_{X,\twist}(s;z,\theta') E_{X,\twist}(1-s;z,\theta')^T \bigg)\, d\mu_{\pa X}(\theta') \, d\sigma_\eps(z) \\
        &= I_{X,\twist}(s; \eps)\,,
    \end{align*}
    and consequently we obtain
    \begin{align*}
        \Upsilon_{X,\twist}(s) &= - (2s-1)^2 \zTr\left( E_{X,\twist}(s) E_{X,\twist}(1 - s)^T \right) \\
        &= (2s- 1) \FP_{\eps \to 0} \tr_{\bundle} I_{X,\twist}(s; \eps)\,.
    \end{align*}
    Analogously we deduce that
    \begin{align*}
        \Upsilon_{X_f,\twist}(s) &= (2s - 1) \FP_{\eps \to 0} \tr_{\bundle} I_{X_f,\twist}(s, \eps)\,.
    \end{align*}
\end{proof}
\begin{proof}[Proof of Proposition~\ref{prop:upsilon-tau}]
     By Lemma~\ref{lem:Upsilon_FP}
    \begin{align*}
        \Upsilon_{X,\twist}(s) - \Upsilon_{X_f,\twist}(s) = (2s-1) \FP_{\eps \to 0} \tr_{\bundle} \left( I_{X,\twist}(s, \eps) - I_{X_f,\twist}(s, \eps)\right)\,.
    \end{align*}
    We use the splitting
    \begin{align}\label{eq:sum_Upsilon}
        \Upsilon_{X,\twist}(s) - \Upsilon_{X_f,\twist}(s) &=
        \Upsilon_{X,\twist}^{\ff}(s) +
        \Upsilon_{X,\twist}^{\cf}(s) +
        \Upsilon_{X,\twist}^{\fc}(s) +
        \Upsilon_{X,\twist}^{\cc}(s)
    \end{align}
    where
    \begin{align*}
        \Upsilon_{X,\twist}^{\ff}(s) &\coloneqq (2s-1) \FP_{\eps \to 0} \tr_{\bundle} \left( I_{X,\twist}^{\ff}(s, \eps) - I_{X_f,\twist}(s, \eps)\right)\,,\\
        \Upsilon_{X,\twist}^{\cf}(s) &\coloneqq (2s-1) \FP_{\eps \to 0} \tr_{\bundle} I_{X,\twist}^{\cf}(s, \eps)\,,\\
        \Upsilon_{X,\twist}^{\fc}(s) &\coloneqq (2s-1) \FP_{\eps \to 0} \tr_{\bundle} I_{X,\twist}^{\fc}(s, \eps)\,,\\
        \Upsilon_{X,\twist}^{\cc}(s) &\coloneqq (2s-1) \FP_{\eps \to 0} \tr_{\bundle} I_{X,\twist}^{\cc}(s, \eps)\,.
    \end{align*}
    We calculate each term in \eqref{eq:sum_Upsilon} separately.

    \medskip

    \noindent
    \textbf{Funnel-funnel term.}
    Let $i,j \in \{1, \dotsc, n_f\}$. The measure $d\sigma_\eps$ restricted to $X_{f,j,\eps}$ is given by
    \begin{align}\label{eq:dsigma}
        d\sigma_\eps = \frac1\eps \frac{\ell_j \cdot d\phi}{2\pi}\,,
    \end{align}
    where $\ell_j \in (0,\infty)$ is the length of the periodic geodesic associated to the funnel end~$X_{f,k}$.
    Moreover, $\pa_\nu = -\rho \pa_\rho + O(\rho^2)$.

    For $z \in X_{f,j,\eps}$ and $\eps \in (0,\tfrac12)$ we have
    \begin{align*}
        E_{X,\twist}^{f,i}(s;z,\phi') &= \frac{\rho^{1-s}}{2s-1} \delta_{ij} \delta(\phi - \phi') + \frac{\rho^s}{2s-1} S_{X,\twist}(s;\phi,\phi')^{\ff}_{ij} + O\left(\rho^{\frac32}\right)\,,
    \end{align*}
    where $S_{X,\twist}(s;\phi,\phi')^{\ff}_{ij}$ denotes the funnel-funnel contribution at $(s;\phi,\phi')$ of~$(X,_{f,j},\twist)$ to $(X_{f,i},\twist)$.
    We need to take into account the additional terms coming from~$I^f_j(s,\eps)$.
    The model Poisson operator is given by
    \begin{align*}
        E_{X_{f,j},\twist}(s;z,\phi') &= \frac{\rho^{1-s}}{2s-1} \delta_{ij} \delta(\phi - \phi') + \frac{\rho^s}{2s-1} \delta_{ij} S_{X_{f,j},\twist}(s;\phi,\phi') + O\left(\rho^{\frac32}\right)\,.
    \end{align*}
    We first note that the remainder terms do not contribute to the finite part. Hence we may replace $E_{X,\twist}^{f,i}(s;z,\phi')$ by
    \begin{align*}
        \tilde{E}_{X,\twist}^{f,i}(s;z,\phi') &= \frac{\rho^{1-s}}{2s-1} \delta_{ij} \delta(\phi - \phi') + \frac{\rho^s}{2s-1} S_{X,\twist}(s;\phi,\phi')^{\ff}_{ij}\,,
    \end{align*}
    and we may replace $E_{X_{f,i},\twist}(s;z,\phi')$ by
    \begin{align*}
        \tilde{E}_{X_{f,i},\twist}(s;z,\phi') &= \frac{\rho^{1-s}}{2s-1} \delta_{ij} \delta(\phi - \phi') + \frac{\rho^s}{2s-1} \delta_{ij} S_{X_{f,j},\twist}(s;\phi,\phi')\,.
    \end{align*}
    From this and \eqref{eq:dsigma} we conclude that the funnel-funnel contribution of the difference $I_{X,\twist}(s;\eps) - I_{X_f,\twist}(s;\eps)$ is of the form
    \begin{align*}
        I_{ij}^{\ff}(s;\eps) - \delta_{ij} I_{f,j}(s;\eps) = \log(\eps) \left( \dotsc \right) + \tilde{a} + \eps^{1-2s} \tilde{b} + \eps^{2s-1} \tilde{c}\,.
    \end{align*}
    The $\log(\eps)$ term does not contribute to the finite part by definition.
    Since the finite part $\FP_{\eps\to 0} \tr_{\bundle} (I_{X,\twist}(s;\eps) - I_{X_f,\twist}(s;\eps))$ is defined in the sense of distributions, we can use integration by parts away from $s = \tfrac12$
    to obtain that the terms $\tilde{b}$ and $\tilde{c}$ do not contribute to the finite part.
    Ignoring $\log(\rho)$ terms, we calculate
    \begin{align*}
        \pa_s \tilde{E}_{X,\twist}^{f,i}(s;z,\phi') &= (2s-1)^{-2} \Big( -2 \rho^{1-s} \delta_{ij} \delta(\phi-\phi') \\
        &\qquad \phantom{(2s-1)^{-2} \Big(} - 2\rho^s S_{X,\twist}(s;\phi,\phi')^{\ff}_{ij} \\
        &\qquad \phantom{(2s-1)^{-2} \Big(} + (2s-1) \rho^s \pa_s S_{X,\twist}(s;\phi,\phi')^{\ff}_{ij} \Big) + \dotsc \\
        \pa_s \pa_\nu \tilde{E}_{X,\twist}^{f,i}(s;z,\phi') &= (2s-1)^{-2} \Big( 2(1-s) \rho^{1-s} \delta_{ij} \delta(\phi-\phi') \\
        &\qquad \phantom{(2s-1)^{-2} \Big(} + 2s \rho^s S_{X,\twist}(s;\phi,\phi')^{\ff}_{ij} \\
        &\qquad \phantom{(2s-1)^{-2} \Big(} - (2s-1) s \rho^s \pa_s S_{X,\twist}(s;\phi,\phi')^{\ff}_{ij} \Big) + \dotsc \\
        \left(\tilde{E}_{X,\twist}^{f,i}\right)^T(1-s;z,\phi') &= (2s-1)^{-1} \Big( - \rho^{s} \delta_{ij} \delta(\phi-\phi') \\
        &\qquad \phantom{(2s-1)^{-1} \Big(} - \rho^{1-s} S_{X,\twist}^T(1-s;\phi,\phi')^{\ff}_{ij}\Big) \\
        \pa_\nu \left(\tilde{E}_{X,\twist}^{f,i}\right)^T(1-s;z,\phi') &= (2s-1)^{-1} \Big( s \rho^{s} \delta_{ij} \delta(\phi-\phi') \\
        &\qquad \phantom{(2s-1)^{-2} \Big(} + (1-s)\rho^{1-s} S_{X,\twist}^T(1-s;\phi,\phi')^{\ff}_{ij}\Big) + \dotsc
    \end{align*}
    and similarly for $\tilde{E}_{X_{f,j},\twist}(s;z,\phi')$.
    Note that terms involving only $\delta(\phi - \phi')$ cancel out with a term of the same form coming from the model funnel. Hence we obtain that
    \begin{align*}
        I_{ij}^{\ff}(s;\eps) - \delta_{ij} I_{f,j}(s;\eps)
        &= \int_0^{2\pi} \int_{0}^{2\pi} \left( a_{ij}^{\ff}(s;\phi,\phi') - \delta_{ij} a_j^f(s;\phi,\phi') \right) \frac{\ell_i \cdot d\phi'}{2\pi} \frac{\ell_j \cdot d\phi}{2\pi} \\
        &\phantom{=}\ + \log(\eps) \int_0^{2\pi} \int_{0}^{2\pi} \dotsc \frac{\ell_i \cdot d\phi'}{2\pi} \frac{\ell_j \cdot d\phi}{2\pi} \\
        &\phantom{=}\ + \eps^{1-2s} \int_0^{2\pi} \int_{0}^{2\pi} \dotsc \frac{\ell_i \cdot d\phi'}{2\pi} \frac{\ell_j \cdot d\phi}{2\pi} \\
        &\phantom{=}\ + \eps^{2s-1} \int_0^{2\pi} \int_{0}^{2\pi} \dotsc \frac{\ell_i \cdot d\phi'}{2\pi} \frac{\ell_j \cdot d\phi}{2\pi}
        + O(\eps \log(\eps)) \,,
    \end{align*}
    where
    \begin{align*}
        a_{ij}^{\ff}(s;\phi,\phi') &= (2s-1)^{-3} \big( - 2 S_{X,\twist}(s;\phi,\phi')^{\ff}_{ij} (1-s) S_{X,\twist}^T(1-s;\phi,\phi')^{\ff}_{ij} \\
        &\phantom{= (2s-1)^{-3} \big( } + (2s-1) \pa_s S_{X,\twist}(s;\phi,\phi')^{\ff}_{ij} (1-s) S_{X,\twist}^T(1-s;\phi,\phi')^{\ff}_{ij} \\
        &\phantom{= (2s-1)^{-3} \big( } + 2s S_{X,\twist}(s;\phi,\phi')^{\ff}_{ij} S_{X,\twist}^T(1-s;\phi,\phi')^{\ff}_{ij} \\
        &\phantom{= (2s-1)^{-3} \big( } - (2s-1) s \pa_s S_{X,\twist}(s;\phi,\phi')^{\ff}_{ij} S_{X,\twist}^T(1-s;\phi,\phi')^{\ff}_{ij} \;\big) \\
        &= 2(2s-1)^{-2} S_{X,\twist}(s;\phi,\phi')^{\ff}_{ij} S_{X,\twist}^T(1-s;\phi,\phi')^{\ff}_{ij} \\
        &\phantom{=}\ - (2s-1)^{-1} \pa_s S_{X,\twist}(s;\phi,\phi')^{\ff}_{ij} S_{X,\twist}^T(1-s;\phi,\phi')^{\ff}_{ij}\,.
    \end{align*}
    and
    \begin{align*}
        a_j^f(s;\phi,\phi')
        &= 2(2s-1)^{-2} S_{X_{f,j},\twist}(s;\phi,\phi') S_{X_{f,j},\twist}^T(1-s;\phi,\phi') \\
        &\phantom{=}\ - (2s-1)^{-1} \pa_s S_{X_{f,j},\twist}(s;\phi,\phi') S_{X_{f,j},\twist}^T(1-s;\phi,\phi')\,.
    \end{align*}
    Therefore, we obtain that
    \begin{align*}
        \Upsilon_{X,\twist}^{\ff}(s) &= (2s-1) \FP_{\eps \to 0} \left( I_{ij}^{\ff}(s;\eps) - \delta_{ij} I_{f,j}(s;\eps) \right) \\
        &=- \tr_{\bundle} \Big( \pa_s S_{X,\twist}^{\ff}(s) S_{X,\twist}^{\ff}(1-s)^T - \pa_s S_{X_f,\twist}(s) S_{X_f,\twist}(1-s)^T \Big) \\
        &\phantom{=}\ + \frac{2}{2s-1} \tr_{\bundle} \Big( S_{X,\twist}^{\ff}(s) S_{X,\twist}^{\ff}(1-s)^T - \id_{\pa_\infty X_f} \Big) \,,
    \end{align*}
    where we used that $S_{X_f,\twist}(s) S_{X_f,\twist}(1-s) = \id_{\pa_\infty X_f}$ and $S_{X_f,\twist}(s)^T = S_{X_f,\twist}(s)$.

    \medskip
    \noindent
    \textbf{Funnel-cusp term.}
    The Poisson operator is given by
    \begin{align*}
        E_{X,\twist}^{c,k}(s; z, \phi') &= \frac{\rho^{s}}{2s-1} S_{X,\twist}^{\fc}(s;\phi,\phi') + O(\rho^{3/2})\,,
    \end{align*}
    where $z = (\rho,\phi) \in X_{c,\eps}$.
    Analogously to the funnel-funnel case we obtain that
    \begin{align*}
        \Upsilon_{X,\twist}^{\fc}(s) &= (2s-1) \FP_{\eps \to 0} \sum_{j,k} I_{jk}^{\fc}(s;\eps) \\
        &=- \tr_{\bundle} \Big( \pa_s S_{X,\twist}^{\fc}(s) S_{X,\twist}^{\cf}(1-s)^T \Big) \\
        &\phantom{=}\ + \frac{2}{2s-1} \tr_{\bundle} \Big( S_{X,\twist}^{\fc}(s) S_{X,\twist}^{\cf}(1-s)^T \Big) \,.
    \end{align*}

    \medskip

    \noindent
    \textbf{Cusp-funnel term.}
    For the cusp-funnel terms, the Poisson operator is given by
    \begin{align*}
        E_{X,\twist}^{f,k}(s; z, \phi') &= \frac{\rho^{s-1}}{2s-1} S_{X,\twist}^{\cf}(s;\phi,\phi') + O(\rho^{\tfrac12})\,,
    \end{align*}
    where $z = (\rho,\phi) \in X_{c,\eps}$.
    Analogously to the previous two cases, we obtain that
    \begin{align*}
        \Upsilon_{X,\twist}^{\cf}(s) &= (2s-1) \FP_{\eps \to 0} \sum_{j,k} I_{jk}^{\cf}(s;\eps) \\
        &=- \tr_{\bundle} \Big( \pa_s S_{X,\twist}^{\cf}(s) S_{X,\twist}^{\fc}(1-s)^T \Big) \\
        &\phantom{=}\ + \frac{2}{2s-1} \tr_{\bundle} \Big( S_{X,\twist}^{\cf}(s) S_{X,\twist}^{\fc}(1-s)^T \Big) \,.
    \end{align*}
    
    \medskip

    \noindent
    \textbf{Cusp-cusp term.}
    For $j \in \{1,\dotsc, n_c\}$ we let $p_j$ be a primitive parabolic elements associated to the cusp end~$X_{c,j}$, and we set $m_j\coloneqq m_{\twist(p_j)}(1)$.

    For the cusp-cusp terms, the Poisson operator takes the form
    \begin{align*}
        E_{X,\twist}^{c,k}(s;z,\phi') &= (2s-1)^{-1} \left( \delta_{j,k} \id_{\C^{m_j}} \rho^{-s} + \rho^{s-1} S^{\cc}_{jk}(s) \right) + O(\rho^{\infty})\,,
    \end{align*}
    where $z = (\rho,\phi) \in X_{c,j}$ and $\phi' \in \pa_\infty X_{c,k}$. We note that the scattering matrix~$S^{\cc}_{jk}(s)$ is a $m_j \times m_k$ matrix.

    Straightforward calculations yield
    \begin{align*}
        \pa_s E_{X,\twist}^{c,k}(s;z,\phi') &= (2s-1)^{-2} \Big( -2\delta_{j,k} \id_{\C^{m_j}} \rho^{-s} + -2\rho^{s-1} S^{\cc}_{jk}(s) \\
        &\phantom{= (2s-1)^{-2} \Big( } + (2s-1) \rho^{s-1} \pa_s S^{\cc}_{jk}(s) \Big) + \dotsc\,,\\
        \pa_s \pa_\nu E_{X,\twist}^{c,k}(s;z,\phi') &= (2s-1)^{-2} \Big( 2s \delta_{j,k} \id_{\C^{m_j}} \rho^{-s} + 2(s-1)\rho^{s-1} S^{\cc}_{jk}(s) \\
        &\phantom{= (2s-1)^{-2} \Big( } - (2s-1) (s-1) \rho^{s-1} \pa_s S^{\cc}_{jk}(s) \Big) + \dotsc\,,\\
        E_{X,\twist}^{c,k}(1-s;z,\phi')^T &= (2s-1)^{-1} \left( -\delta_{j,k} \id_{\C^{m_j}} \rho^{s-1} -\rho^{-s} S^{\cc}_{jk}(1-s)^T \right) \\
        &\phantom{=}\ + O(\rho^{\infty})\,,\\
        \pa_\nu E_{X,\twist}^{c,k}(1-s;z,\phi')^T &= (2s-1)^{-1} \left( (s-1)\delta_{j,k} \id_{\C^{m_j}} \rho^{s-1} -s\rho^{-s} S^{\cc}_{jk}(1-s)^T \right) \\
        &\phantom{=}\ + O(\rho^{\infty})\,.
    \end{align*}
    By the same argument as for the funnel-funnel case, products of $\rho^{s-1}$ with itself and products of $\rho^{-s}$ with itself do not contribute to $\Upsilon_{X,\twist}^{\cc}(s)$. Therefore we arrive at
    \begin{align*}
        \Upsilon_{X,\twist}^{\cc}(s) &= (2s-1) \FP_{\eps \to 0} \sum_{j,k} I_{jk}^{\cc}(s;\eps) \\
        &=- \tr_{\bundle} \Big( \pa_s S_{X,\twist}^{\cc}(s) S_{X,\twist}^{\cc}(1-s)^T \Big) \\
        &\phantom{=}\ + \frac{2}{2s-1} \tr_{\bundle} \Big( S_{X,\twist}^{\cc}(s) S_{X,\twist}^{\cc}(1-s)^T - \id_{\pa_\infty X_c} \Big) \,.
    \end{align*}

    \medskip

    \noindent
    \textbf{Putting everything together.}
    We have that
    \begin{align*}
        S_{X,\twist}(1-s) S_{X,\twist}(s) = \id
    \end{align*}
    and therefore
    \begin{align*}
        0 &= \tr_{\bundle} \left( S_{X,\twist}^{\ff}(s) S_{X,\twist}^{\ff}(1-s)^T - \id_{\pa_\infty X_f} + S_{X,\twist}^{\fc}(s) S_{X,\twist}^{\cf}(1-s)^T \right) \,,\\
        &\phantom{=}\ + \tr_{\bundle} \left( S_{X,\twist}^{\cc}(s) S_{X,\twist}^{\cc}(1-s)^T - \id_{\pa_\infty X_c} + S_{X,\twist}^{\cf}(s) S_{X,\twist}^{\fc}(1-s)^T \right) \,.
    \end{align*}
    This implies that
    \begin{align*}
        \Upsilon_{X,\twist}(s) - \Upsilon_{X_f,\twist}(s) &=
        \Upsilon_{X,\twist}^{\ff}(s) +
        \Upsilon_{X,\twist}^{\cf}(s) +
        \Upsilon_{X,\twist}^{\fc}(s) +
        \Upsilon_{X,\twist}^{\cc}(s) \\
        &=- \tr_{\bundle} \Big( \pa_s S_{X,\twist}^{\ff}(s) S_{X,\twist}^{\ff}(1-s)^T \\
        &\phantom{= - \tr_{\bundle} \Big(}- \pa_s S_{X_f,\twist}(s) S_{X_f,\twist}(1-s)^T \Big) \\
        &\phantom{=}\ - \tr_{\bundle} \Big( \pa_s S_{X,\twist}^{\fc}(s) S_{X,\twist}^{\cf}(1-s)^T \Big) \\
        &\phantom{=}\ - \tr_{\bundle} \Big( \pa_s S_{X,\twist}^{\cf}(s) S_{X,\twist}^{\fc}(1-s)^T \Big) \\
        &\phantom{=}\ - \tr_{\bundle} \Big( \pa_s S_{X,\twist}^{\cc}(s) S_{X,\twist}^{\cc}(1-s)^T \Big) \,.
    \end{align*}
    It is now straightforward to check that this is $- (\tau_{X,\twist}'/\tau_{X,\twist})(s)$, which yields
    the equality \eqref{eq:upsilon-tau}.
\end{proof}

\section{Proof of the main theorem}\label{sec:proof}

In this section we will first establish the meromorphy of the Selberg zeta function and determine its set of poles. Then we will discuss the relationship between the Selberg zeta function and the scattering determinant. After some additional considerations needed for proving the claimed properties of the function~$p$ in the statement of Theorem~\ref{thm:selberg-factorization}, we will finally present a proof of Theorem~\ref{thm:selberg-factorization}.

As we established Theorem~\ref{thm:selberg-factorization} already in Section~\ref{sec:cyclic} for the case that $\group$ is cyclic, we will suppose throughout this section that $\group$ is non-cyclic. As before, we let $X=\group\backslash\h$ denote the associated hyperbolic orbisurface. We recall further that we have shown already with~\eqref{eq:barnes_orbifold} that \eqref{eq:SZF_fact_orb} and \eqref{eq:SZF_fact_top} in Theorem~\ref{thm:selberg-factorization} are equal.

Throughout this section we will use for brevity (albeit a slight abuse of notion)
\begin{align*}
    \frac{d}{ds} \log f(s) = \frac{f'(s)}{f(s)}\,,
\end{align*}
where the right hand side is well-defined irrespective of whether $\log f(s)$ is well-defined or not in the classical sense.

We start by determining the
poles of the logarithmic derivative of the Selberg zeta function with a proof based on
Lemma~\ref{lem:decomp-phi}.

\begin{prop}\label{prop:logZeta-poles}
    The logarithmic derivative~$Z_{X,\twist}'/Z_{X,\twist}$ of the Selberg zeta function~$Z_{X,\twist}$ is meromorphic on $\{s\in\C : \Rea s > \tfrac12\}$. Its poles in this domain are at $s = s_0$ whenever $s_0(1- s_0) \in \sigma_d(\LapTwist)$. Each such pole is simple; its residue is given by the multiplicity of the eigenvalue $s_0(1-s_0)$ of $\LapTwist$.
    Further, the only pole of~$Z_{X,\twist}'/Z_{X,\twist}$ on $\{s\in\C : \Rea s = \tfrac12\}$ is at $s = \tfrac12$. It is simple and its residue is $m_{X,\twist}(\tfrac12) - n_p$.
\end{prop}
\begin{proof}
    By Lemma~\ref{lem:decomp-phi},
    \begin{align*}
        \frac{Z_{X,\twist}'(s)}{Z_{X,\twist}(s)} = \Phi_{X,\twist}(s) - n_p \Phi_{C_\infty}(s) + H(s)\,,
    \end{align*}
    where $H$ is holomorphic on~$\{\Rea s > 0\}$.
    This implies in particular that $\pa_s \log Z_{X,\twist}$ is meromorphic on this domain.

    On $\{\Rea s > \tfrac12\}$, the function~$\Phi_{C_\infty}$ is holomorphic. Therefore it
    suffices to consider $\Phi_{X,\twist}$. The residues of $\Phi_{X,\twist}$ on $\{\Rea s \geq \tfrac12\}$ are provided by Lemma~\ref{lem:Phi_cont}.
    The residue of $\Phi_{X,\twist}$ at $s = \tfrac12$ is
    \begin{align*}
        \res_{s=\frac12} \Phi_{X,\twist}(s) &= m_{X,\twist}\left(\frac12\right) - \frac{n_p}{2}\,.
    \end{align*}
    Further, $\Phi_{C_\infty}$ has a pole at $s = \tfrac12$ of order~$1$ with residue~$\tfrac12$ by Lemma~\ref{lem:phi-cusp}. From this it follows that
    \begin{align*}
        \res_{s=\frac12} \frac{Z_{X,\twist}'}{Z_{X,\twist}}(s) = m_{X,\twist}\left(\frac12\right) - n_p\,.
    \end{align*}
    This completes the proof.
\end{proof}

With the following proposition we provide a relation between the scattering determinant and the Selberg zeta function.

\begin{prop}\label{prop:det-selberg}
    The Selberg zeta function $Z_{X,\twist}$ admits a meromorphic continuation to all of~$\C$. It satisfies the identity
    \begin{align*}
        \tau_{X,\twist}(s) &= e^{p(s)}
        \frac{G_\infty(s)^{-\dim(V)\, \eulertop(X)}}{G_\infty(1-s)^{-\dim(V)\, \eulertop(X)}}
        \frac{G_{X^\wedge,\twist}(s)}{G_{X^\wedge,\twist}(1-s)}
        \frac{Z_{X,\twist} (1-s)}{Z_{X,\twist}(s)}\\
        &\phantom{=} \cdot \frac{Z_{X_f,\twist}(s)}{Z_{X_f,\twist}(1-s)}
        \frac{\gammafunc\left(s-\frac12\right)^{n_p}}{\gammafunc\left(\frac12-s\right)^{n_p}}\,,
    \end{align*}
    on all of~$\C$. Here, $p(s) = -s\thin n_p\log4 + c$ for some $c \in \C$.
\end{prop}

\begin{proof}
    Let $s \in \C$ with $\Rea s = \tfrac12$ and $s \not = \tfrac12$.
    By Proposition~\ref{prop:upsilon-tau} and \eqref{eq:upsilon-Xf},
    \begin{align*}
        \Upsilon_{X,\twist}(s) &= \frac{d}{ds} \log\left( \frac{1}{\tau_{X,\twist}(s)} \frac{Z_{X_f,\twist}(s)}{Z_{X_f,\twist}(1-s)} \right)\,.
    \end{align*}
    By Proposition~\ref{prop:upsilon-zeta},
    \begin{align*}
        \Upsilon_{X,\twist}(s)
        &= \frac{Z_{X,\twist}'(s)}{Z_{X,\twist}(s)} + \frac{Z_{X,\twist}'(1-s)}{Z_{X,\twist}(1-s)} \\
        &\phantom{=} + \frac{d}{ds} \log \left( e^{p(s)} \frac{\gammafunc\left(\frac12-s\right)^{n_p}}{\gammafunc\left(s-\frac12\right)^{n_p}}
        \frac{G_{X^\wedge,\twist}(1-s)}{G_{X^\wedge,\twist}(s)} \frac{G_\infty(1-s)^{-\dim(V)\,\eulertop(X)}}{G_\infty(s)^{-\dim(V)\,\eulertop(X)}} \right)\,.
    \end{align*}
    Thus, with
    \begin{align*}
        f(s) & \coloneqq
        e^{p(s)} \tau_{X,\twist}(s) \frac{Z_{X_f,\twist}(1-s)}{Z_{X_f,\twist}(s) } \frac{\gammafunc\left(\frac12-s\right)^{n_p}}{\gammafunc\left(s-\frac12\right)^{n_p}}
         \\
         &\phantom{=}\quad \cdot
        \frac{G_{X^\wedge,\twist}(1-s)}{G_{X^\wedge,\twist}(s)} \frac{G_\infty(1-s)^{-\dim(V)\, \eulertop(X)}}{G_\infty(s)^{-\dim(V)\, \eulertop(X)}}
    \end{align*}
    we have
    \begin{equation}\label{eq:func_eq-zeta}
        \frac{Z_{X,\twist}'(s)}{Z_{X,\twist}(s)} + \frac{Z_{X,\twist}'(1-s)}{Z_{X,\twist}(1-s)} = - \frac{d}{ds} \log f(s)\,.
    \end{equation}
    As $f$ is meromorphic on~$\C$, its logarithmic derivative is meromorphic on all of~$\C$ as well, and every pole is simple with integer residue.
    Since the logarithmic derivative of $Z_{X,\twist}$ is meromorphic by Lemma~\ref{lem:decomp-phi}, the identity theorem of complex analysis implies that \eqref{eq:func_eq-zeta} is valid on all of~$\C$ as an equality of meromorphic functions.

    By Proposition~\ref{prop:logZeta-poles}, each pole of~$Z_{X,\twist}'/Z_{X,\twist}$ in $\{\Rea s \geq \tfrac12\}$ is simple and has an integer residue. Combination with~\eqref{eq:func_eq-zeta} yields that this holds on all of~$\C$.
    Therefore, $Z_{X,\twist}'/Z_{X,\twist}$ can be integrated to a meromorphic function $Z_{X,\twist}$ on all of~$\C$. Moreover, this establishes the  claimed equation between $Z_{X,\twist}$ and $\tau_{X,\twist}$.

    We emphasize that the polynomial in Proposition~\ref{prop:det-selberg} is not the same as in Proposition~\ref{prop:upsilon-zeta}
    since the integration yields a factor $e^{c}$ for some $c \in \C$.
\end{proof}

To establish that the function~$p$ in Theorem~\ref{thm:selberg-factorization} is indeed a polynomial of degree at most~$2$, we will require the bound on the regularized trace of the resolvent shown in Lemma~\ref{lem:growth}. To that end let
\begin{align*}
    \mathcal{B} \coloneqq B_{1}\left(\tfrac12\right) \cup \bigcup_{\zeta \in \ResSet_{X,\twist}} B_{\ang{\zeta}^{-(2+\delta)}}(\zeta)\,,
\end{align*}
where $B_r(z)$ denotes the open Euclidean ball in~$\C$ of radius~$r$ with center~$z$.

\begin{lemma}\label{lem:growth}
    For any $\kappa > 0$ and $s \in \C$ with $\abs{\Rea s - \tfrac12} \leq \eps$ and $s \not \in \mathcal{B}$, we have
    \begin{align*}
        \log \abs{\Phi_{X,\twist}(s)} = O(\ang{s}^{2 + \kappa})\,.
    \end{align*}
\end{lemma}
\begin{proof}
    We recall from \cite[Proof of Theorem 4.1]{DFP2} that the resolvent can be written as $\ResTwist(s) = M(s) (\id + K(s))$ with $M$ from~\eqref{def:M} and $K$ suitable. Further, we pick $s_0 \in \C$ with sufficiently large real part.
    We have
    \begin{align*}
        \frac{1}{2s-1} \left( \Phi_{X,\twist}(s) - \Phi_{X,\twist}(s_0) \right)
        &= \zTr\left(\ResTwist(s) - \ResTwist(s_0)\right) \\
        &\phantom{=}\quad - \zTr \left(R_\h(s) - R_\h(s_0)\right) \\
        &= \zTr( M(s) - M(s_0) )\\
        &\phantom{=}\quad + \zTr\left( (M(s) - M(s_0)) K(s)\right) \\
        &\phantom{=}\quad + \zTr\left( M(s_0) (K(s) - K(s_0))\right) \\
        &\phantom{=}\quad + a(s,s_0) \dim(V) \zVol(X)\,,
    \end{align*}
    where $a(s,s_0) = (2\pi)^{-1} \left( \psi(s) - \psi(s_0) \right)$ by \eqref{eq:resolv_H_diag}.
    We have to find bounds on the three traces.

    \medskip

    \paragraph{\textbf{The first term}}
    Using $M(s) = M_i + M_f(s) + M_c(s)$ (see \cite[Section 4]{DFP2}) we obtain that
    \begin{align*}
        M(s) - M(s_0) = M_f(s) - M_f(s_0) + M_c(s) - M_c(s_0)\,.
    \end{align*}
    We consider first the funnel terms. To that end we write
    \begin{align*}
        M_f(s) - M_f(s_0) = (1 - \eta_{f,0}) \left( R_{X_f,\twist}(s) - R_{X_f,\twist}(s_0) \right) (1 - \eta_{f,1})\,.
    \end{align*}
    Let $R_f(s,s_0) \coloneqq R_{X_f,\twist}(s) - R_{X_f,\twist}(s_0)$. Then
    \begin{align*}
        \zTr\left(M_f(s) - M_f(s_0)\right)
        &= \zInt{X_f} R_f(s,s_0,z,z) \left(1 - \eta_{f,1}(z)\right) \,d\mu_X(z) \\
        &= \zInt{X_f} R_f(s,s_0,z,z) \,d\mu_X(z) \\
        &\phantom{=}\quad - \zInt{X_f} R_f(s,s_0,z,z) \eta_{f,1}(z)\,d\mu_X(z)\\
        &= \zInt{X_f} R_f(s,s_0,z,z) \,d\mu_X(z) \\
        &\phantom{=}\quad - \zInt{X_f} \eta_{f,2}(z) R_f(s,s_0,z,z) \eta_{f,1}(z)\,d\mu_X(z)\,.
    \end{align*}
    Since $\zInt{X_f} R_\h(s;z,z) d\mu_X(z) = 0$, we obtain that
    \begin{align*}
        \zTr\bigl(M_f(s) & - M_f(s_0)\bigr)
        \\
        &= \frac{1}{2s-1} \left(\Phi_{X_f,\twist}(s) - \Phi_{X_f,\twist}(s_0)\right)
           - \Tr\left( \eta_{f,2} R_f(s,s_0) \eta_{f,1} \right)\,.
    \end{align*}
    From Proposition~\ref{prop:funnel-zeta-deriv} we obtain a bound in the strip $\abs{\Rea s - \tfrac12} \leq \eps$.
    Using an integral representation of the hypergeometric function, \eqref{eq:resolv_H_int}, we obtain for $\Rea s > 0$ that
    the model cylinder resolvent can be written as
    \begin{align*}
        R_{C_\ell,\twist}(s;z,w) = \frac{1}{4\pi} \sum_{k\in \Z} \int_0^1 \twist(g_\ell)^k \frac{ (t (1-t))^{s-1}}{(\sigma(z,e^{k\ell}w) - t)^s} \,dt\,.
    \end{align*}
    We can hence apply the argument of the untwisted case (see Borthwick~\cite[p. 233]{Borthwick_book}) together with the unitarity of $\twist(g_\ell)$ to obtain the bound
    \begin{align*}
        \abs*{\Tr \left( \eta_{f,2} R_f(s,s_0) \eta_{f,1} \right)} \lesssim \ang{s}
    \end{align*}
    for $\abs{\Rea s - \tfrac12} \leq \eps$ and $\abs{s - \tfrac12} \geq 1$.

    Analogously, setting $R_c(s,s_0) \coloneqq R_{X_c,\twist}(s) - R_{X_c,\twist}(s_0)$, we obtain for the cusp end
    \begin{align*}
        \zTr\bigl(M_c(s) & - M_c(s_0)\bigr)
        \\
        &= \frac{1}{2s-1} \left(\Phi_{X_c,\twist}(s) - \Phi_{X_c,\twist}(s_0)\right)  - \Tr\left( \eta_{c,2} R_c(s,s_0) \eta_{c,1} \right)\,.
    \end{align*}
    The first two summands can be estimated using Lemma~\ref{lem:phi-cusp}.
    Using the asymptotic expansion of the Bessel functions (see \cite[10.40.6]{NIST:DLMF}),
    we see that, for fixed $s$ and $s_0$ and uniformly on compact subsets of~$\h$ for~$y$,
    \begin{align*}
        u_{\kappa}(s;y,y) - u_{\kappa}(s_0;y,y) = O(\kappa^{-2})
    \end{align*}
    as $\kappa \to +\infty$.
    Thus, the Fourier series of $R_{X_c,\twist}(s;z,z) - R_{X_c,\twist}(s_0;z,z)$ converges and
    hence we can use the estimates of Bessel functions to obtain the bound
    \begin{align*}
        \abs*{\Tr \left( \eta_{c,2} R_c(s,s_0) \eta_{c,1} \right)} \lesssim \ang{s}
    \end{align*}
    for $\abs{\Rea s - \tfrac12} \leq \eps$ and $\abs{s - \tfrac12} \geq 1$.

    Combining the model funnel and the model cusp terms gives
    \begin{align*}
        \abs*{\zTr \left( M(s) - M(s_0) \right)} \lesssim \ang{s}
    \end{align*}
    for $\abs{\Rea s - \tfrac12} \leq \eps$ and $\abs{s - \tfrac12} \geq 1$.

    \medskip

    \paragraph{\textbf{The second term}}
    For the term $\zTr\left( (M(s) - M(s_0)) K(s)\right)$
    we use \cite[(37)]{DFP2} and $\eta_3 L(s) = L(s)$ to see that
    \begin{align}\label{eq:KasL}
        K(s) = \eta_3 (\id - \eta_3 L(s))^{-1} L(s)\,.
    \end{align}
    It follows that
    \begin{align*}
        (M(s) - M(s_0)) K(s)= (M(s) - M(s_0)) \eta_3 (\id - \eta_3 L(s))^{-1} L(s)\,.
    \end{align*}
    Therefore
    \begin{align*}
        \rho^{-1} (M(s) - M(s_0)) \eta_3
    \end{align*}
    is trace-class for $\Rea s > -\tfrac12$ (recall that $\rho$ is the boundary defining function). Analogously as for the first term, we obtain the estimate
    \begin{align*}
        \norm{ \rho^{-1} (M(s) - M(s_0)) \eta_3}_{\Tr} \lesssim \ang{s}
    \end{align*}
    for $\abs{\Rea s - \tfrac12} \leq \eps$ and $\abs{s - \tfrac12} \geq 1$.
    The operator norm of $(\id - \eta_3 L(s))^{-1}$ was estimated in \cite[(85)]{DFP2} as
    \begin{align}\label{eq:Linverse-bound}
        \norm{ (\id - \eta_3 L(s))^{-1} } \lesssim e^{C \ang{s}^{2+\eps}}
    \end{align}
    for $s \not \in \mathcal{B}$, using \cite[Lemma~6.1]{DFP}.
    Taking advantage of the control of the operator norm by the Hilbert--Schmidt norm, we obtain from Lemma~\ref{lem:bound-Ls} that
    \begin{align*}
        \norm{ L(s) \rho} \lesssim \ang{s}^2\,.
    \end{align*}
    Combining these estimate, we arrive at
    \begin{align*}
        \abs{ \zTr&\left( (M(s) - M(s_0)) K(s)\right) } \\
        &\leq \norm{ \rho^{-1} (M(s) - M(s_0)) K(s) \rho}_{\Tr} \\
        &\leq \norm{ \rho^{-1} (M(s) - M(s_0)) \eta_3 \rho}_{\Tr} \, \norm{(\id - \eta_3 L(s))^{-1} } \, \norm{ L(s) \rho} \\
        &\lesssim e^{C\ang{s}^{2 + \eps}}\,.
    \end{align*}

    \medskip

    \paragraph{\textbf{The third term}}
    As for the second term, we again conjugate by $\rho$ to convert the zero-trace into an actual trace.
    From the model calculations in \cite{DFP} we see that
    \begin{align*}
        \rho^{-1} M(s_0) \eta_3
    \end{align*}
    has integrable kernel for $\Rea s_0 > \tfrac32$ and hence is a Hilbert--Schmidt operator.
    We remark that integrating the model terms amounts to \[\int_0^1 \rho^{2s_0-2} \frac{d\rho}{\rho}\] for the funnel
        and to \[\int_1^\infty y^{2s_0 -2} \frac{dy}{y^2}\] for the cusp.
    We have to show that
    \begin{align*}
        (\id - \eta_3 L(s))^{-1} L(s) \rho - (\id - \eta_3 L(s_0))^{-1} L(s_0) \rho
    \end{align*}
    is Hilbert--Schmidt as well and estimate its norm.
    The second summand poses no issues
    since we can estimate $L(s_0) \rho$ in the same way as $\rho^{-1} M(s_0) \eta_3$, and $(\id - \eta_3 L(s_0))^{-1}$ is bounded.
    By Lemma~\ref{lem:bound-Ls},
    \begin{align*}
        \norm{L(s) \rho}_{\HS} \lesssim \ang{s}^2
    \end{align*}
    for $\abs{\Rea s - \tfrac12} \leq \eps$ and $\abs{s - \tfrac12} > 1$. For $(\id - \eta_3 L(s))^{-1}$ we again use \eqref{eq:Linverse-bound}.
\end{proof}

With these preparations we can now prove the main theorem.

\begin{proof}[Proof of Theorem~\ref{thm:selberg-factorization}]
Because Theorem~\ref{thm:selberg-factorization} is shown in Section~\ref{sec:cyclic} for all cyclic~$\group$, we now suppose that $\group$ is non-cyclic. We further require that $X=\group\backslash\h$ has at least one funnel. Therefore $n_d=0$ and $X$ is of infinite area. Let $\twist \colon \group \to \Unit(V)$ be any finite-dimensional unitary representation of~$\group$.

Combining \cite[Theorem A]{DFP2} and Proposition~\ref{prop:det-selberg} yields
    \begin{equation}\label{eq:reflection}
        \begin{aligned}
            \frac{Z_{X,\twist}(s)}{Z_{X,\twist}(1-s)} &= e^{q(s)}
            \frac{G_\infty(s)^{- \dim(V)\, \eulertop(X)}}{G_\infty(1-s)^{-\dim(V)\, \eulertop(X)}}
            \frac{G_{X^\wedge,\twist}(s)}{G_{X^\wedge,\twist}(1 - s)} \\
            &\phantom{=}\quad \cdot \frac{\mc P_{X,\twist}(s)}{\mc P_{X,\twist}(1-s)}
            \frac{\gammafunc\left(s-\frac12\right)^{n_p}}{\gammafunc\left(\frac12-s\right)^{n_p}}
        \end{aligned}
    \end{equation}
    for some polynomial $q$ of degree 4.
    Note that all contributions from the funnels cancel due to factorization of
    the Selberg zeta function for the model funnel as provided by
    Proposition~\ref{prop:selberg-funnel}.

    We now mimick the reasoning of~\cite{BJP} and set
    \begin{align*}
        F(s) \coloneqq \frac{Z_{X,\twist}(s)}
        {G_\infty(s)^{-\dim(V)\, \eulertop(X)} G_{X^\wedge,\twist}(s)\ProdTwist(s) \gammafunc\left(s-\frac12\right)^{n_p} }\,.
    \end{align*}
    We note that for $\Rea s \geq \tfrac12$ the function $G_\infty(s)^{-\dim(V)\,\eulertop(X)}$
    is holomorphic with no zeros, $\gammafunc\left(s-\tfrac12\right)^{n_p}$ has a single pole, it is of
    order $n_p$ at $s = \tfrac12$, and $\ProdTwist(s)$ has poles of order $m_{X,\twist}(s)$ at
    $s(1-s) \in \sigma_d(\LapTwist)$.
    Using Proposition \ref{prop:logZeta-poles} we obtain that $F$ has
    no zeros or poles for $\Rea s \geq \tfrac12$.
    Reformulating \eqref{eq:reflection} in terms of $F(s)$ and $F(1-s)$, we observe
    the reflection formula $F(s) = e^{q(s)} F(1-s)$. This allows us to conclude that $F$ is an
    entire function without zeros. Therefore
    \begin{align}\label{eq:factorization-proof}
        Z_{X,\twist}(s) = e^{p(s)} G_\infty(s)^{-\dim(V)\,\eulertop(X)} G_{X^\wedge,\twist}(s)
        \gammafunc\left(s-\frac12\right)^{n_p} \cdot \ProdTwist(s)\,,
    \end{align}
    where $p$ is an entire function.

    It remains to show that $p$ is a
    polynomial of degree at most~$2$. To that end we use that $Z_{X,\twist}$ is bounded on $\{\Rea s > \delta\}$ ($\delta$ being the Hausdorff dimension of the limit set of~$\group$). Thus, there is $C>0$ such that
   \begin{align*}
        \abs{Z_{X,\twist}(s)} < C
    \end{align*}
    for all $s\in\C$ with $\Rea s > \delta$.
    Using that $G_\infty$, $G_{X^\wedge,\twist}$ and $\ProdTwist$ are of order $2$ (see Section~\ref{sec:special-functions}) and $1/\gammafunc$ is entire and of order $1$, we conclude that
    \begin{align*}
        \abs{p(s)} \leq C_\kappa \abs{s}^{2 + \kappa}\quad \text{for $\Rea(s) > \delta$}
    \end{align*}
    for some $\kappa > 0$.
    From the functional equation~\eqref{eq:reflection} we obtain the bound
    \begin{align*}
        \abs{p(s)} \leq C_\kappa \abs{s}^{4 + \kappa}\quad \text{for $\Rea(s) < 1-\delta$}\,.
    \end{align*}
    For $\Rea(s) \in [1-\delta,\delta]$ with $s \not \in \ResSet_{X,\twist}$, \eqref{eq:factorization-proof} and Lemma~\ref{lem:decomp-phi} imply that
    \begin{align*}
        p'(s) = \Phi_{X,\twist}(s) + \dim(V) \eulerorb(X) \frac{G_\infty'(s)}{G_\infty(s)} + f(s)\,,
    \end{align*}
    where
    $f$ is the logarithmic derivative of an entire function of order $2$.
    Since $G_\infty$ is entire of order $2$, Lemma~\ref{lem:growth} yields a bound for $p'$ that can be integrated to
    \begin{align*}
        \abs{p(s)} = O\left(e^{C\ang{s}^{2+\kappa}}\right)
    \end{align*}
    for $s \not \in \mathcal{B}$ and $\abs{\Rea s - \tfrac12} \leq \eps$.
    Since $p$ is entire, the bound is valid for all $s \in \C$ with $\abs{\Rea s - \tfrac12} \leq \eps$.
    From the Phragm\'en--Lindelöf theorem we obtain that
    \begin{align*}
        \abs{p(s)} \leq C_\kappa \abs{s}^{4 + \kappa}
    \end{align*}
    for all $s \in \C$. This shows that $p$ is a polynomial of degree at most $4$.

    In summary, we have
    \begin{align*}
        p(s) &= \log Z_{X,\twist}(s) + \dim(V)\, \eulertop(X) \log G_\infty(s) \\
        &\phantom{=}\ - \log G_{X,\twist}(s) - \log\ProdTwist(s) - n_p \log\gammafunc\left(s-\frac12\right)\,.
    \end{align*}
    For $\Rea s \to \infty$, the term $\log Z_{X,\twist}(s)$ decays exponentially and all other summands are bounded by $\ang{s}^{2+\eps}$
    for all $\eps > 0$. Therefore $p$ is a polynomial of degree at most~$2$.
\end{proof}

\printbibliography

@article{APW,
	AUTHOR = "A. Adam and A. Pohl and A. Wei{\ss}e",
	TITLE = "Zero is a resonance of every {S}chottky surface",
	NOTE = "arXiv:1808.09239",
	year = {2018}
}

@article{Barnes,
 author = {Barnes, E.},
 title = {The theory of the {{\(G\)}}-function.},
 fjournal = {The Quarterly Journal of Pure and Applied Mathematics},
 journal = {Quart. J.},
 volume = {31},
 pages = {264--314},
 year = {1900},
 language = {English},
 zbMATH = {2668382},
 JFM = {30.0389.02}
}

@book{Boas,
 author = {Boas, R.},
 title = {Entire functions},
 fseries = {Pure and Applied Mathematics (Academic Press)},
 series = {Pure Appl. Math., Academic Press},
 issn = {0079-8169},
 volume = {5},
 year = {1954},
 _publisher = {Academic Press, New York, NY},
 _language = {English},
 _keywords = {30-01,30-02,30D30},
 _zbMATH = {3093575},
 _Zbl = {0058.30201}
}

@Book{Borthwick_book,
    Author = {D. {Borthwick}},
    Title = {{Spectral theory of infinite-area hyperbolic surfaces}},
    Edition = {2nd edition},
    Pages = {xiii + 463},
    Year = {2016},
    Publisher = {Birkh\"auser/Springer},
}

@article {BJP,
    AUTHOR = {Borthwick, D. and Judge, C. and Perry, P.},
     TITLE = {Selberg's zeta function and the spectral geometry of
              geometrically finite hyperbolic surfaces},
   JOURNAL = {Comment. Math. Helv.},
  FJOURNAL = {Commentarii Mathematici Helvetici. A Journal of the Swiss
              Mathematical Society},
    VOLUME = {80},
      YEAR = {2005},
    NUMBER = {3},
     PAGES = {483--515},
}

@InCollection{Chang_Mayer_eigen,
	AUTHOR = {Chang, C. and Mayer, D.},
	TITLE = "Eigenfunctions of the transfer operators and the period functions for modular groups",
	_BOOKTITLE = "Dynamical, spectral, and arithmetic zeta functions ({S}an {A}ntonio, {TX}, 1999)",
	SERIES = "Contemp. Math.",
	VOLUME = "290",
	PAGES = "1--40",
	PUBLISHER = "Amer. Math. Soc.",
	ADDRESS = "Providence, RI",
	YEAR = "2001",
	_MRCLASS = "11F67 (11M36 37C30 37D40)",
	_MRNUMBER = "MR1868466 (2003h:11056)"
}

@misc{Napkin,
 author = {Chen, Evan},
 title = {An infinitely large napkin},
 note={https://github.com/vEnhance/napkin}
}

@misc{NIST:DLMF,
         key = "{\relax DLMF}",
       title = "{\it NIST Digital Library of Mathematical Functions}",
howpublished = "\url{https://dlmf.nist.gov/}, Release 1.2.7 of 2026-06-15",
         url = "https://dlmf.nist.gov/",
        note = "F.~W.~J. Olver, A.~B. {Olde Daalhuis}, D.~W. Lozier, B.~I. Schneider,
                R.~F. Boisvert, C.~W. Clark, B.~R. Miller, B.~V. Saunders,
                H.~S. Cohl, and M.~A. McClain, eds."
}

@article {DFP,
    AUTHOR = {Doll, M. and Fedosova, K. and Pohl, A.},
     TITLE = {Counting resonances on hyperbolic surfaces with unitary
              twists},
   JOURNAL = {Comm. Anal. Geom.},
  FJOURNAL = {Communications in Analysis and Geometry},
    VOLUME = {32},
      YEAR = {2024},
    NUMBER = {10},
     PAGES = {2805--2887},
      ISSN = {1019-8385,1944-9992},
   MRCLASS = {30F15 (20H10 30F35 57K20)},
  MRNUMBER = {4844965},
       _DOI = {10.4310/cag.241231033317},
       _URL = {https://doi.org/10.4310/cag.241231033317},
}

@article {DFP2,
    AUTHOR = {Doll, M. and Fedosova, K. and Pohl, A.},
     TITLE = {Scattering theory with unitary twists},
   JOURNAL = {J. Anal. Math.},
  FJOURNAL = {Journal d'Analyse Math\'ematique},
    VOLUME = {153},
      YEAR = {2024},
    NUMBER = {1},
     PAGES = {111--167},
      ISSN = {0021-7670,1565-8538},
   MRCLASS = {58J50},
  MRNUMBER = {4803436},
       _DOI = {10.1007/s11854-023-0313-0},
       _URL = {https://doi.org/10.1007/s11854-023-0313-0},
}

@book{ErdelyiI,
    AUTHOR = {Erd\'{e}lyi, A. and Magnus, W. and Oberhettinger, F.
              and Tricomi, F.},
     TITLE = {Higher transcendental functions. {V}ol. {I}},
      _NOTE = {Based on notes left by Harry Bateman,
              With a preface by Mina Rees,
              With a foreword by E. C. Watson,
              Reprint of the 1953 original},
 PUBLISHER = {Robert E. Krieger Publishing Co., Inc., Melbourne, Fla.},
      YEAR = {1981},
     PAGES = {xiii+302},
      ISBN = {0-89874-069-X},
}

@Article{FP_szf,
 Author = {Fedosova, K. and Pohl, A.},
 Title = {Meromorphic continuation of {Selberg} zeta functions with twists having non-expanding cusp monodromy},
 FJournal = {Selecta Mathematica. New Series},
 Journal = {Sel. Math., New Ser.},
 ISSN = {1022-1824},
 Volume = {26},
 Number = {1},
 Pages = {55},
 Note = {Id/No 9},
 Year = {2020},
 _Language = {English},
 _DOI = {10.1007/s00029-019-0534-3},
 _Keywords = {11M36,37C30,37D35},
 _zbMATH = {7161023},
 _Zbl = {1453.11118}
}

@Misc{Fischer,
    Author = {J. {Fischer}},
    Title = {{An approach to the Selberg trace formula via the Selberg zeta-function}},
    ISBN = {3-540-15208-3},
    Year = {1987},
    _Language = {English},
    HowPublished = {{Springer Lecture Notes in Mathematics, vol. 1253}},
    _MSC2010 = {11F72 58J50 11F12 11-02 11M36},
    _Zbl = {0618.10029}
}

@InCollection{Guillope,
    Author = {L. {Guillop\'e}},
    Title = {{Fonctions z\^eta de Selberg et surfaces de g\'eom\'etrie finie}},
    BookTitle = {{Zeta functions in geometry}},
    ISBN = {4-314-10078-8/hbk},
    Pages = {33--70},
    Year = {1992},
    Publisher = {Tokyo: Kinokuniya Company Ltd.},
    _Language = {French},
    _MSC2010 = {58J50 11F72},
    _Zbl = {0794.58044}
}

@article {Mayer_thermo,
    AUTHOR = {Mayer, D.},
     TITLE = {On the thermodynamic formalism for the {G}auss map},
   JOURNAL = {Comm. Math. Phys.},
  FJOURNAL = {Communications in Mathematical Physics},
    VOLUME = {130},
      YEAR = {1990},
    NUMBER = {2},
     PAGES = {311--333},
      ISSN = {0010-3616},
     CODEN = {CMPHAY},
   _MRCLASS = {58F19 (47A35 47B10 47B38 58F20)},
  _MRNUMBER = {MR1059321 (91g:58216)},
_MRREVIEWER = {Viviane Baladi},
       _URL = {http://projecteuclid.org/getRecord?id=euclid.cmp/1104200514},
}

@article {Mayer_thermoPSL,
     AUTHOR = {Mayer, D.},
     TITLE = {The thermodynamic formalism approach to {S}elberg's zeta
              function for {${\rm PSL}(2,{\bf Z})$}},
   JOURNAL = {Bull. Amer. Math. Soc. (N.S.)},
  FJOURNAL = {American Mathematical Society. Bulletin. New Series},
    VOLUME = {25},
      YEAR = {1991},
    NUMBER = {1},
     PAGES = {55--60},
      ISSN = {0273-0979},
     CODEN = {BAMOAD},
   MRCLASS = {58F20 (11F72 11M26 58G25)},
  _MRNUMBER = {MR1080004 (91j:58130)},
_MRREVIEWER = {Viviane Baladi},
       _DOI = {10.1090/S0273-0979-1991-16023-4},
       _URL = {http://dx.doi.org/10.1090/S0273-0979-1991-16023-4},
}

@article{Moeller_Pohl,
	AUTHOR = "M. M{\"o}ller and A. Pohl",
	TITLE = "Period functions for {H}ecke triangle groups, and the {S}elberg zeta function as a {F}redholm determinant",
	journal = "Ergodic Theory Dynam. Systems",
	Volume = "33",
	NUMBER = {1},
	PAGES = {247--283},
	YEAR = "2013"
}

@book{Olver74,
    AUTHOR = {Olver, F.},
     TITLE = {Asymptotics and special functions},
    SERIES = {AKP Classics},
      NOTE = {Reprint of the 1974 original.},
 PUBLISHER = {A K Peters, Ltd., Wellesley, MA},
      YEAR = {1997},
     PAGES = {xviii+572},
      ISBN = {1-56881-069-5},
   _MRCLASS = {41-02 (33Cxx 41A60 65D20)},
  _MRNUMBER = {1429619},
}

@incollection{Patterson_zeta,
    AUTHOR = {Patterson, S. J.},
     TITLE = {The {S}elberg zeta-function of a {K}leinian group},
 BOOKTITLE = {Number theory, trace formulas and discrete groups ({O}slo,
              1987)},
     PAGES = {409--441},
 PUBLISHER = {Academic Press, Boston, MA},
      YEAR = {1989},
}

@article {Phillips_perturb_twist,
    AUTHOR = {Phillips, R.},
     TITLE = {Perturbation theory for twisted automorphic functions},
   JOURNAL = {Geom. Funct. Anal.},
  FJOURNAL = {Geometric and Functional Analysis},
    VOLUME = {7},
      YEAR = {1997},
    NUMBER = {1},
     PAGES = {120--144},
      ISSN = {1016-443X},
   _MRCLASS = {11F72 (58G25)},
  _MRNUMBER = {1437475},
_MRREVIEWER = {Christopher M. Judge},
       _DOI = {10.1007/PL00001614},
       _URL = {https://doi.org/10.1007/PL00001614},
}

@article {Phillips_scattering_twist,
    AUTHOR = {Phillips, R.},
     TITLE = {Scattering theory for twisted automorphic functions},
   JOURNAL = {Trans. Amer. Math. Soc.},
  FJOURNAL = {Transactions of the American Mathematical Society},
    VOLUME = {350},
      YEAR = {1998},
    NUMBER = {7},
     PAGES = {2753--2778},
      ISSN = {0002-9947},
   _MRCLASS = {11F72 (35L05 35P25 47A40 58G25)},
  _MRNUMBER = {1466954},
_MRREVIEWER = {Jeffrey Stopple},
       _DOI = {10.1090/S0002-9947-98-02164-3},
       _URL = {https://doi.org/10.1090/S0002-9947-98-02164-3},
}

@Article{ Pohl_Symdyn2d,
	author = "A. Pohl",
	title = "Symbolic dynamics for the geodesic flow on two-dimensional hyperbolic good orbifolds",
	journal = "Discrete Contin. Dyn. Syst., Ser. A",
	year = "2014",
	volume = "34",
	number = "5",
	pages = "2173--2241",
}

@Article{Pohl_hecke_infinite,
 Author = {Pohl, A.},
 Title = {A thermodynamic formalism approach to the {Selberg} zeta function for {Hecke} triangle surfaces of infinite area},
 FJournal = {Communications in Mathematical Physics},
 Journal = {Commun. Math. Phys.},
 ISSN = {0010-3616},
 Volume = {337},
 Number = {1},
 Pages = {103--126},
 Year = {2015},
 _Language = {English},
 _DOI = {10.1007/s00220-015-2304-1},
 _Keywords = {37C30,37B10,11M36,53D25},
 _zbMATH = {6424646},
 _Zbl = {1348.37042}
}

@article{Pohl_Wabnitz,
 author = {Pohl, Anke and Wabnitz, Paul},
 title = {Selberg zeta functions, cuspidal accelerations, and existence of strict transfer operator approaches},
 fseries = {Memoirs of the American Mathematical Society},
 series = {Mem. Am. Math. Soc.},
 issn = {0065-9266},
 volume = {1616},
 _isbn = {978-1-4704-7867-4; 978-1-4704-8611-2},
 year = {2026},
 publisher = {Providence, RI: American Mathematical Society (AMS)},
 _language = {English},
 _doi = {10.1090/memo/1616},
 _keywords = {11-02,11M36,37C30,37D35,37D40,58J51},
 _zbMATH = {8175029}
}

@article {Scott83,
    AUTHOR = {Scott, P.},
     TITLE = {The geometries of {$3$}-manifolds},
   JOURNAL = {Bull. London Math. Soc.},
  FJOURNAL = {The Bulletin of the London Mathematical Society},
    VOLUME = {15},
      YEAR = {1983},
    NUMBER = {5},
     PAGES = {401--487},
      ISSN = {0024-6093},
}

@Article{Selberg,
    Author = {A. {Selberg}},
    Title = {{Harmonic analysis and discontinuous groups in weakly symmetric Riemannian spaces with applications to Dirichlet series}},
    FJournal = {{The Journal of the Indian Mathematical Society. New Series}},
    Journal = {{J. Indian Math. Soc., New Ser.}},
    ISSN = {0019-5839},
    Volume = {20},
    Pages = {47--87},
    Year = {1956},
    _Publisher = {Indian Mathematical Society c/o Department of Mathematics, University of Pune, Pune, Maharashtra, India},
    _Language = {English},
    _Zbl = {0072.08201}
}

@incollection {Selberg_lemma,
    AUTHOR = {Selberg, Atle},
     TITLE = {On discontinuous groups in higher-dimensional symmetric
              spaces},
 BOOKTITLE = {Contributions to function theory (internat. {C}olloq.
              {F}unction {T}heory, {B}ombay, 1960)},
     PAGES = {147--164},
 PUBLISHER = {Tata Institute of Fundamental Research, Bombay},
      YEAR = {1960},
   MRCLASS = {20.65 (22.70)},
  MRNUMBER = {0130324},
MRREVIEWER = {E. Grosswald},
}

@Book{Stein_Shakarchi,
 Author = {Stein, E.  and Shakarchi, R.},
 Title = {Complex analysis},
 FSeries = {Princeton Lectures in Analysis},
 Series = {Princeton Lect. Anal.},
 Volume = {2},
 ISBN = {0-691-11385-8},
 Year = {2003},
 Publisher = {Princeton, NJ: Princeton University Press},
 Language = {English},
 zbMATH = {1889799},
 Zbl = {1020.30001}
}

@article {Venkov,
    AUTHOR = {Venkov, A.},
     TITLE = {On {D}irichlet series that are associated with the defining
              equations and continued fractions in the theory of automorphic
              functions},
      NOTE = {Analytic number theory, mathematical analysis and their
              applications},
   JOURNAL = {Trudy Mat. Inst. Steklov.},
  FJOURNAL = {Akademiya Nauk Soyuza Sovetskikh Sotsialisticheskikh
              Respublik. Trudy Matematicheskogo Instituta imeni V. A.
              Steklova},
    VOLUME = {158},
      YEAR = {1981},
     PAGES = {31--44, 228},
      ISSN = {0371-9685},
   _MRCLASS = {10D24 (10D15)},
  _MRNUMBER = {662832},
_MRREVIEWER = {J. Szmidt},
}

@article {Venkov_book,
    AUTHOR = {Venkov, A.},
     TITLE = {Spectral theory of automorphic functions},
      NOTE = {A translation of Trudy Mat. Inst. Steklov.~153 (1981)},
   JOURNAL = {Proc. Steklov Inst. Math.},
  FJOURNAL = {Proceedings of the Steklov Institute of Mathematics},
      YEAR = {1982},
    NUMBER = {4 (153)},
     PAGES = {ix+163 pp.},
}

@misc{Wabnitz,
	AUTHOR = {P. Wabnitz},
	TITLE = {Strict transfer operator approaches for non-compact hyperbolic orbisurfaces},
	NOTE = {arXiv:2209.06601},
    YEAR = {2022},
}

@article {Zworski89,
    AUTHOR = {Zworski, M.},
     TITLE = {Sharp polynomial bounds on the number of scattering poles},
   JOURNAL = {Duke Math. J.},
  FJOURNAL = {Duke Mathematical Journal},
    VOLUME = {59},
      YEAR = {1989},
    NUMBER = {2},
     PAGES = {311--323},
      ISSN = {0012-7094},
   _MRCLASS = {35P25 (35P20 47A40 47F05 81F15)},
  _MRNUMBER = {1016891},
_MRREVIEWER = {J. A. Goldstein},
       _DOI = {10.1215/S0012-7094-89-05913-9},
       _URL = {https://doi.org/10.1215/S0012-7094-89-05913-9},
}

\setlength{\parindent}{0pt}

\end{document}